\documentclass[12pt]{amsart}
\usepackage[english]{babel}
\usepackage[utf8]{inputenc}
\usepackage[T1]{fontenc}
\usepackage{amsmath,changes,amsthm,amsfonts,hyperref,epsfig,amssymb}
\usepackage{mathrsfs,graphicx,tikz-cd}

\numberwithin{equation}{section}

\def\comment#1{}

\def\SAW{\text{SAW}}

\def\comment#1{}

\def\ignore#1{}

\newtheorem{theorem}{Theorem}
\newtheorem{proposition}[theorem]{Proposition}
\newtheorem{lemma}[theorem]{Lemma}
\newtheorem{remark}[theorem]{Remark}

\newtheorem{definition}[theorem]{Definition}
\newtheorem{cor}[theorem]{Corollary}
\newtheorem{notation}[theorem]{Notation}
\newtheorem{fact}[theorem]{Fact}
\newtheorem{conj}[theorem]{Conjecture}

\renewcommand{\epsilon}{\varepsilon}

\title[Trees of self-avoiding walks]{Trees of self-avoiding walks
 }
\date{}
\author[V.~Beffara]{Vincent Beffara}
\address{Vincent Beffara\\ Univ. Grenoble Alpes, CNRS, Institut Fourier, F-38000 Grenoble, France} \email{vincent.beffara@univ-grenoble-alpes.fr}

 \author[C.-B.~Huynh]{Cong Bang Huynh}
\address{Cong Bang Huynh\\ Univ. Grenoble Alpes, CNRS, Institut Fourier, F-38000 Grenoble, France}
\email{cong-bang.huynh@univ-grenoble-alpes.fr}


\begin{document}

\begin{abstract}
  We consider  the biased random walk  on a tree constructed  from the
  set  of finite  self-avoiding  walks on  a lattice,  and  use it  to
  construct probability measures on  infinite self-avoiding walks. The
  limit measure (if it exists) obtained when the bias converges to its
  critical value is conjectured to coincide with the weak limit of the
  uniform  SAW\@.  Along  the  way,  we obtain  a  criterion  for  the
  continuity of  the escape probability of  a biased random walk  on a
  tree as  a function  of the  bias, and show  that the  collection of
  escape  probability functions  for  spherically  symmetric trees  of
  bounded degree is stable under uniform convergence.
\end{abstract}

\maketitle
\thispagestyle{empty}


\vfill

\begin{center}
  \includegraphics[width=\hsize,angle=180]{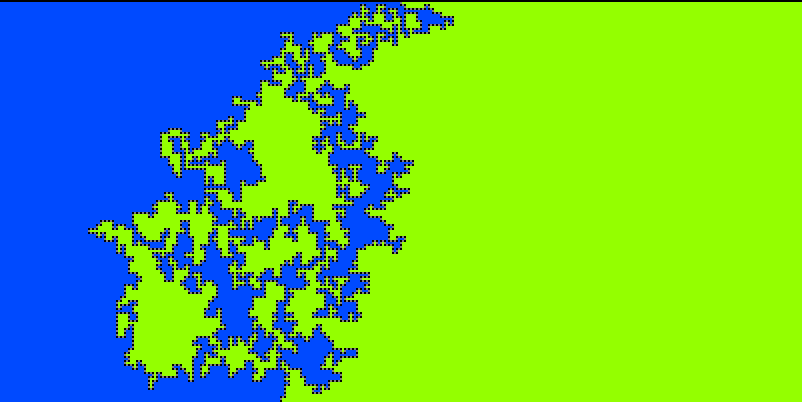}

  \bigskip

  \emph{\footnotesize  A   realization  of  the  limit   walk  in  the
    upper-half plane, with bias $\lambda=1$.}
\end{center}

\vfill

\newpage

\section{Introduction}

An  $n$-step self-avoiding  walk  (SAW) (or  a  self-avoiding walk  of
length $n$)  in a  regular lattice  $\mathbb L$  (such as  the integer
lattice  $\mathbb{Z}^2$,  triangular lattice  $\mathbb{T}$,  hexagonal
lattice,      etc)      is      a      nearest      neighbor      path
$\gamma=( \gamma_0, \gamma_1, \ldots, \gamma_n)$ that visits no vertex
more than once. Self-avoiding walks were first introduced as a lattice
model for polymer chains  (see~\cite{flory1953principles}); while they
are  very easy  to define,  they  are extremely  difficult to  analyze
rigorously and  there are still  many basic open questions  about them
(see~\cite{madras2013self}, Chapter 1).

Let $c_n$ be the number of SAWs  of length $n$ starting at the origin.
The \emph{connective constant} of $\mathbb L$, which we will denote by
$\mu$, is defined by
$$c_n =  \mu^{n+o(n)} \quad \text{when} \quad n
  \rightarrow \infty.$$ The existence of the connective constant is easy
to     establish     from    the     sub-multiplicativity     relation
$c_{n+m}  \leq  c_n  c_m$,  from   which  one  can  also  deduce  that
$c_n\geq \mu^n$ for all $n$; the existence of $\mu$ was first observed
by      Hammersley      and     Morton~\cite{hammersley1954poor}.
Nienhuis~\cite{nienhuis1982exact}  gave  a  prediction  that  for  all
regular   planar    lattices,   $c_n=\mu^n    n^{\alpha+o(1)}$   where
$\alpha=\frac{11}{32}$, and this prediction is known to hold under the
assumption of the existence of  a conformally invariant scaling limit,
see \emph{e.g.}~\cite{werner:saw}.

We are interested in defining a natural probability measure on the set
$SAW_{\infty}$  of \emph{infinite}  self-avoiding walks  (\emph{i.e.},
nearest-neighbors  paths ${(\gamma_k)}_{k  \geq 0}$  visiting no  vertex
more    than    once,     see    the    sections~\ref{kesten'smeasure}
and~\ref{newmeasure}). Such  a measure  was constructed before  in the
half-plane case  as the weak limit  of the uniform measures  on finite
self-avoiding     walks,    relying     on    results     by    Kesten
(see~\cite{madras2013self,KT2}),  and  it  is  part  of  our  goal  to
investigate whether that measure and our construction are related.

\subsection{The model}

In  this paper,  we  consider a  one-parameter  family of  probability measures
on         $SAW_{\infty}$,         denoted         by
${(\mathbb{P}_{\lambda})}_{\lambda>\lambda_c}$, defined  informally as follows
(see  Section~\ref{randomwalkontrees}    for     a    formal definition). Let
$\mathcal T_{\mathbb{Z}^2}$ be the tree  whose vertices are the finite
self-avoiding walks  in the  plane starting at  the origin $o := (0,0)$,  where
two such vertices  are adjacent when one  walk is a one-step  extension of the
other.  We will  call this tree  the \emph{self-avoiding  tree} on $\mathbb
Z^2$. Denoting  by  $\mathbb H$  the  upper-half  plane  in $\mathbb Z^2$  and
by $\mathbb Q$  the first quadrant, defined as \[\mathbb H := \{(x,y) \in
\mathbb Z^2 : y \geq 0 \} \quad \text{and} \quad \mathbb Q := \{(x,y)\in \mathbb
Z^2 : x \geq 0 \text{~and~} y \geq 0\},\] one  can construct the
self-avoiding     trees $\mathcal     T_{\mathbb{H}}$    and $\mathcal
T_{\mathbb{Q}}$ accordingly, and all the constructions below can be extended to
these cases in  a natural fashion which we will not make explicit in this
introduction.

Then,  consider the  continuous-time biased  random walk  of parameter
$\lambda>0$ on $\mathcal T_{\mathbb Z^2}$, which from a given location
jumps towards the root with rate  $1$ and towards each of its children
vertices with  rate $\lambda$. If $\lambda$  is such that the  walk is
transient,    its   path    determines   an    infinite   branch    in
$\mathcal  T_{\mathbb Z^2}$  which can  be seen  as a  random infinite
self-avoiding   walk  $\omega_\lambda^\infty$;   we  will   denote  by
$\mathbb P_\lambda^{\mathbb Z^2}$  the law of $\omega_\lambda^\infty$,
omitting the mention of $\mathbb Z^2$ in the notation when it is clear
from the  context, and  call it the  \emph{limit walk}  with parameter
$\lambda$.

The idea  of seeing the  self-avoiding walk  as a dynamical  object is
very  natural, and  not new;  it  seems that  the biased  walk on  the
``self-avoiding    tree''   was    first   considered,    mostly   for
$\lambda<\lambda_c$, by  Berretti and Sokal  (\cite{Berretti1985}, see
also~\cite{Sokal1996,Randall1994}) as a Monte-Carlo method to estimate
connective  constants  and   sample  finite-size  self-avoiding  paths
uniformly. The model was discussed informally by one of the authors of
the present paper (VB) with S.  Sidoravicius and W. Werner a number of
years ago, as  a failed attempt to understand  conformal invariance of
the  SAW model  in the  scaling limit,  and in  particular a  proof of
Theorem~\ref{proposition1.1}  was  obtained  at that  time  but  never
written down; one of our informal goals here is to revive this line of
thought: even  though the question  of SAW  proper still seems  out of
reach,     the     link     with    critical     percolation     (cf.\
Section~\ref{sec:percolation})  could  be  a promising  direction  for
further research.

\subsection{Main results}

It   is    well-known   that    there   exists   a    critical   value
$\lambda_c  =  \lambda_c(\mathcal  T_{\mathbb   Z^2})$  such  that  if
$\lambda>\lambda_c$  the  biased  random  walk  is  transient  and  if
$\lambda<\lambda_c$         it        is         recurrent        (see
Lyons~\cite{lyons1990random}). In  the general  case of  biased random
walk on a tree, the recurrence or transience of the random walk at the
critical point  depends in subtle ways  on the structure of  the tree.
The value  of $\lambda_c$ on  the other  hand is easier  to determine:
indeed, Lyons~\cite{lyons1990random} proved that it coincides with the
reciprocal  of the  branching number  of the  tree (for  background on
branching      numbers     and      trees     in      general,     see
\emph{e.g.}~\cite{LP:book}).  The  following   proposition  gives  the
critical value for self-avoiding trees.

\begin{theorem}\label{proposition1.1}
  Let
  $\mathcal      T_{\mathbb{Z}^2},\mathcal     T_{\mathbb{H}},\mathcal
    T_\mathbb{Q}$  be   the  self-avoiding   trees  defined   as  above,
  respectively in the plan, half-plane and first quadrant. Then,
  $$\lambda_c(\mathcal T_{\mathbb{Z}^2}) =
    \lambda_c(\mathcal     T_{\mathbb{H}})      =     \lambda_c(\mathcal
    T_\mathbb{Q}) = \frac1\mu,$$ where  $\mu$ is the connective constant
  of lattice $\mathbb{Z}^2$ as defined above.
\end{theorem}

This is a direct consequence of Proposition~\ref{prop10} below. Notice
that it is clear from the definition  that $\mu$ is the growth rate of
$\mathcal T_{\mathbb Z^2}$;  there are rather large  classes of trees,
including  $\mathcal T_{\mathbb  Z^2}$,  for which  the branching  and
growth coincide (for  instance, this holds for  sub- or super-periodic
trees, cf.\ below, or  for typical supercritical Galton-Watson trees),
but   none   of    the   classical   results   seem    to   apply   to
$\mathcal T_{\mathbb H}$ or $\mathcal T_{\mathbb Q}$.

\bigskip

The geometry of  the limit walk is  our main object of  interest. As a
first    property   of    it,   we    obtain   the    following   (see
Section~\ref{sec:plambda}):

\begin{theorem}\label{thm:propertieoflimitwalk}
  For    all     $\lambda    >\lambda_c$,    under     the    measures
  $\mathbb{P}_{\lambda}^{\mathbb               Z^2}$               and
  ${\mathbb  P}_\lambda^{\mathbb H}$,  the  limit  walk almost  surely
  visits the line $\mathbb{Z} \times  \left \{ 0 \right \}$ infinitely
  many times.
\end{theorem}

A useful tool in our proofs is the \emph{effective conductance} of the
biased random walk on a tree  $\mathcal T$, defined as the probability
of never  returning to  the root  $o$ of $\mathcal  T$ and  denoted by
$\mathcal C(\lambda,\mathcal T)$ --- see~\cite{LP:book}. Along the way, we will
be interested in several properties of  it as a function of $\lambda$.
Most important for us will be the question of continuity: in a general
tree,  the  effective  conductance  is not  necessarily  a  continuous
function of $\lambda$.  We will derive criteria  for continuity, which
are forms of  \emph{uniform transience} of the random  walk, and apply
them to prove that the effective conductance of self-avoiding trees is
a continuous function (see Section~\ref{section4.5}):

\begin{theorem}\label{thm:continuousofselfavoidingtree}
  The                      effective                      conductances
  $\mathcal            C(\lambda,\mathcal            T_{\mathbb{Q}})$,
  $\mathcal       C(\lambda,\mathcal        T_{\mathbb{H}})$       and
  $\mathcal   C(\lambda,\mathcal  T_{\mathbb{Z}^2})$   are  continuous
  functions of $\lambda$ on the interval $(\lambda_c,+\infty)$.
\end{theorem}

A related question is that of the convergence of effective conductance
along a  sequence of trees.  More precisely, let  ${(\mathcal C_n)}_n$
denote the effective  conductances for a sequence  $(\mathcal T_n)$ of
infinite  trees,  again  seen  as  functions  of  the  bias  parameter
$\lambda$, and  assume that  ${(\mathcal C_n)}_n$  converges uniformly
towards  a function  $\mathcal C$  that  is not  identically $0$.  The
question is:  is $\mathcal C$  the effective conductance of  a certain
tree? We  study this  question on the  class of  spherically symmetric
trees (a tree $\mathcal T$ is  said to be spherically symmetric if for
every vertex $\nu$, $\deg \nu$ depends only on $\left | \nu \right |$,
where $\left  | \nu \right  |$ denote its  distance from the  root and
$\deg \nu$  is its number  of neighbors). If $\mathbb{S}$  denotes the
set of spherically symmetric trees  and $m \in \mathbb{N}^*$ is fixed,
define
$$A_m:= \{ \mathcal T  \in  \mathbb{S}; \forall  \nu \in  \mathcal T,  
  \deg \nu \leq  m \}   \text{~and~}$$
$$\mathbb{F}_m:= \{ f\in  C^0([0,1]):\exists \mathcal  T\in A_m,
  \forall \lambda>0, \mathcal  C(\lambda,\mathcal T)=f(\lambda)
  \}.$$ Then (see Section~\ref{section3.4}):

\begin{theorem}\label{theorem1.4}
  Let ${(f_n)}_n$ be a sequence of functions in $\mathbb{F}_m$. Assume
  that  $f_n$   converges  uniformly  towards   $f  \neq  0   $.  Then
  $f \in \mathbb{F}_m$.
\end{theorem}

\subsection{Open questions}

One natural probability  measure on the set  of infinite self-avoiding
walks   is   the  limit   of   $\mathbb   P^{\mathbb  H}_\lambda$   as
$\lambda \to \lambda_c$, assuming that  this limit exists. We were not
able  to  show  convergence,  but obtained  partial  results  in  this
direction by restricting  the set of allowed paths.  Our conjecture is
that the limit exists and has to do with Kesten's measure, \emph{i.e.}
the  weak  limit   of  uniform  finite  self-avoiding   walks  in  the
half-plane, in a  way similar to the fact that  the two definitions of
the incipient  infinite cluster  for percolation (seen  as a  limit as
$p\to  p_c$  or  as  a  limit  of  conditioned  critical  percolation)
coincide, see~\cite{IIC}.

This  is  motivated  by  a  few observations.  First,  the  model  for
$\lambda  <  \lambda_c$ gives  rise  to  a  recurrent random  walk  on
$\mathcal T_{\mathbb H}$ for which the invariant measure $\mu_\lambda$
is rather  explicit (by reversibility, the  mass of a vertex  $\nu$ is
proportional to  $\lambda^{|\nu|}$), in particular it  depends only on
the  distance to  the root,  and  on the  other  hand it  tends to  be
concentrated      on     longer      and      longer     walks      as
$\lambda \uparrow \lambda_c$. This means that the initial segment of a
walk distributed as the stationary measure  can be seen as the initial
segment of a uniform self-avoiding  walk with random total length, and
we get convergence to Kesten's measure as soon as we can show that for
all       $\nu$,      $\mu_\lambda(\{\nu\})       \to      0$       as
$\lambda \uparrow \lambda_c$.  On the other hand, the  behavior of the
biased walk in a fixed neighborhood  of the origin changes very little
when  $\lambda$ is  close to  $\lambda_c$, so  for $\lambda$  slightly
larger than $\lambda_c$  it seems reasonable to predict  that the walk
will spend  a long time close  to the origin, following  an occupation
measure  close to  $\mu_{\lambda_c^-}$, before  escaping to  infinity.
Unfortunately we were unable to formalize this intuition.

Another observation is  that convergence of the law of  the limit walk
holds within  the class  of paths for  which the  bridge decomposition
involves only bridges of height less than some fixed bound $m>0$. More
precisely:    for   fixed    $m$,    the    critical   parameter    is
$\lambda_{c,m}      \geq      \lambda_c$,      and      the      limit
$\lambda \downarrow \lambda_{c,m}$ followed by $m \to \infty$ leads to
Kesten's measure, while  the limit $m \to \infty$  for fixed $\lambda$
coincides  with  the  limit  walk on  $\mathcal  T_{\mathbb  H}$  with
parameter $\lambda$ --- see  Theorem~\ref{theorem6.1} for more detail.
Exchanging the  limits would lead  to the claim. Unfortunately,  it is
not true that this can be done  in the general setting of biased walks
on   trees,  due   to  phenomena   similar  to   those  described   in
Section~\ref{section3}, so it seems that a deeper understanding of the
structure of $\mathcal T_{\mathbb H}$  would be necessary to conclude.

\subsection{Organization of the paper}

The  paper is  structured  as follows.  In Section~\ref{section2},  we
review some  basic definitions on  graphs, trees, branching  number and
growth rate of a tree, as well as a few classical results about random
walks on trees. Section~\ref{section3}  gathers some relevant examples
and counter-examples exhibiting some similarities to the self-avoiding
trees  while  being  treatable   explicitly.  The  criterion  for  the
continuity    of   the    effective    conductance    is   given    in
Section~\ref{sectioncontinuity}. Then  Section~\ref{section4} provides
some background  on self-avoiding  walks and the  self-avoiding trees,
and   some   properties  of   the   limit   walks  are   obtained   in
Section~\ref{newmeasure}.  Finally, we  state  a  few conjectures  and
conditional         results         in         Section~\ref{section6}.

\subsection*{Acknowledgments}

The authors are grateful to an anonymous referee for extremely detailed and
extensive comments on a previous version of this paper, which were a significant
help in making some key arguments clearer.

\section{Notation and basic definitions}\label{section2}

\subsection{Graphs and trees}

In this section, we review some basic definitions; we refer the reader
to  the   book~\cite{LP:book}  for  a  more   developed  treatment.  A
\emph{graph}  is a  pair  $\mathcal G=(V,E)$  where $V$  is  a set  of
\emph{vertices} and $E$ is a symmetric  subset of $V \times V$ (i.e if
$(\nu,\mu)\in E$  then $(\mu,\nu)\in E$), called  the \emph{edge set},
containing no  element of the  form $(\nu,\nu)$. If  $(\nu,\mu)\in E$,
then we call  $\nu$ and $\mu$ \emph{adjacent}  or \emph{neighbors} and
we  write  $\nu\sim \mu$.  For  any  vertex  $\nu  \in V$,  denote  by
$\deg \nu$  its number  of neighbors.  A \emph{path} in  a graph  is a
sequence of  vertices, any  two consecutive of  which are  adjacent. A
\emph{self-avoiding path}  is a path  which does not pass  through any
vertex more than once. For any $(\nu, \mu)\in V\times V$, the distance
between $\nu$ and $\mu$ is the minimum number of edges among all paths
joining  $\nu$   and  $\mu$,  denoted   $d(\nu,  \mu)$.  A   graph  is
\emph{connected}  if, for  each pair  $(\nu,\mu)\in V\times  V$, there
exist a path starting at $\nu$  and ending at $\mu$. A connected graph
with no cycles is called a \emph{tree}. A \emph{morphism} from a graph
$\mathcal G_1$  to a  graph $\mathcal  G_2$ is  a mapping  $\phi$ from
$V(\mathcal G_1)$ to $V(\mathcal G_2)$ such that the image of any edge
of $\mathcal G_1$ is an edge of $\mathcal G_2$ We will always consider
trees to  be \emph{rooted} by the  choice of a vertex  $o$, called the
\emph{root}. \bigskip

Let $\mathcal  T=(V,E)$ be  an infinite,  locally finite,  rooted tree
with set of vertices $V$ and set of  edges $E$. Let $o$ be the root of
$\mathcal T$.  For any vertex  $\nu \in V\setminus\{{o}\}$,  denote by
$  {\nu}^{-1}$ its  \emph{parent}  (we also  say that  $  {\nu}$ is  a
\emph{child}  of $  {\nu}^{-1}$), \emph{i.e.}  the neighbor  of $\nu$
with shortest distance  from $o$. For any $\nu \in  V$, let $|\nu|$ be
the number of edges in  the unique self-avoiding path connecting $\nu$
to~$o$ and call $|\nu|$ the  \emph{generation} of $\nu$. In particular,
we  have  $|o|=0$. Denote by $\mathcal{T}_n$ the set of all vertices of $\mathcal T$ that are at generation $n$.

If  a  vertex  has no  child,  it  is  called  a
\emph{leaf}. For  any edge  $e \in  E$ denote by  $e^-$ and  $e^+$ its
endpoints with $|e^+|=|e^-|+1$,  and define the generation  of an edge
as $|e|=|e^+|$. We  define an order on $V(\mathcal T)$  as follows: if
$\nu,\mu \in V(\mathcal T)$, we say  that $\nu \leq \mu$ if the simple
path   joining  $o$   to  $\mu$   passes  through   $\nu$.  For   each
$\nu \in V(\mathcal T)$, we define the \emph{subtree} of $\mathcal T$
rooted   at    $\nu$,   denoted   by   $\mathcal    T^\nu$,   where
$V(\mathcal T^\nu):=\{ \mu\in V(\mathcal T): \nu \leq \mu 
  \}$                                                                and
$E(\mathcal  T^\nu)={E(\mathcal  T)}\!\!\mid_{V(\mathcal  T^\nu)\times
    V(\mathcal T^\nu)}$.

An infinite  simple path starting at  $o$ is called a  \emph{ray}. The
set   of  all   rays,  denoted   by  $\partial   \mathcal T$,  is   called  the
\emph{boundary} of  $\mathcal T$. The set $\mathcal T  \cup \partial \mathcal T$ can  be equipped
with a metric that makes it a compact space, see~\cite{LP:book}.
\bigskip

The remaining part of this paper, we consider only infinite, locally finite and rooted trees with the root $o$.
\subsection{Branching and growth}

\begin{definition}
  Let $ \mathcal T$ be an infinite, locally finite and rooted tree. A E-cutset
  (resp.\ V-cutset) in $\mathcal T$  is a set $\pi$ of edges (resp.\ vertices)
  such that, for any infinite self-avoiding path ${(\nu_i)}_{i\ge0}$ started at
  the root,  there exists a $i\ge0$ such that $[\nu_{i-1},\nu_i]\in \pi$ (resp.\
  $\nu_i\in \pi$). In other words, a E-cutset (resp. V-cutset) is a set of edges
  (resp.\ vertices) separating the root from infinity. We use $\Pi$ to denote the
  set of E-cutsets.
\end{definition}

\begin{definition}
  Let $\mathcal T$ be an infinite, locally finite and rooted tree.
  \begin{itemize}
    \item
          The \emph{branching number} of $\mathcal T$ is defined by:
          $$br(\mathcal T)  =  \sup \left\{  \lambda  \geq  1 :  \inf_{\pi\in \Pi}  \sum_{e\in
              \pi}\lambda^{-\left  |  e \right  |}>0  \right  \}$$

    \item  We  define also
          \[ \overline{gr} (\mathcal T)  = \limsup\left | \mathcal T_{n}  \right |^{1/n} \quad
            \text{and}  \quad \underline{gr}  (\mathcal T)  = \liminf  \left |  \mathcal T_{n}
            \right|^{1/n}.\]             In             the             case
          $\overline{gr}(\mathcal T)=\underline{gr}(\mathcal T)$,  the  \emph{growth rate}  of
          $\mathcal T$ is defined by their common value and denoted by $gr(\mathcal T)$.
  \end{itemize}
\end{definition}
\begin{remark}\label{rem:comparebranchingnumber}
  It follows immediately from the definition of branching number that if $\mathcal{T}'$ is a subtree of $\mathcal{T}$, then $br(\mathcal{T}')\leq br(\mathcal{T})$.
\end{remark}

\begin{proposition}[\cite{LP:book}]\label{prop:br-gr}
  Let $\mathcal T$ be a tree, then $br(\mathcal T) \leq \underline{gr}(\mathcal T)$.
\end{proposition}

In  general, the  inequality  in  Proposition~\ref{prop:br-gr} may  be
strict:  The  \emph{1--3  tree}  (see~\cite{LP:book}, page  4)  is  an
example for which  the branching number is $1$ and  the growth rate is
$2$. There  are classes  of trees however  where branching  and growth
match.

\begin{definition}
  The tree $\mathcal T$ is said to  be \emph{spherically symmetric} if $\deg \nu$
  depends only on $\left | \nu \right |$.
\end{definition}

\begin{theorem}[\cite{LP:book} page 83]\label{symetrie}
  For      every       spherically      symmetric       tree      $\mathcal T$,
  $br(\mathcal T)=\underline{gr} (\mathcal T) $.
\end{theorem}

\begin{definition}
  Let $N\geq 0$: an infinite, locally finite and rooted tree $\mathcal T$ with the root $o$, is said to be
  \begin{itemize}
    \item \emph{$N$-sub-periodic} if  for every $\nu \in  V(\mathcal T)$, there exists
          an  injective  morphism  $f:\mathcal T^{\nu}\rightarrow  \mathcal T^{f(\nu)}$
          with $\left | f(\nu) \right |\leq N$.
    \item \emph{$N$-super-periodic} if for every $\nu \in V(\mathcal T)$, there exists
          an injective morphism $f:\mathcal T\rightarrow  \mathcal T^{f(o)}$ with $f(o)\in \mathcal T^\nu$ and
          $|f(o)|-|\nu|\leq N$.
  \end{itemize}
\end{definition}

\begin{theorem}[see~\cite{furstenberg1967,LP:book}]\label{sousperiodic}
  Let $\mathcal T$ be an infinite, locally finite and rooted tree that
  is    either    $N$-sub-periodic,   or    $N$-super-periodic    with
  $\overline{gr}(\mathcal  T)  <  \infty$.  Then the  growth  rate  of
  $\mathcal T$ \!\!exists and $gr(\mathcal T)=br(\mathcal T)$.
\end{theorem}

\subsection{Random walks on trees}\label{randomwalkontrees}

Let $\mathcal T$ be a tree, we  now define the discrete-time biased random walk
on $\mathcal T$. Working in discrete time will make some of the arguments below
a little  simpler, at the cost  of a slightly heavier  definition here
---   notice    though   that   the   definition    of   the   measure
$\mathbb P_\lambda$ and  the main results of the paper  are not at all
affected by this choice.

Let $\lambda > 0$: the biased walk $RW_\lambda$ with bias $\lambda$ on
$\mathcal T$ is  the discrete-time Markov chain  on the vertex set  of $\mathcal T$ with
transition  probabilities given,  at  a  vertex $x  \neq  o$ with  $k$
children, by \[ p_\lambda (x,y) := \begin{cases}
    \frac 1 {1 + k \lambda}      & \text{if $y$ is the father of $x$,} \\
    \frac \lambda {1 + k\lambda} & \text{if $y$ is a child of $x$,}    \\
    0                            & \text{otherwise}.
  \end{cases} \] If the root has $k>0$ children, then $p_\lambda(o,x)$
is $1/k$ if  $x$ is a child  of $o$ and $0$  otherwise. The degenerate
case  $\mathcal T=\{o\}$ where  the root  has no  child will  not occur  in our
context, so  we will silently  ignore it.  We also allow  ourselves to
consider the  cases $\lambda \in  \{ 0,  \infty \}$, with  the natural
convention that $RW_0$ remains stuck  at the root and that $RW_\infty$
always moves away  from the root, getting stuck whenever  it reaches a
leaf.

\begin{definition}
  Let $\mathcal G=(V,E)$ be a graph, and $c : E \to \mathbb R_+^\ast$ be labels
  on the edges, referred  to as \emph{conductances}. Equivalently, one
  can fix  \emph{resistances} by  letting $r(e)  := 1/c(e)$.  The pair
  $(G,c)$ is called  a \emph{network}. Given a subset $K$  of $V$, the
  restriction of $c$ to the edges  joining vertices in $K$ defines the
  \emph{induced  sub-network} $\mathcal G_{|K}$.  The \emph{random  walk} on
  the network  $(\mathcal G,c)$ is the  discrete-time Markov chain on  $V$ with
  transition probabilities proportional to the conductances.
\end{definition}

Given a network $(\mathcal T,c)$ on a tree, let $\pi(o)$ be the sum of
the conductances  of the  edges incident  to the  root, and  denote by
$T(o)$   the  first   return  time   to  the   origin  by   the  walk.
Following~\cite{LP:book} (page 25),  we can define the \emph{effective
  conductance} of the network by
\begin{equation}
  \label{equ:relationconductanceproba}
  \mathcal C_c(\mathcal T) := \pi(o) \widetilde {\mathcal C_c} (\mathcal
  T),
\end{equation}
where
$\widetilde {\mathcal C_c} (\mathcal T) := \mathbb P[T(o) = +\infty]$.
The  reciprocal   $\mathcal  R_c   (\mathcal  T)$  of   the  effective
conductance is called the \emph{effective resistance}.

We will be particularly interested in the case where the conductances are chosen exponentially in the distance to the root, more precisely if for every edge $e = (x,y)$ where $x$ is the parent of $y$ we let $c(e) = \lambda^{|x|}$, because in that case the random walk on the network is exactly the same process as the
random  walk $RW_\lambda$  defined  earlier. We will use the following notation many times in what follows:

\begin{notation}
  For every parameter $\lambda>0$, assigning to every edge $e = (x,y)$ conductance $\lambda^{|x|}$, we denote by $\mathcal C(\lambda, \mathcal T)$ (resp.\ $\mathcal R(\lambda, \mathcal T)$) the effective conductance (resp.\ resistance) of the associated network. Moreover, if $\nu$ is a child of the root $o$ of $\mathcal T$, we  write
  $\widetilde  {\mathcal   C}  (\lambda,  \mathcal{T},  \nu)$   for  the
  probability  of  the  event  that  the  random  walk  $RW_\lambda$  on
  $\mathcal T$, started  at the root (i.e $X_0=o$), visits $\nu$ at its first step (i.e $X_1=\nu$) and never returns to the root.
\end{notation}


\begin{theorem}[Rayleigh's monotonicity principle~\cite{LP:book}]\label{RL}
  Let $\mathcal T$  be an infinite tree with two  assignments, $c$ and
  $c'$, of conductances on $\mathcal T$ with $c \leq c'$ (everywhere).
  Then  the  effective  conductances  are ordered  in  the  same  way:
  $\mathcal C_{c}(\mathcal  T)\leq \mathcal C_{\widetilde{c}}(\mathcal
    T)$.
\end{theorem}

\begin{cor}\label{corRL}
  Let  $\mathcal T,\mathcal  T'$ be  two infinite  trees; we  say that
  $\mathcal  T  \subset \mathcal  T'$  if  there exists  an  injective
  morphism $f:\mathcal T\rightarrow \mathcal  T'$. If this holds, then
  for                        every                        $\lambda>0$,
  $\mathcal  C(\lambda,\mathcal  T')\leq  \mathcal  C(\lambda,\mathcal
    T)$.
\end{cor}

In the case of spherically  symmetric trees, the equivalent resistance
is explicit:

\begin{proposition}[see~\cite{LP:book}]\label{prop3}
  Let  $\mathcal   T$  be   spherically  symmetric  and   $(c(e))$  be
  conductances  that   are  themselves  constant  on   the  levels  of
  $\mathcal                          T$.                          Then
  $\mathcal   R_c(\mathcal    T)=\sum_{n\geq1}\frac{1}{c_n   |\mathcal
      T_n|}$, where  $c_n$ is  the conductance of  the edges  going from
  level $n-1$ to level $n$.
\end{proposition}
The following  corollaries are the consequences of Proposition~\ref{prop3}:

\begin{cor}\label{cor:continuousofSS}
  Let  $\mathcal T$ be  a spherically  symmetric tree. The effective conductance $\mathcal C(\lambda,\mathcal T)$ is a continuous function on $(\lambda_c,+\infty)$.
\end{cor}

\begin{cor}\label{criterert}
  Let  $\mathcal T$ be  a spherically  symmetric tree.  Then $RW_{\lambda}$  is
  transient            if            and            only            if
  $\sum_{n}\frac{1}{\lambda^{n}\left | \mathcal T_n \right |}< \infty$.
\end{cor}



\begin{theorem}[Nash-Williams criterion, see~\cite{Nash1959}]\label{NW}
  If $(\pi_n, n\geq 0)$ is  a sequence of pairwise disjoint finite  E-cutsets in a
  locally finite network $\mathcal G$, then
  $$\mathcal R_c(\mathcal T)\geq \sum_{n}{\left ( \sum_{e \in
      \pi_n}     c(e)\right     )}^{-1}.$$    In     particular,     if
  $\underset{n}{\sum}{\left(    \sum_{e   \in    \Pi_n}   c(e)\right)}^{-1}=+\infty$, then the  random walk associated to  this family of
  conductances $(c(e), e\in E(\mathcal T))$ is recurrent.
\end{theorem}

We end  this subsection  by stating a  classical theorem  relating the
recurrence  or transience  of  $RW_\lambda$ to  the  branching of  the
underlying tree:

\begin{theorem}[see~\cite{lyons1990random}]\label{lyons90}
  Let $\mathcal T$ be an infinite, locally finite and rooted tree. If  $\lambda
  <  \frac{1}{br(\mathcal T)}$ then  $RW_{\lambda}$ is  recurrent, whereas  if
  $\lambda   >\frac{1}{br(\mathcal T)}$,  then  $RW_{\lambda}$  is transient.
  The critical  value  of  biased random  walk  on $\mathcal T$  is therefore
  $\lambda_c(\mathcal T):=\frac{1}{br(\mathcal T)}$.
\end{theorem}

\subsection{The law of the first \texorpdfstring{$k$}{k} steps of the limit walk}

Let $\mathcal T$  be a tree  and $(c(e))$ be conductances  on the edges  of
$\mathcal T$ such that the  associated random walk $(X_n)$ is  transient. For
every $k \geq 0$,  the walk visits $\mathcal T_k$ finitely many  times: we can
define an infinite path $\omega^\infty$  on $\mathcal T$ by letting
$\omega^\infty(k)$ be the last vertex of $\mathcal T_k$ visited by the walk.
Equivalently:
\begin{equation}
  \omega^{\infty}(k)=\nu \iff  \nu \in \mathcal T_k  \text{~and~} \exists
  n_0, \forall n>n_0: X_n \in \mathcal T^{\nu}.
\end{equation}
Let  $k \in  \mathbb{N}^{*}$  and $\nu_0=o,  \nu_1,\nu_2,\ldots,\nu_k$ be  $k$
elements of  $V(\mathcal T)$ such  that the path
$(\nu_0,\nu_1,\nu_2,\ldots,\nu_k)$ is a possible prefix of $\omega^\infty$: we
can then define
\begin{equation}
  \varphi_{c}(\nu_0,\nu_1,\nu_2,\ldots,\nu_k) :=
  \mathbb{P}(\omega^{\infty}(0)                  =
  \nu_{0},\omega^{\infty}(1) = \nu_{1},
  \ldots,
  \omega^{\infty}(k)=\nu_{k}).
\end{equation}
We will refer to this function as  the \emph{law of first $k$ steps of
  limit walk}.  In the case of  the biased walk $RW_\lambda$,  we will
denote the function  by $\varphi^{\lambda, k}$; this will  not lead to
ambiguities. We finish this section with the following lemma.

\comment{réécrit ce lemme}
\begin{lemma}\label{lemmac}
  The value of $\varphi_c(\nu_0,\ldots,\nu_k)$ depends continuously on any
  finite  collection   of  the  conductances  in   the  network.  More
  precisely, given  a finite set  $U = \{  e_1, \ldots, e_\ell  \}$ of
  edges   and    a   collection   $(c(e))$   of    conductances,   let
  $\tilde c(u_1, \ldots,  u_\ell)$ be the family  of conductances that
  coincides with $c$ outside $U$ and  takes value $u_i$ at $e_i$: then
  the                                                              map
  \[ \psi_{U,c} : (u_1, \ldots, u_\ell) \mapsto \varphi_{\tilde c(u_1,
      \ldots, u_\ell)} (\nu_0,\ldots,\nu_k) \]
  is continuous on ${(\mathbb R_+^\ast)}^\ell$.
\end{lemma}
\begin{proof}
  The proof is simple, therefore it is omitted.
\end{proof}

\section{A few examples}
\label{section3}

The  self-avoiding tree  in  the plane,  which we  alluded  to in  the
introduction  and will  formally  introduce in  the  next section,  is
sub-periodic but  quite inhomogeneous,  and the self-avoiding  tree in
the half-plane sits in none of  the classes of trees defined above. To
get an intuition of the kind of behavior we should expect or rule out,
we gather here a few examples of trees with some atypical features.

\subsection{Trees with discontinuous conductance}

Let  $0<\lambda_0 \leq  1$.  In the  first part  of  this section,  we
construct   two   trees   $\mathcal  T,\overline{\mathcal   T}$   with
$\lambda_c(\mathcal   T)=\lambda_c(\overline{\mathcal  T})=\lambda_0$,
such  that  the  effective conductances  $C(\lambda,\mathcal  T)$  and
$C(\lambda,\overline{\mathcal   T})$  of   the   biased  random   walk
$RW_{\lambda}$  on $\mathcal  T$ and  $\overline{\mathcal T}$  satisfy
$\mathcal      C(\lambda_c(\mathcal     T),\mathcal      T)=0$     but
$\mathcal     C(\lambda_c(\overline{\mathcal    T}),\overline{\mathcal
    T})>0$. In  the second part, we  construct a tree $\mathcal  T$ such
that $C(\lambda,T)$ is not continuous on $(\lambda_c,1)$.

\begin{proposition}
  \label{prop5}
  For    every    $x\geqslant    1$,    there    exist    two    trees
  $\mathcal T$ and $\overline{\mathcal T}$ such that:
  \begin{itemize}
    \item $br(\mathcal T)=br(\overline{\mathcal T})=x$;
    \item  $RW_{{1}/{x}}$  is recurrent  on  $\mathcal T$  and transient  on
          $\overline{\mathcal T}$.
  \end{itemize}
\end{proposition}

\begin{proof}
  We  will  construct  spherically  symmetric  trees  satisfying  both
  conditions. Denote by $\lfloor y  \rfloor$ the integer part of
  $y$. We construct  the sequence  ${(\ell_{i})}_{i\in  \mathbb{N}^*}$
  inductively as follows:
  \[   \ell_{1}   =   \left\lfloor  x    \right\rfloor,   \quad   \ell_{2}   =   \left\lfloor
    \frac{x^{2}}{\ell_{1}}    \right\rfloor,    \quad    \ell_{3}    =    \left\lfloor
    \frac{x^{3}}{\ell_{1}\ell_{2}}  \right\rfloor,  \quad \ldots,\quad  \ell_{n}  =
    \left\lfloor  \frac{x^{n}}{\prod_{i=1}^{n-1}\ell_i} \right\rfloor,  \quad\ldots\]
  and let $\mathcal T$  be the tree where vertices at  distance $i-1$ from $o$
  have $\ell_{i}$ children, so that the sizes of the levels of $\mathcal T$ are
  given by $\left | \mathcal T_n \right |=\prod_{i=1}^{n}\ell_i$. We construct the
  tree     $\overline{\mathcal T}$     from      the     degree     sequence
  ${(\ell'_{i})}_{i\in \mathbb{N}}$ by posing $\ell'_{i}=2\ell_{i}$ if $i$ can be
  written  under the  form  $i=k^{2}$,  and $\ell'_{i}=\ell_{i}$  otherwise.
  Notice that $|\overline {\mathcal T}_n| = 2^{[\sqrt n]} |\mathcal T_n|$.

  We first show that both trees  have branching number $x$. Since they
  are spherically symmetric,  it is enough to check  that their growth
  rate is  $x$; the case $x=1$  is trivial, so assume  $x>1$. From the
  definition,
  \[x^{n}-\prod_{i=1}^{n-1}\ell_i\leq     \prod_{i=1}^{n}\ell_i\leq    x^{n}
    \quad \text{hence} \quad x^n - x^{n-1} \leq |\mathcal T_n| \leq x^n\]
  so $gr(\mathcal T)=x$; the case of $\overline {\mathcal T}$ follows directly.

  The recurrence  or transience  of the critical  random walks  can be
  determined using Lemma~\ref{criterert}:
  \[ \sum \frac  1 {\lambda_c^n |\mathcal T_n|} \geq \sum  \frac 1 {\lambda_c^n
        x^n} = +\infty  \] so the critical walk on  $\mathcal T(x)$ is recurrent,
  while for $x>1$,
  \[ \quad \sum \frac 1  {\lambda_c^n |\overline {\mathcal T}_n|} \leq \sum \frac
    1 {\lambda_c^n (x^n-x^{n-1})  2^{\left \lfloor\sqrt n \right \rfloor}} = \frac  x {x-1} \sum
    \frac {1}  {2^{\left \lfloor\sqrt n\right \rfloor}}  < \infty  \] so  the critical  walk on
  $\overline  {\mathcal T}(x)$ is  transient. In  the case  $x=1$ one  gets $\sum
    2^{-\left\lfloor\sqrt n\right \rfloor}<\infty$ instead, and the conclusion is the same.
\end{proof}

\begin{proposition}
  \label{prop4}
  For every  $k \in \mathbb{N^{*}}$  and $\lambda_c \in  (0,1)$, there
  exists a  tree $\mathcal T$ with critical  drift $\lambda_c(\mathcal T)=\lambda_c$
  such  that the  ratio $\mathcal C(\lambda, \mathcal T)  / {(\lambda-\lambda_c)}^k$  remains
  bounded away from $0$ as $\lambda \to \lambda_c^+$.
\end{proposition}

\begin{proof}
  We construct  a spherically symmetric  tree $\mathcal T$ which  satisfies the
  conditions of this  proposition in a similar way  as before. Letting
  $x=1/\lambda_c>1$, define inductively:

  \[ \ell_{1}=\left  \lfloor   x    \right \rfloor,    \quad    \ell_{2}=\left \lfloor
    \frac{x^{2}}{2^k \ell_{1}}   \right  \rfloor,  \quad   \ldots,\quad
    \ell_{n}=\left  \lfloor \frac{x^{n}}{n^k\prod_{i=1}^{n-1}\ell_i}  \right \rfloor,
    \quad \ldots. \]

  Let $\mathcal T$  be the spherically symmetric tree with
  degree sequence $(\ell_i)$. It is easy  to check that $br(\mathcal T)=x$ like in
  the     previous     proposition;      in     a     similar     way,
  \[ {x^{n}} -  {n^k\prod_{i=1}^{n-1}\ell_i} \leq {n^k\prod_{i=1}^{n}\ell_i}
    \leq {x^{n}} \quad \text{hence} \quad  \frac {x^{n}} {n^k} - \frac
    {x^{n-1}} {{(n-1)}^k}  \leq |\mathcal T_n| \leq \frac  {x^{n}} {n^k}.\] Recall that $x=1/\lambda_c$ and by using
  Proposition~\ref{prop3},  the  effective  resistance  at  parameter
  $\lambda > \lambda_c$ is given by
  \[ \mathcal R  (\lambda, \mathcal T) = \sum \frac  {1} {\lambda^n |\mathcal T_n|} \geq  \sum \frac
    {n^k}       {{(\lambda      x)}^n}=\sum
    {n^k}      {\left (\frac{\lambda_c}{\lambda}\right)}^n.  \]
  By an easy computation, there exists a constant $C_k>0$ such that
    $$\sum
      {n^k}      {\left (\frac{\lambda_c}{\lambda}\right)}^n       \sim_{\lambda \rightarrow \lambda_c^+}      \frac       {C_k}{{\left(1-\frac{\lambda_c}{\lambda}\right)}^{k+1}}  \sim_{\lambda \rightarrow \lambda_c^+}      \frac       {\lambda^{k+1} C_k}{{\left(\lambda-\lambda_c\right)}^{k+1}},$$
    implying that there exists a constant $D_k>0$, uniform in $\lambda$ close to $\lambda_c$, such that $\mathcal R  (\lambda, \mathcal T) \geq \frac       {D_k}{{\left(\lambda-\lambda_c\right)}^{k+1}}$. An upper bound of the same order can be obtained in a very similar fashion, leading to the conclusion.
\end{proof}

We end  this subsection with  the following proposition,  showing that
discontinuities can occur elsewhere than at $\lambda_c$:

\begin{proposition}
  \label{notcontinuous}
  There exists a  tree $\mathcal T$ such that the function  $C(\lambda,\mathcal T)$ is not
  continuous on $(\lambda_c,1)$, i.e it will discontinuous at a certain $\lambda'\in (\lambda_c,1)$.
\end{proposition}

\begin{proof}
  Let  $0<\lambda_1<\lambda_2<1 $.  By Proposition~\ref{prop5},  there
  exist    two trees      $ \mathcal H$ and $ \mathcal G $               such               that
  $\lambda_c(\mathcal H)=\lambda_1,\lambda_c(\mathcal G)=\lambda_2$ and
  \begin{equation}
    \label{equa:discontinuous1}
    \mathcal C(\lambda_1,\mathcal H)=0, \, \, \mathcal C(\lambda_2,\mathcal G)>0.
  \end{equation}

  We construct a tree $\mathcal{T}$ rooted at $o$ as  follows:
  $$\mathcal{T}_1=\left \{ \nu_1,\nu_2 \right \},\, \, \, \, \mathcal{T}^{\nu_1}=\mathcal H \, \, \, \text{ and } \mathcal{T}^{\nu_2}=\mathcal G.$$
  Hence,   $$\lambda_c(\mathcal T)=\lambda_1.$$

  Denote by  $\deg \nu_1$ (resp. $\deg \nu_2$) the degree of $\nu_1$ (resp. $\nu_2$) in the tree $\mathcal{T}$. By an easy computation, for any $\lambda\in (\lambda_1,1)$, we obtain:
  \begin{equation}
    \label{equa:discontinuous2}
    \mathcal C(\lambda, \mathcal{T})=\frac{1}{2}\times \frac{\lambda \mathcal C(\lambda, \mathcal H) \deg \nu_1}{1+\lambda \mathcal C(\lambda, \mathcal H) \deg \nu_1}+ \frac{1}{2}\times \frac{\lambda \mathcal C(\lambda, \mathcal G) \deg \nu_2}{1+\lambda \mathcal C(\lambda, \mathcal G) \deg \nu_2}.
  \end{equation}

  By Corollary~\ref{cor:continuousofSS}, the function $\mathcal C(\lambda,\mathcal H)$ is continuous on $(\lambda_1,1)$ and since $\mathcal C(\lambda, \mathcal G)=0$ for any $\lambda\in (\lambda_1,\lambda_2)$, therefore:
  \begin{equation}
    \label{equa:discontinuous3}
    \underset{\lambda\rightarrow \lambda_2^-}{\lim}\mathcal C(\lambda, \mathcal T)=\frac{1}{2}\times \frac{\lambda_2 \mathcal C(\lambda_2, \mathcal H) \deg \nu_1}{1+\lambda_2 \mathcal C(\lambda_2, \mathcal H) \deg \nu_1}.
  \end{equation}
  By Equations~\eqref{equa:discontinuous1},~\eqref{equa:discontinuous2} and~\eqref{equa:discontinuous3}, we obtain:
  $$\underset{\lambda\rightarrow \lambda_2^-}{\lim}\mathcal C(\lambda, \mathcal T)<\mathcal C(\lambda_2, \mathcal{T}).$$
  The latter inequality implies that the function  $\mathcal C(\lambda,\mathcal T)$  is  discontinuous  at $\lambda_2$.
\end{proof}

Note  that continuity  properties  at $\lambda  \geq  1$ are  actually
easier to obtain, and we will investigate them further below.

\subsection{The convergence of the law of the first \texorpdfstring{$k$}{k} steps}

\mbox{}

If
$\lim_{\lambda   \rightarrow   \lambda   _c,\lambda>\lambda_c}\mathcal
  C(\lambda,\mathcal  T)>0$,  by  Lemma~\ref{lemmactggg}  the  limit  of
$\varphi^{\lambda,k}(y_1,  \ldots, y_k)$  when $\lambda$  decreases to
$\lambda_c$           exists.           If           one           has
$\lim_{\lambda \downarrow \lambda_c}\mathcal C(\lambda,\mathcal T)=0$,
the situation is more delicate and we cannot yet conclude on the limit
of  the  function  $\varphi^{\lambda,k}(\nu_0,  \ldots,  \nu_k)$  when
$\lambda$ decreases  to $  \lambda_c $.  Indeed, convergence  does not
always hold,  as we  will see  in a counterexample.  The idea  of what
follows  is  easy to  describe:  we  are  going  to construct  a  very
inhomogeneous  tree  with  various   subtrees  of  higher  and  higher
branching numbers, at locations alternating  between two halves of the
whole tree; a biased random walk  will wander until it finds the first
such subtree  inside which  it is transient,  and escape  to infinity
within this subtree with high probability.

\begin{proposition}
  \label{example1}
  There    exists    a   tree    $\mathcal T$    such    that   the    function
  $\varphi^{\lambda,1}(y_0,y_1)$     does      not     converge     as
  $\lambda \to \lambda_c$.
\end{proposition}

\begin{notation}
  Let $\mathcal T, \mathcal T'$  be two trees and  $A \subset V(\mathcal T)$. We  can construct a
  new tree by grafting  a copy of $\mathcal T'$ at all the  vertices of $A$; we
    will denote by $\mathcal T\overset{A }{\bigoplus}\mathcal T'$ this new tree. Note that
  for all  $x \in A$,  ${(\mathcal T\overset{A }{\bigoplus}\mathcal T')}^x \simeq  \mathcal T'$. In
  the   case  $A=\{x   \}$,   we  will   use   the  simpler   notation
  $\mathcal T\overset{x}{\bigoplus}\mathcal T'$ for $\mathcal T\overset{\{x\}}{\bigoplus}\mathcal T'$.
\end{notation}

\begin{proof}
  Fix $\epsilon>0$  small enough. By Proposition~\ref{prop5},  for all
  $0<a\leq 1$, there exists a tree,  denoted by $\mathcal T(a)$, such that its
  branching   number   is    $\frac{1}{a}$   and   $\mathcal C(a,\mathcal T(a))=0$.   Let
  $\mathcal H=\mathbb{Z}$, seen as  a tree rooted at $0$, so  that the integers
  are the  vertices of  $\mathcal H$ (see  the Figure~\ref{treeTinfty}).  We are
  going to construct a tree inductively.

  Let  ${(a_{i})}_{i  \geq  1}$  be  a  decreasing  sequence  such  that
  $a_{1}<1$. {Set} $a_{c} :=  \lim a_i$ and  assume that $a_c>0$. Choose a
  sequence ${(b_{i})}_{i\geq 1}$  such that $b_{i}\in (a_{i+1},a_{i})$  for all $i$.
  First,                                                           set
  $\mathcal H^0:=(\mathcal H\overset{-2                   }{\bigoplus}\mathcal T(a_1))\overset{2
    }{\bigoplus}T(a_2)$. We consider the biased random walk $RW_{b_1}$,
  then it is recurrent on $\mathcal T(a_1)$ and transient on $\mathcal T(a_2)$. On
  $\mathcal H^{0}$,  the biased  random walk  $RW_{b_1}$ is  transient, and  in
  addition  we  know that  it  stays  eventually  within the  copy  of
  $\mathcal T(a_2)$.  There exists  then $N_{1}>2$  such that  the probability
  that the  limit walk remains  in that  copy after time  $N_{1}-1$ is
  greater than $1-\epsilon$.

  Then we set $\mathcal H^1=(\mathcal H^0\overset{-N_1}{\bigoplus}\mathcal T(a_3))$. On $\mathcal H^{1}$,
  the walk of bias $b_1$ is  still transient and still has probability
  at  least $1-\epsilon$  to  escape through  the  copy of  $\mathcal T(a_2)$,
  because $\mathcal T(a_3)$  is grafted too far  to be relevant. On  the other
  hand,  consider  the biased  random  walk  $RW_{b_2}$: it  is  still
  transient on $\mathcal H^1$ but only through the new copy of $\mathcal T(a_3)$. There
  exists then $N_{2}>2$ such that  the probability that the limit walk
  remains  in  that   copy  after  time  $N_{2}-1$   is  greater  than
  $1-\epsilon$.

  We can set  $\mathcal H^2:=(\mathcal H^1\overset{N_2}{\bigoplus}\mathcal T(a_4))$ and continue
  this  procedure to  graft all  the trees  $\mathcal T(a_i) $,  further and
  further from  the origin and  alternatively on  the left and  on the
  right; denote by $ \mathcal H^{\infty} $ the union of all the $\mathcal H^k$.

  \begin{figure}[ht!]
    \centering
    \includegraphics[scale=0.6]{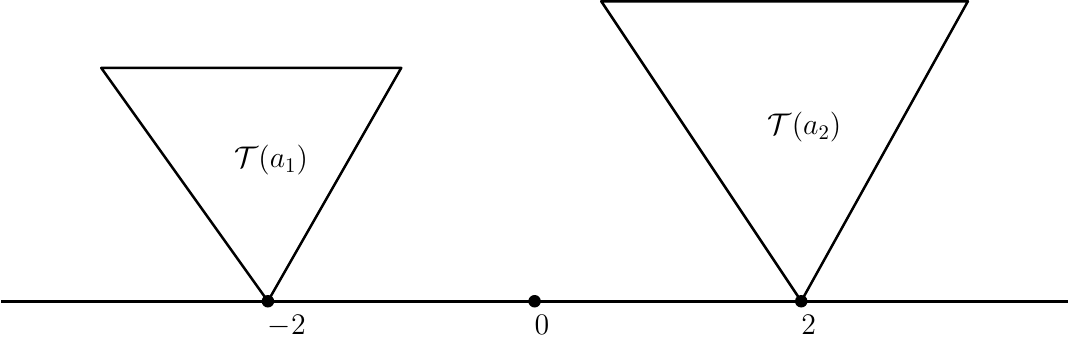}
    \caption{Tree $H^{0}$}
    \label{fig:1222}
    \label{treeTinfty}
  \end{figure}

  It remains to show  that the function $\varphi^{\lambda,1}(0,1)$ for
  the biased random  walk on the tree $\mathcal H^{\infty}$  does not converge.
  We  have $br(\mathcal H^{\infty})=\sup_{i}br(\mathcal T(a_i))=\frac{1}{a_{c}}$  and
  $\varphi^{b_{i},1}(0,1)  \geq  1-\epsilon$  if   $i$  is  odd  while
  $\varphi^{b_{i},1}(0,-1) \geq 1-\epsilon$ if $i$ is even. Then,
  $$ \forall k \geq 0,
    \begin{cases}
      \varphi^{b_{i},1}(0,1)\geq 1-\epsilon & \text{~if~} i= 2k+1 \\
      \varphi^{b_{i},1}(0,1) \leq \epsilon  & \text{~if~} i= 2k+2
    \end{cases}     $$     This     implies    that     the     function
  $\varphi^{\lambda,1}(0,1)$ does  not converge  when $\lambda$  go to
  $a_c$.
\end{proof}

The tree we just constructed is tailored to be extremely inhomogeneous.
At the other end of the spectrum, some trees have enough structure for
all the functions we are considering to be essentially explicit:

\begin{definition}
  A  tree  $\mathcal T$  is  called  \emph{periodic} (or  \emph{finite
    type}) if, for all $v  \in V(\mathcal T)\setminus \{o\}$, there is
  a bijective  morphism $f:\mathcal T^v\rightarrow  \mathcal T^{f(v)}$
  with  $f(v)$  in  a  fixed,  finite  neighborhood  of  the  root  of
  $\mathcal T$.
\end{definition}

\begin{definition}
  Let $\mathcal T$ be a  finite tree and $\mathcal L(\mathcal T)$ be the set of  leafs of $\mathcal T$. We
  set                               $\mathcal T^1=\mathcal T\overset{\mathcal L(\mathcal T)}{\bigoplus}\mathcal T$,
  $\mathcal T^2=\mathcal T^1\overset{\mathcal L(\mathcal T^1)}{\bigoplus}\mathcal T$,                        \dots,
  $\mathcal T^n=\mathcal T^{n-1}\overset{\mathcal L(\mathcal T^{n-1})}{\bigoplus}\mathcal T$ for every  $n \geq 2$.
  We continue this procedure an infinite  number of times to obtain an
  infinite tree $\mathcal T^{\infty,\mathcal T}$. Note  that $\mathcal T^{\infty,\mathcal T}$ is  also a
  periodic tree.
\end{definition}


\begin{fact}[see Lyons~\cite{lyons1990random}, theorem 5.1]
  \label{factlyons}
  Let $\mathcal T$  be a  periodic tree  and $(\nu_0=o,\nu_1,\nu_2,\ldots,\nu_k)$  be a
  simple      path      on       $\mathcal T$.      Then     
  $\varphi^{\lambda,k}(\nu_0,\nu_1,\ldots,\nu_k)$  converges when  $\lambda$
  decreases to  $\lambda_c(\mathcal T)$. 
\end{fact}

In  the rest  of this  section we
provide a new proof of a particular case (the case of $\mathcal T^{\infty,\mathcal T}$) of fact~\ref{factlyons}:

\begin{proposition}
  \label{proposition6.1}
  Let  $\mathcal T$ be  a  finite tree  and  $(y_0=o, y_1, \ldots, y_k)$ be  a
  simple    path     on    $\mathcal T^{\infty,\mathcal T}$.    Then     the    function
  $\varphi^{\lambda,k}(y_0,y_1,\ldots,y_k)$      of     $\mathcal T^{\infty,\mathcal T}$
  converges when $\lambda$ decreases to $\lambda_c(\mathcal T^{\infty,\mathcal T})$.
\end{proposition}

Before  showing  Proposition~\ref{proposition6.1},  we  need to show the following lemma:


\begin{lemma}
  \label{lemma6.1}
  Let $\mathcal T$ be a tree rooted at $o$ such that $\deg o=d_0$ and
  $$\left\{\begin{matrix}
      \mathcal T_1=\left \{ \nu_1,\nu_2,\ldots,\nu_{d_0} \right \} \\
      \forall i   \in  \{  1,2,\ldots,d_0
      \},\lambda_c(\mathcal T)=\lambda_c(\mathcal T^{\nu_i})=\lambda_c        \text{~and~}
      \mathcal C(\lambda_c,\mathcal T)=\mathcal C(\lambda_c,\mathcal T^{\nu_i})=0
    \end{matrix}\right.$$
  Then for all $i$, we have $\widetilde{\mathcal C}(\lambda,\mathcal T,\nu_i)=\frac{(d_{\nu_i}-1) \lambda \widetilde{\mathcal C}(\lambda,\mathcal T^{\nu_i})}{d_0(1+(d_{\nu_i}-1)\lambda \widetilde{\mathcal C}(\lambda,\mathcal T^{\nu_i}))}$, where $d_{\nu_i}=\deg \nu_i$. 
\end{lemma}

\begin{proof}
  Recall that $\widetilde{\mathcal C}(\lambda,\mathcal T,\nu_i)=\mathbb{P}(\mathcal{A})$, where $\mathcal A$ is the event that the random walk  $RW_\lambda$ on $\mathcal T$, started at
  the    root (i.e $X_0=o$),     never    returns    to    it and reached $\nu_i$ at the first step (i.e $X_1=\nu_i$). We can write
  $$\mathcal{A}=\underset{k\geq 0}{\bigcup}\mathcal A_k$$
  where
  $$\mathcal{A}_k:=\left\{\#\{j>0: X_j=o\}=0\right\} \cap \{X_1=\nu_{i}\} \cap \left\{\#\{j>1: X_j=\nu_i\}=k\right\}.$$
  Let $m=\frac{(d_{\nu_i}-1)\lambda }{1+(d_{\nu_i}-1)\lambda}$ and $c=\widetilde{\mathcal C}(\lambda,\mathcal T^{\nu_i})$. Note that the sequence $(\mathcal A_k, k\geq 0)$ are pairwise disjoint and $\mathbb{P}(\mathcal{A}_k)=\frac{mc{(m(1-c))}^k}{d_0}$, therefore we obtain:

  \begin{align*}
    \widetilde{\mathcal                             C}(\lambda,\mathcal
    T,\nu_i) & =\frac{mc}{d_0}\sum_{k=0}^{\infty}{(m(1-c))}^k=\frac{(d_{\nu_i}-1)
      \lambda     \widetilde{\mathcal     C}(\lambda,\mathcal
      T^{\nu_i})}{d_0(1+(d_{\nu_i}-1)\lambda
      \widetilde{\mathcal C}(\lambda,\mathcal T^{\nu_i}))}. \qedhere
  \end{align*}
\end{proof}

\begin{proof}[\bf Proof of Proposition~\ref{proposition6.1}]
  First, since $\mathcal T^{\infty, \mathcal T}$ is a periodic tree, the biased random walk $RW_{\lambda_c}$ on $\mathcal T^{\infty, \mathcal T}$ is
  recurrent (see~\cite{lyons1990random}). Recall that
  $L(\mathcal T)$ is the  set of all leafs  of the finite tree $\mathcal T$; let $S^i$ be
  the set  of all finite  paths starting  at the origin, ending  at one
  element  of $L(\mathcal T)$  and  passing  through $\nu_i$.  For all
  $\nu \in L(\mathcal T)$, we have  ${(\mathcal T^{\infty, \mathcal T})}^\nu = \mathcal T^{\infty, \mathcal T}$ and we can apply Lemma~\ref{lemma6.1} several times to obtain:
  $$\widetilde{\mathcal C}(\lambda,\mathcal T^{\infty, \mathcal T},\nu_i)=\sum_{\gamma            \in
      S^i}f^{\gamma}_1(\lambda)f^{\gamma}_2(\lambda)\cdots f^{\gamma}_{\left | \gamma \right |}(\lambda)\widetilde{\mathcal C}(\lambda,{(\mathcal T^{\infty, \mathcal T})}^{\gamma_{\left | \gamma \right |}}),$$
  where $f^{\gamma}_j(\lambda)=\frac{m_{\gamma_j}\lambda}{m_{\gamma_{j-1}}(1+m_{\gamma_j}\lambda \mathcal C(\lambda,{(\mathcal T^{\infty,\mathcal{T}})}^{\gamma_j}))}$ and $m_{\gamma_j}=d_{\gamma_j}-1$ if $j>1$ and $m_{\gamma_0}=d_0$.
  Moreover, we have
  $$\widetilde{\mathcal C}(\lambda,{(\mathcal T^{\infty, \mathcal T})}^{\gamma_{\left | \gamma \right |}})=\widetilde{\mathcal C}(\lambda,\mathcal T^{\infty, \mathcal T})$$ then
  $$\widetilde{\mathcal C}(\lambda,\mathcal T^{\infty, \mathcal T},\nu_i)=\sum_{\gamma            \in
      S^i}f^{\gamma}_1(\lambda)f^{\gamma}_2(\lambda) \cdots f^{\gamma}_{\left | \gamma \right |}(\lambda)\widetilde{\mathcal C}(\lambda,\mathcal T^{\infty, \mathcal T}).$$
  By Lemma~\ref{lemmactggg}, we obtain
  $$\varphi^{\lambda,1}(o,\nu_i)=\frac{\widetilde{\mathcal C}(\lambda,\mathcal T^{\infty, \mathcal T},\nu_i)}{\widetilde{\mathcal C}(\lambda,\mathcal T^{\infty, \mathcal T})}=\sum_{\gamma
      \in S^i}f^{\gamma}_1(\lambda)f^{\gamma}_2(\lambda) \cdots f^{\gamma}_{\left | \gamma \right |}(\lambda).$$
  Note that for all $\gamma \in S^i$ we have $ m_{\gamma_0}=m(\gamma_{\left | \gamma \right |})$. Moreover, since the biased random walk $RW_{\lambda_c}$ on $\mathcal T^{\infty, \mathcal T}$ is recurrent, for all $0\leq j\leq \left|\gamma \right |$, we have:
  $$\mathcal C \left( \lambda_c
        \left(\mathcal T^{\infty, \mathcal{T}}\right),
        {\left(\mathcal T^{\infty,\mathcal{T}}\right)}^{\gamma_j}\right)=0.$$
  Hence, $\varphi^{\lambda,1}(o,\nu_i)$ converges when $\lambda$ decreases towards $\lambda_c(\mathcal T^{\infty, \mathcal T})$ and
  \begin{equation}
    \label{equation4}
    \lim_{\lambda \rightarrow \lambda_c(\mathcal T^{\infty, \mathcal T})}\varphi^{\lambda,1}(o,\nu_i)=\sum_{\gamma\in S^i}\lambda_c^{\left | \gamma \right |}.
  \end{equation}
  This proves the statement of the proposition in the case $k=1$; general cases are handled in a very similar fashion, enumerating the vertices at distance $k$ from the root rather than children of $o$.
\end{proof}

\begin{remark}
  The  equation~\eqref{equation4}  gives us  a  way  to calculate  the
  critical value of $RW_{\lambda}$ on $\mathcal T^{\infty, \mathcal T}$, as the solution of
  the following equation:
  \begin{equation}
    \label{eq:critvalueperiodic}
    \sum_{i=1}^{m_o}\sum_{\gamma  \in  S^i}x^{\left|  \gamma  \right|}=1.
  \end{equation}
\end{remark}

\section{The continuity of effective conductance}
\label{sectioncontinuity}

We  end  the first  half  of  the paper  with  a  few results  on  the
conductance functions  of trees, namely  we prove a criterion  for the
continuity  of  $\mathcal  C(\lambda,\mathcal T)$  in  $\lambda$  (see
Theorems~\ref{ctg}  and~\ref{ctgg}  below)   and  study  the  set  of
conductance functions of spherically symmetric trees of bounded degree
(see Theorem~\ref{theorem1.4}).

\subsection{Left- and right-continuity of \texorpdfstring{$\mathcal{C} (\mathcal{T}, \lambda)$}{C (T,λ)}}
\label{sectiononcontinuous}

\begin{lemma}
  \label{lemma2}
  Let $\mathcal T$ be an infinite, locally finite and rooted tree. Then $\mathcal C(\lambda,\mathcal T)$ is right
  continuous on $\left (0,+\infty \right )$.
\end{lemma}

\begin{proof}
  Let $(X_n, n\geq 0)$ be the biased random walk with parameter $\lambda$ on $\mathcal{T}$. We  define  $S_{0}:=\inf  \left\{k>0: X_k=o  \right\}$ and for any $n>0$,
  $$ S_{n}:=\inf  \left\{k>0: d(o,X_{k})=n  \right\}.$$

  Recall that the random walk on a network $(\mathcal{T}, c)$, where $c(e)=\lambda^{|e|}$ is exactly the same  process as the biased random walk with parameter $\lambda$. We use Equation~\eqref{equ:relationconductanceproba} to obtain
  $$\mathcal C(\lambda,\mathcal T)=\pi(o)\lim_{n\rightarrow + \infty}
    \mathbb{P}(S_{n}<S_{0}).$$

  We set  $\mathcal C(\lambda,\mathcal{T}, n):=\pi(o) \mathbb{P}(S_{n}<S_{0})$. It  is easy  to see
  that  $\mathcal C(\lambda,\mathcal{T},n)  \geq  C(\lambda,\mathcal T, n+1)$. On the other hand, by
  Lemma~\ref{lemmac}, we obtain  $\mathcal C(\lambda,\mathcal T, n)$   is  a   continuous function. Hence, $\mathcal C(\lambda,\mathcal T, n)$ is a continuous  increasing function for each $n$. It implies that $\mathcal C(\lambda,\mathcal T)$  is the decreasing limit of increasing  functions.  Therefore  $\mathcal C(\lambda,\mathcal T)$   is  right
  continuous.
\end{proof}
\begin{definition}
  Let $\mathcal  T$ be a tree. For  each $\nu \in  \mathcal T$, denote by  $X^{\nu}_{n}$ the biased random walk on the subtree $\mathcal T^\nu$ (i.e
  $X^{\nu}_0=\nu$ and  $\forall n>0, X^\nu_n \in  \mathcal T^\nu$). We
  say that $\mathcal T$ is \emph{uniformly transient} if
  $$\forall  \lambda>\lambda_{c},\exists \alpha_{\lambda}>0,\forall  \nu
    \in   \mathcal T,\mathbb{P}(\forall     n>0,    X^{\nu}_{n}     \neq    \nu)\geq
    \alpha_{\lambda}.$$
  It is called \emph{weakly uniformly transient} if there exists a sequence of
  finite pairwise disjoint V-cutsets $(\pi_n, n\geq 1)$, such that
  $$\forall  \lambda>\lambda_{c},\exists \alpha_{\lambda}>0,\forall  \nu
  \in   \bigcup_{k=1}^{+\infty} \pi_k,\mathbb{P}(\forall  n>0,X^{\nu}_{n}\neq   \nu)\geq
  \alpha_{\lambda}.$$
\end{definition}

\begin{remark}
  It is easy to see that if $\lambda_c(\mathcal T)=1$, then $\mathcal T$ is uniformly transient: indeed, on every infinite subtree and for every $\lambda>1$, escape probabilities are bounded below by the escape probability in $\mathbb Z_+$ which is itself strictly positive for $\lambda>1$.
\end{remark}

\begin{theorem}
  \label{ctg}
  Let $\mathcal T$ be a uniformly transient  tree. Then $\mathcal C(\lambda,\mathcal T)$ is
  left  continuous  on  $(\lambda_{c},+\infty)$.
\end{theorem}

\begin{proof}
  Fix $\lambda_{1}>\lambda_{c}$, we will  prove that $\mathcal C(\lambda,\mathcal T)$ is
  left        continuous        at        $\lambda_{1}$.        Choose
  $\lambda_{0} \in (\lambda_{c},\lambda_{1})$. By Theorem~\ref{RL}, we
  can   find   a   constant   $\alpha  >0$   (does   not   depend   on
  $\lambda \in [\lambda_{0},\lambda_{1}]$) such that
  $$\forall \lambda \in  [\lambda_{0},\lambda_{1}],\forall \nu \in V(\mathcal T),
    \mathbb{P}(\forall  n>0,X^{\nu}_{n}\neq \nu)\geq  \alpha.$$ Given  a
  family                        of                        conductances
  $c =  {c(e)}_{e\in E(\mathcal T)}\in {(0,+\infty)}^{E}$,  let $Y_{n}$ be
  the associated random  walk. Let $A \subset  {(0,+\infty)}^{E}$ be the
  subset  of  elements  of  ${(0,+\infty)}^{E}$  such  that  $Y_{n}$  is
  transient  for those  choices of  conductances. Then  we define  the
  function $\psi : A \to \mathbb R_+^*$ as
  $$\psi(c):=\mathcal
    C_c(\mathcal T).$$ Recall  that $\mathcal T_k$ is  the collection of
  all the vertices at distance $k$ from the root: then we have
  $$\mathcal C(\lambda,\mathcal T)=\psi(
    \underbrace{\lambda,\lambda,\ldots\lambda}_{\left | \mathcal T_1  \right |} ,
    \underbrace{\lambda^2,\lambda^2,\ldots\lambda^2}_{\left | \mathcal T_2 \right|} ,\ldots.).$$
  We will abuse notation until the end of the argument, writing
  $$\psi(\lambda_1,\lambda_2^2,\lambda_3^3,\ldots) \quad \text{for}
    \quad  \psi( \underbrace{\lambda_1,\lambda_1,\ldots\lambda_1}_{\left|             \mathcal T_1             \right            |}             ,
    \underbrace{\lambda_2^2,\lambda_2^2,\ldots\lambda_2^2}_{\left  | \mathcal T_2
      \right     |}     ,\ldots)$$     so     that     in     particular
  $\mathcal C(\lambda,\mathcal T)=\psi(\lambda,\lambda^2,\lambda^3,\ldots)$.

  Let        $\epsilon        >0$,       we        choose        $L\in
    \mathbb{N}$  such that  ${(1-\alpha)}^{L}<\epsilon$. For  $\lambda \in
    (\lambda_{0},\lambda_{1})$       we        have       $\left       |
    \mathcal C(\lambda_1,\mathcal T)-\mathcal C(\lambda,\mathcal T)  \right  |=\left  |  \psi(\lambda_{1},
    \lambda_{1}^{2},         \lambda_{1}^{3},        \ldots)         -
    \psi(\lambda,\lambda^{2},\lambda^{3},\ldots)                \right|$ and by the triangular inequality, we get
  \begin{align}
    \label{equation1}
    \left | \mathcal C(\lambda_{1},\mathcal T)-\mathcal C(\lambda,\mathcal T) \right | & \leq \left |
    \psi(\lambda_{1},\ldots,\lambda_{1}^L,b_{1})                   -
    \psi(\lambda,\ldots,\lambda^L,b_{1})
    \right |
    \nonumber                                                                                         \\ & +\left |
    \psi(\lambda,\ldots,\lambda^L,b_{1})                           -
    \psi(\lambda,\ldots,\lambda^L,b)
    \right |
  \end{align}
  where        $b:={(\lambda^{L+k})}_{k        \geq       1}$        and
  $b_{1}:={(\lambda_{1}^{L+k})}_{k\geq 1}$.

  Let $\lambda'  \in \left [  \lambda_0,\lambda_1 \right ]$ and denote by
    $S^{\lambda'}_n$ the first hitting point of $\mathcal T_n$ by the random walk
    with conductances
  $$(\underbrace{\lambda,\ldots,\lambda}_{\left    |   \mathcal T_1
      \right |} ,\underbrace{\lambda^2,\ldots,\lambda^2}_{\left|                 \mathcal T_2                  \right                 |}
    ,\ldots,\underbrace{\lambda^L,\ldots,\lambda^L}_{\left    |
      \mathcal T_L                 \right                  |}                 ,
    \underbrace{{(\lambda')}^{L+1},\ldots,{(\lambda')}^{L+1}}_{\left        |
      \mathcal T_{L+1}                        \right                        |},\ldots$$  We  can   see  that  the  law  of
  $S^{\lambda_1}_L$ and the law of $S^{\lambda}_L$ are identical. Since $\mathcal T$ is uniformly transient, then when
  the random  walk reaches $\mathcal T_L$, it  returns to $o$ with  a probability
  strictly smaller than ${(1-\alpha)}^{L}$. It implies that
  \begin{equation}
    \label{equation2}
    \left                                                              |
    \psi(\lambda,\ldots,\lambda^L,b_{1})                         -
    \psi(\lambda,\ldots,\lambda^L,b)
    \right |)\leq 2{(1-\alpha)}^{L}\leq 2\epsilon.
  \end{equation}

  It                remains                 to                estimate
  $\left     |      \psi(\lambda_{1},\ldots,\lambda_{1}^L,b_{1})     -
    \psi(\lambda,\ldots,\lambda^L,b_{1})      \right       |$.      By
  Theorem~\ref{RL}, we have
  $$\psi(\lambda_{1},\ldots,\lambda_{1}^L,b_{1})\geq \mathcal{C}(\lambda_0, \mathcal T)>0 \text{ } \text{and}
    \text{   }   \psi(\lambda,\ldots,\lambda^L,b)\geq \mathcal{C}(\lambda_0, \mathcal T)>0.$$
  We   apply   the Lemma~\ref{lemmac} to obtain
  \begin{equation}
    \label{equation3}
    \exists \delta >0, \forall \lambda \in \left [ \lambda_1-\delta,\lambda_1
      \right ],
    \left | \psi(\lambda_{1},\ldots,\lambda_{1}^L,b_{1}) -
    \psi(\lambda,\ldots,\lambda^L,b_{1}) \right |<\epsilon.
  \end{equation}

  We            combine~\eqref{equation1},~\eqref{equation2}
  and~\eqref{equation3} to get
  $$\exists \delta >0, \forall \lambda \in \left [ \lambda_0,\lambda_1
      \right     ]    \text{     }    \text{such     that}    \text{     }
    \lambda_1-\lambda<\delta:   \left  |  \mathcal  C(\lambda_{1},\mathcal T)-\mathcal C(\lambda,\mathcal T)
    \right |\leq  3\epsilon.$$ This implies that  $\mathcal C(\lambda,\mathcal T)$ is left
  continuous  at  $\lambda_1$.
\end{proof}

Using the same  method as in the proof of Theorem~\ref{ctg},  we can prove
the slightly stronger result (the proof of which we omit):

\begin{theorem}
  \label{ctgg}
  Let $\mathcal T$  be a weakly  uniformly transient tree: then  the effective
  conductance $\mathcal C(\lambda,\mathcal T)$ is left continuous on $(\lambda_{c},1]$.
\end{theorem}

\subsection{Proof of Theorem~\ref{theorem1.4}}
\label{section3.4}

\begin{definition}
  Let  $(\mathcal T^n, n\geq 1)$ be  a  sequence of infinite, locally finite and rooted trees. We  say  that $\mathcal T^n$  converges
  locally             towards              $\mathcal T^{\infty}$             if
  $\forall  k, \exists n_0, \forall n \geq n_0, \mathcal T^n_{\leq k}=\mathcal T^{\infty}_{\leq k}$, where $ \mathcal T_{\leq k}$ is a finite tree defined by:
  $$\left\{\begin{matrix}
      V(\mathcal T_{\leq k}):=\left \{ \nu \in V(\mathcal T), d(o,\nu)\leq k \right\} \\
      E(\mathcal T_{\leq k})=E_{\mid   V(\mathcal T_{\leq k})\times V(\mathcal T_{\leq k})}
    \end{matrix}\right.$$
\end{definition}


Recall from the introduction that $\mathbb F_m$ denotes the collection
of all effective conductance functions for spherically symmetric trees
with degree uniformly bounded by $m$.

\begin{lemma}
  \label{2}
  Let $(f_n, n\geq  1)$ be a sequence of  functions in $\mathbb{F}_m$.
  Assume  that  $f_n$ converges   towards some function $f$. Then,  there
  exists  a  function   $g  \in  \mathbb{F}_m$  such   that,  for  any
  $\lambda>0$,
  $$f(\lambda)\leq g(\lambda).$$
\end{lemma}

\begin{proof}
  Recall from the introduction that $A_m$ denotes the collection of all spherically symmetric trees with maximal degree at most $m$;
  let  $(\mathcal T^n, n\geq 1)$  be  a  sequence  of  elements  of  $A_m$  such  that, for every $n>0$,
  $$f_n(\lambda)=\mathcal C(\lambda,\mathcal T^n).$$  Since the  degree  of
  vertices  of $\mathcal T^n$  are  bounded by  $m$, we  can  apply the  diagonal
  extraction  argument.  After  renumbering indices,  there  exists  a
  subsequence  of  $(\mathcal T^n, n\geq 1)$,  denoted also  by  $(\mathcal T^n, n\geq 1)$ below, which converges
  locally towards some tree $\mathcal T^{\infty} \in A_m$. Moreover, we can assume that for any $n>0$,
  \begin{equation}
    \label{equ:rec0}
    \mathcal T^n_{\leq n}=\mathcal T^{\infty}_{\leq n}.
  \end{equation}

  Set  $g(\lambda)=\mathcal C(\lambda,  \mathcal T^{\infty})$, it  remains  to  show that for every $\lambda>0$,
  $$f(\lambda)\leq g(\lambda).$$
  Assume     that    there     exists     $\lambda_0$    such     that
  $f(\lambda_0)>g(\lambda_0)$ and  let $c:=f(\lambda_0)-g(\lambda_0)>0$.
  Since the sequence $(f_n(\lambda_0), n\geq 1)$ converges towards $f(\lambda_0)$,
  \begin{equation}
    \label{equ:rec2}
    \exists \ell_1>0, \forall n\geq \ell_1,
    f_n(\lambda_0)>f(\lambda_0)-\frac{c}{4}.
  \end{equation}
  Recall the definition of the function $\mathcal{C}(\lambda_0, \mathcal{T}, n)$ in the proof of Lemma~\ref{lemma2}, the  sequence
  $\left(\mathcal C(\lambda_0, \mathcal T^{\infty}, n), n\geq 1\right)$ decreases towards $g(\lambda_0)$, implying that
  \begin{equation}
    \label{equ:rec3}
    \exists             \ell_2>0,              \forall             n\geq \ell_2,
    \mathcal C(\lambda_0,\mathcal T^{\infty},n)<g(\lambda_0)+\frac{c}{4}.
  \end{equation}
  Letting  $\ell:=\ell_1 \vee\ell_2$, combine~\eqref{equ:rec2} and~\eqref{equ:rec3} to obtain:
  \begin{equation}
    \label{equ:rec4}
    f_\ell(\lambda_0)>f(\lambda_0)-\frac{c}{4}\quad\text{and}\quad
    \mathcal C(\lambda_0,\mathcal T^{\infty}, \ell)<g(\lambda_0)+\frac{c}{4}.
  \end{equation}

  On the other hand    $\mathcal C(\lambda_0,\mathcal T^\ell, \ell)=\mathcal C(\lambda_0,\mathcal T^{\infty},\ell)$ and  by~\eqref{equ:rec4} we obtain:
  \begin{equation}
    \label{equ:rec4bis}
    f_\ell(\lambda_0)>f(\lambda_0)-\frac{c}{4}\quad \text{and}\quad
    \mathcal C(\lambda_0,\mathcal T^{\ell}, \ell)<g(\lambda_0)+\frac{c}{4}.
  \end{equation}
  The   sequence  $\left(\mathcal C(\lambda_0,\mathcal T^\ell, k), k\geq 1\right)$   decreases  towards
  $f_\ell(\lambda_0)$ when $k$ goes to $+\infty$. Hence,
  \begin{equation}
    \label{equ:rec5}
    f_\ell(\lambda_0)\leq  \mathcal C(\lambda_0,\mathcal T^{\ell}, \ell) < g(\lambda_0)+\frac{c}{4}.
  \end{equation}
  From~\eqref{equ:rec4bis} and~\eqref{equ:rec5} we obtain
  $$f(\lambda_0)-\frac{c}{4}<f_\ell(\lambda_0)<g(\lambda_0)+\frac{c}{4},$$
  hence
  $c=f(\lambda_0)-g(\lambda_0)<\frac{c}{2}$,
  leading to a contradiction.
\end{proof}


\begin{proof}[Proof of Theorem~\ref{theorem1.4}]
  Let  $(\mathcal T^n, n\geq 1)$  be  a  sequence  of  elements  of  $A_m$  such  that, for any $n>0$,
  $$f_n(\lambda)=\mathcal C(\lambda,\mathcal T^n).$$
  Fix a sub-sequence of $(\mathcal T^n, n\geq 1)$ which converges locally towards $\mathcal{T}^{\infty}$ and such that~\eqref{equ:rec0} holds as  in  the   proof  of  the Lemma~\ref{2}. We set $g(\lambda)=\mathcal{C}(\lambda, \mathcal T^{\infty})$ and we need to prove that $f=g$.

  By  Lemma~\ref{2}, we have $f(\lambda)  \leq g(\lambda)$. Assume
  that       there       exists      $\lambda_0$       such       that
  $0<f(\lambda_0)<g(\lambda_0)$. We prove that for any $\lambda<\lambda_0$, we have $f(\lambda)=0$.

  Use  Proposition~\ref{prop3} to obtain
  \begin{equation}
    \label{equ:rec6}
    \left\{\begin{array}{rcl}
      \forall                                                       n>0,\,\, \,
      \mathcal R(\lambda_0,\mathcal T^n)&=&\sum_{k=1}^{+\infty}\frac{1}{\lambda_0^k \left |
      \mathcal T^n_k \right |} \\
      \mathcal R(\lambda_0,\mathcal T^{\infty})&=&\sum_{k=1}^{\infty}\frac{1}{ \lambda_0^k\left| \mathcal T^{\infty}_k \right |}
    \end{array}\right.
  \end{equation}

  We write
  $$\mathcal R(\lambda_0,\mathcal T^n)=\sum_{k=1}^{+\infty}\frac{1}{\lambda_0^k\left      |
      \mathcal T^n_k  \right |}=\sum_{k  \leq n}\frac{1}{\lambda_0^k\left  | \mathcal T^n_k
      \right |}+\sum_{k > n}\frac{1}{\lambda_0^k\left  | \mathcal T^n_k \right |}.$$
  On the other hand, for any $k\leq n$ we have $\left  |  \mathcal T^n_k \right  |=\left |  \mathcal T^{\infty}_k
    \right |$, hence
  \begin{equation}
    \label{equ:rec9}
    \mathcal R(\lambda_0,\mathcal T^n)=\sum_{k    \leq     n}\frac{1}{\lambda_0^k\left    |
      \mathcal T^{\infty}_k  \right  |}+\sum_{k  >  n}\frac{1}{\lambda_0^k \left  |
      \mathcal T^n_k \right |}.
  \end{equation}

  Since $f_n$ converges to $f$, then
  \begin{equation}
    \label{equ:rec10}
    \left\{\begin{array}{rcl} \underset{n\rightarrow\infty}{\lim}\mathcal R(\lambda_0,\mathcal T^n)&=&\frac{1}{f(\lambda_0)}<\infty \\
      \mathcal R(\lambda_0,\mathcal T^{\infty})&=&\frac{1}{g(\lambda_0)}        <
      \frac{1}{f(\lambda_0)}
    \end{array}\right.
  \end{equation}

  By using~\eqref{equ:rec9} and~\eqref{equ:rec10}, we obtain
  \begin{equation}
    \label{equ:rec11}
    \lim_{n\rightarrow  +\infty}\sum_{k  > n}\frac{1}{\lambda_0^k \left  |
      \mathcal T^n_k                                                     \right|}=\frac{1}{f(\lambda_0)}-\frac{1}{g(\lambda_0)}>0.
  \end{equation}
  Now    we    take    $\lambda   < \lambda_0$   and    we    apply    the
  Proposition~\ref{prop3} in order to get
  \begin{equation}
    \label{equ:rec12}
    \mathcal R\left(\lambda,\mathcal T^n\right)=\sum_{k=0}^{+\infty}\frac{1}{\lambda^k \left  |
      \mathcal T^n_k  \right |}>\sum_{k>n}\frac{1}{\lambda^k \left  | \mathcal T^n_k  \right|}\geq {\left(\frac{\lambda_0}{\lambda}\right)}^n\sum_{k>n}\frac{1}{\lambda_0^k \left |
      \mathcal T^n_k                        \right                       |}.
  \end{equation}
  We combine~\eqref{equ:rec11} and~\eqref{equ:rec12} to obtain:
  \begin{equation}
    \label{equ:rec13}
    \underset{n\rightarrow\infty}{\lim}\mathcal{R}\left(\lambda, \mathcal{T}^n\right)=\infty
  \end{equation}
  This implies that
  $f\left(\lambda\right)=\underset{n\rightarrow\infty}{\lim}f_n\left(\lambda\right)=\underset{n\rightarrow\infty}{\lim}\frac{1}{\mathcal R\left(\lambda,\mathcal T^n\right)}=0$. Therefore, we proved that:
  $$\forall \lambda <\lambda_0, f(\lambda)=0.$$

  Since $f \neq 0$, we can define
  $\lambda_c:=\inf\left  \{ 0\leq\lambda  \leq 1  : f(\lambda)>0\right\}$:  we  proved that
  \begin{equation}
    \label{equ:rec13bis}
    \forall  \lambda   >\lambda_c,  f(\lambda)=g(\lambda).
  \end{equation}
  As  the
  sequence ${(f_n)}_n$ converges  uniformly to $f$, then  $f$ is continuous,
  and  hence $f(\lambda_c)=0$.  By Lemma~\ref{lemma2}, $g$  is right
  continuous, so we get
  \begin{equation}
    \label{equ:rec14}
    f(\lambda_c)=\lim_{\lambda                             \rightarrow
      \lambda_c^+}f(\lambda)=\lim_{\lambda                     \rightarrow
      \lambda_c^+}g(\lambda)=g(\lambda_c)=0.
  \end{equation}
  On the other hand, by Theorem~\ref{RL}, $g$  is a non-decreasing  function,   hence:
  \begin{equation}
    \label{equ:rec15}
    \forall \lambda <\lambda_c, g(\lambda)=0=f(\lambda).
  \end{equation}
  Combine~\eqref{equ:rec13bis},~\eqref{equ:rec14} and~\eqref{equ:rec15} to obtain the identity $f=g$, thus concluding the proof.
\end{proof}

\color{black}

\section{Self-avoiding walks}
\label{section4}
The main goal of this section is to prove Theorem~\ref{proposition1.1} (Section~\ref{section4.3}) and Theorem~\ref{thm:continuousofselfavoidingtree} (Section~\ref{section4.5}).

\subsection{Walks and bridges}

In this section, we review some definitions on the self-avoiding walk, bridges
and connective constant (see~\cite{madras2013self}). Denote by $c_n$ the number
of self-avoiding  walks of length $n$, starting at the origin on  the considered
graph. If $\mathcal G$  is transitive,  the sequence \smash{$c_n^{1/n}$}
converges to a constant when $n$ goes to infinity. This constant is called the
connective constant of $\mathcal G$.
\begin{definition}
  An \emph{$n$-step bridge} in the plane $\mathbb{Z}^2$ (or upper half-plane
  $\mathbb{H}$)  is an  $n$-step self-avoiding  walk ($SAW$)  $\gamma$
  such that
  \[   \forall   i=1,2,\ldots,n,   \quad   \gamma_1(0)<\gamma_1(i)\leq
    \gamma_1(n) \] where ${{\gamma  }_{1}}(i)$ is the first coordinate
  of $\gamma (i)$.  
  Denote by ${{b}_{n}}$ the number of all $n$-step
  bridges with $\gamma (0)=o$. By convention, set ${{b}_{0}}=1$.
\end{definition}

We have ${{b}_{m+n}}\ge {{b}_{m}}\cdot {{b}_{n}}$, hence we can define
\[{{\mu }_{b}}=\underset{n\to +\infty }{\mathop{\lim
    }}\,{{b_n}^{\frac{1}{n}}}=\underset{n}{\mathop{\sup
    }}\,b_{n}^{\frac{1}{n}}. \]
Moreover, ${{b}_{n}}\le \mu _{b}^{n}\le {{\mu }^{n}}$.




\begin{definition}
  Given a bridge $\gamma$ of length $n$,
  $\gamma$  is  called   an  \emph{irreducible  bridge} if it cannot be decomposed into two
  bridges  of  length strictly smaller than
  $n$. It means, we  cannot find $i \in \left [  1,n-1 \right ]$ such
  that $\gamma_{\mid  [0,i]},\gamma_{\mid \left  [ i,n \right  ]}$ are
  two   bridges.  The   set   of   all
  irreducible bridges is denoted by $iSAW$.
\end{definition}

\subsection{Kesten's  measure}
\label{kesten'smeasure}
For this  section, we  refer the reader  to (\cite{KT2},\cite{Lawler}) for a
more precise description. Denote  by $SAW_{\infty}$ the set of all
self-avoiding walks  on  the plane  $\mathbb{Z}^2$ or  half-plane $\mathbb{H}$.
In this section, we review the Kesten measure, which is a probability  measure
on the set of infinite self-avoiding paths in the  half-plane constructed from
finite bridges. Denote by $\mathbb{B}$ (resp. $\mathbb{I}$) the set of bridges
(resp.\ irreducible bridges) starting at the origin.  
Denote by $p_n$ the  number of  irreducible bridges  starting at the origin, of
length  $n$. 

We   define   a   notion  of   concatenation   of   paths.   If $\gamma^1=[
\gamma^1(0),\ldots,\gamma^1(m)]$ and $\gamma^2=[
\gamma^2(0),\ldots,\gamma^2(n)]$ are two    SAWs, we     define $\gamma^1\oplus
\gamma^2$ to  be the  $(m+n)$-step walk  (which is not necessarily
self-avoiding) as
\begin{align*}
  \gamma^1\oplus \gamma^2  := [0, &  \gamma^1(1),          \ldots,     \gamma^1(m), \\
  &\gamma^1(m)+\gamma^2(1)-\gamma^2(0),                         \ldots,
  \gamma^1(m)+\gamma^2(n)-\gamma^2(0) ].
\end{align*}
    Similarly, we  can
define $\gamma^1\oplus \gamma^2\oplus \cdots\oplus \gamma^k$. We begin
with the following equality:
\begin{fact}[Kesten {\cite[Theorem 5]{KT2}}]
  \label{Kes2}
  We have $$\sum_{n=1}^{+\infty}\frac{p_n}{\mu^n}=1.$$
\end{fact}

\begin{remark}
  \label{R1}
  An obvious consequence of this equality and the existence of arbitrarily large bridges is that $\sum_{\omega  \in \mathbb{I}}\beta^{\left  | \omega  \right|}<\infty$ is finite for $\beta<1/\mu$ and infinite for $\beta>1/\mu$.
\end{remark}

Let us now define the Kesten measure on  the $SAW_{\infty}$ in the
half-plane.   We    fix   $\beta   \leq   \frac{1}{\mu}$    and   denote by
${Q^{\beta}}$ the  probability measure  on $\mathbb{I}$
defined by
$${Q^{\beta}}(\omega)=\frac{\beta^{\left   |    \omega   \right|}}{Z_{\beta}}, \omega \in \mathbb{I}$$
where
$Z_{\beta}=\sum_{\omega  \in \mathbb{I}}\beta^{\left  | \omega  \right|}$. By Fact~\ref{Kes2} and Remark~\ref{R1}, $Z_{\beta}$ is
finite  and  thus ${Q^{\beta}}$  is  a  probability measure  on
$\mathbb{I}$.

Let $k \geq 1$, we consider the product space $\mathbb{I}^k$ and define
the product probability  measure ${Q}^{\beta}_{k}$ on $k$-tuples of bridges; we also write ${Q}^{\beta}_{k}$ for the measure on finite self-avoiding paths defined as
\[Q_k^\beta (\gamma) = \begin{cases}
  {Q}^{\beta}_{k} (\omega_1,\ldots,\omega_k) & \text {if $\gamma = \omega^1\oplus \omega^2\oplus \cdots\oplus \omega^k$, ${(\omega_i)}_{i \leq k} \in \mathbb I^k$} \\
  0 & \text{otherwise.}
\end{cases}\]
The measure $Q^\beta_\infty$ on the set of infinite self-avoiding paths is defined in the same way as the product measure on the bridge decomposition, or equivalently as the projective limit of the $Q_k^\beta$. From the definition, we directly obtain the following property:

\begin{fact}
  \label{p11}
  Under the  $\beta$-Kesten measure, the infinite  self-avoiding walk,
  denoted by $\omega^{\infty,\beta}_{K}$, almost surely does not reach
  the line $\mathbb{Z}\times
    \{0\}$. 
\end{fact}





\subsection{Proof of Theorem~\ref{proposition1.1}}
\label{section4.3}
\begin{notation}\label{not:proofofprop1.1}
  Consider  the  self-avoiding  walks   in  the  lattice  $\mathbb{Z}^2$
  starting at  the origin. We  construct a tree  $\mathcal T_{\mathbb{Z}^2}$, which is called \emph{self-avoiding tree}, from
  these self-avoiding walks: The  vertices of $\mathcal T_{\mathbb{Z}^2}$ are the
  finite self-avoiding walks and two  such vertices joined when one path
  is  an  extension by  one  step  of  the  other. Formally,  denote  by
  $\Omega_n$ the  set of self-avoiding  walks of length $n$  starting at
  the  origin and  $V:=  \bigcup_{n=0}^{+\infty}\Omega_n$. Two  elements
  $x,y \in V$  are adjacent if one  path is an extension by  one step of
  the other. We then define $\mathcal T_{\mathbb{Z}^2}=(V,E)$. In the same way, we can define other self-avoiding trees $\mathcal T_{\mathbb{H}}, T_{\mathbb{Q}}$, where $\mathbb{H}$
  is a half-plane and $\mathbb{Q}$ is a quarter-plane.
\end{notation}

\begin{remark}
  Note that each vertex (resp.\ a ray) of $\mathcal{T}_{\mathbb{Z}^2}$ (or
  $\mathcal{T}_{\mathbb{H}}$, $\mathcal{T}_{\mathbb{Q}}$) is a finite
  self-avoiding path (resp.\ an infinite self-avoiding path). Moreover, it is
  easy to see that the number of vertices at generation $n$ of
  $\mathcal{T}_{\mathbb{Z}^2}$ (or $\mathcal{T}_{\mathbb{H}}$,
  $\mathcal{T}_{\mathbb{Q}}$) is the number of self-avoiding walks of length $n$
  in $\mathbb{Z}^2$ (resp. $\mathbb{H}$, $\mathbb{Q}$).
\end{remark}

\begin{notation}
  \label{subsubsection:notation}  In~\cite{KT2},  Kesten proved  that all
  bridges  in  a  half-plane  can  be  decomposed  into  a  sequence  of
  irreducible bridges in a unique way. 
  An infinite  self-avoiding path starting at the origin is called
  ``$m$-good'' if it possesses  a decomposition into irreducible bridges
  of length at most $m$.  Denote by  $G_m$ the set  of infinite  self-avoiding paths
  which  are  $m$-good, and  let  $\mathcal  T^m$  be  the  subtree  of
  $\mathcal{T}_{\mathbb{Z}^2}$,  which   we  will  refer  to   as  the
  \emph{$m$-good             tree},              spanning the vertex set
  $$V(\mathcal{T}^m):=\{\omega\in V(\mathcal{T}_{\mathbb Z^2}): \text{ there exists } \gamma\in G_m \text{ such that }  \gamma\!\!\mid_{[0, |\omega|]}=\omega\}.$$
\end{notation}

\begin{proposition}
  \label{prop10}
  Let $\mathcal T_{\mathbb{H}}$ and $\mathcal T_\mathbb{Q}$ be  defined as
  above. Then,
  \[gr(\mathcal T_{\mathbb Z^2})=br(\mathcal{T}_{\mathbb{Z}^2})= gr(\mathcal
    T_{\mathbb{H}})  =  br(\mathcal T_{\mathbb{H}}) =  gr(\mathcal
    T_{\mathbb{Q}})  = br(\mathcal T_{\mathbb{Q}})=\mu,\] where  $\mu$ is the
    connective constant of the lattice $\mathbb{Z}^2$.
\end{proposition}

\begin{proof}
  As explained in the introduction, there are rather  large classes of trees, including
  $\mathcal T_{\mathbb Z^2}$,  for which the  branching and growth  coincide (for
  instance, this holds for sub-  or super-periodic trees, cf.\ below, or
  for  typical  supercritical  Galton-Watson  trees), but  none  of  the
  classical results seem to apply to $\mathcal T_{\mathbb H}$ or $\mathcal T_{\mathbb Q}$.

  Note that $\mathcal{T}_{\mathbb{Z}^2}$ is a sub-periodic tree. By Theorem~\ref{sousperiodic} and the definition of connective constant, we have
  \begin{equation}
    \label{equ:recprop11}
    gr(\mathcal{T}_{\mathbb{Z}^2})=br(\mathcal{T}_{\mathbb{Z}^2})=\mu.
  \end{equation}
  We also know (see~\cite{DC1,hammersley1962}) that there exists a
  constant    $B$    and    $n_0    \in    \mathbb{N}$    such    that for any $ n>n_0$,
  $  c_n\leq  b_n\,  e^{B\sqrt{n}}$
  from which we obtain
  \begin{equation}
    \label{equ:recprop00}
    \mu\leq \underset{n\rightarrow \infty}{\lim}{(b_n)}^{\frac{1}{n}}\leq gr(\mathcal{T}_{\mathbb H})\leq gr(\mathcal{T}_{\mathbb Z^2})=\mu.
  \end{equation}
  Hence,
$    gr(\mathcal T_{\mathbb{H}})=\mu$ (as already mentioned in~\cite{KT2})
and, by Proposition~\ref{prop:br-gr},
  \begin{equation}
    \label{equ:recprop10}
    br(\mathcal{T}_\mathbb H)\leq \mu.
  \end{equation}

  Let  $b_{n}^{(m)}$ be the number  of bridges of  length  $n$ which  possess  a
  decomposition  into irreducible bridges of length at most $m$.  Recall that
  ${(\mathcal{T}^m)}_n$ is the number of vertices of $\mathcal{T}^m$ at generation
  $n$. Then for any $n>0$, we have
  \begin{equation}
    \label{equ:recprop14}
    \left | {(\mathcal T^m)}_n \right |\geq b^{(m)}_n.
  \end{equation}
  Note that $\mathcal T^{m}$ is also a subtree of $\mathcal T_{\mathbb{H}}$, so that by Remark~\ref{rem:comparebranchingnumber} we have
  \begin{equation}
    \label{equ:recprop15}
    br(\mathcal T^m)\leq     br(\mathcal T_{\mathbb{H}}).
  \end{equation}
  On the other hand, $\mathcal T^{m}$ is $m$-super-periodic (because from any of its vertices, one can complete the current irreducible bridge in at most $m$ steps after which every self-avoiding path in $\mathcal T^m$ provides a possible continuation),  so we  can apply  Theorem~\ref{sousperiodic} to
  obtain the existence of $gr(\mathcal T^{m})$ and the equality
  \begin{equation}
    \label{equ:recprop16}
    br(\mathcal T^{m})=gr(\mathcal T^{m}).
  \end{equation}
  We use~\eqref{equ:recprop15} and~\eqref{equ:recprop16} to obtain, for any $m>0$,
  \begin{equation}
    \label{equ:recprop17}
    br(\mathcal T_{\mathbb{H}})\geq gr(\mathcal T^{m}).
  \end{equation}

  It            remains            to            prove            that
  $\lim_{m\rightarrow  \infty}  gr(\mathcal T^{m})=\mu$. Noting that the  concatenation of two
  bridges is itself a bridge, we see that the sequence $(b_n)$ is super-multiplicative: for any $m,n$,
  \begin{equation}
    \label{equ:recprop18}
    b_{m+n}\geq b_m\, b_n\quad\text{and} \quad b_{n_1+n_2}^{(m)}\geq       b_{n_1}^{(m)} \,b_{n_2}^{(m)},
  \end{equation}
  implying the existence of
  \begin{equation}
    \label{equ:recprop000}
    \mu_m:=\underset{n\rightarrow\infty}{\lim}{\left(b^{(m)}_n\right)}^{1/n}=\underset{n\rightarrow\infty}{\sup}{\left(b^{(m)}_n\right)}^{1/n}.
  \end{equation}
  Fix
  $\epsilon > 0$: by~\eqref{equ:recprop18} there exists $m_0$ such that for all $m\geq m_0$,
  \begin{equation}
    \label{equ:recprop19}
    \left  | \mu  -{(b_{m})}^{\frac{1}{m}} \right  |\leq
    \epsilon.
  \end{equation}
  As we already mentioned (see  Notation~\ref{subsubsection:notation}), all bridges
  in a  half-plane can  be decomposed into  a sequence  of irreducible
  bridges in  a unique way. Therefore  each bridge of
  length  $m$ possesses  a decomposition  into irreducible  bridges of length at most $m$. Hence, for any $m>m_0$,
  \begin{equation}
    \label{equ:recprop20}
    b_{m}=b_{m}^{(m)}.
  \end{equation}
  Combining~\eqref{equ:recprop18},~\eqref{equ:recprop000},~\eqref{equ:recprop19}  and~\eqref{equ:recprop20}, we obtain that for any $m>m_0$,
  \begin{equation}
    \label{equ:recprop21}
    \mu_m\geq {(b_{km}^{(m)})}^{\frac{1}{km}}\geq {\left({(b_m^{(m)})}^k\right)}^{\frac{1}{km}}={(b_{m}^{(m)})}^{\frac{1}{m}}={(b_m)}^{\frac{1}{m}}\geq \mu-\epsilon.
  \end{equation}
  By~\eqref{equ:recprop000}, the  sequence  ${(b_{\ell}^{(m)})}^{1/\ell}$  converges to $\mu _{m}$, hence
  $\lim_{k\to\infty} {(b_{km}^{(m)})}^{{1}/{km}}=\mu _{m}$. Using~\eqref{equ:recprop14} and~\eqref{equ:recprop21}, for any $m>m_0$, we have $ \mu\geq gr(\mathcal T^{m})\geq \mu_m\geq \mu -\epsilon$ and then,
  \begin{equation}
    \label{equ:recprop22}
    \underset{m\rightarrow\infty}{\lim}gr(\mathcal T^{m})=\mu.
  \end{equation}
  Combining~\eqref{equ:recprop10},~\eqref{equ:recprop17} and~\eqref{equ:recprop22} leads to $br(\mathcal{T}_{\mathbb{H}})=\mu$. By following a similar strategy in $\mathbb Q$, we obtain $gr(\mathcal T_{\mathbb{Q}})=br(\mathcal T_{\mathbb{Q}})=\mu$.
\end{proof}

Theorem~\ref{proposition1.1} is a consequence of Theorem~\ref{lyons90} and Proposition~\ref{prop10}.

\subsection{Proof of Theorem~\ref{thm:continuousofselfavoidingtree}}
\label{section4.5}
Now,   we   apply  the   results  in Section~\ref{sectiononcontinuous} for   the   self-avoiding   trees $\mathcal{T}_{\mathbb{Q}}$,
$\mathcal{T}_{\mathbb{H}}$ and $\mathcal T_{\mathbb{Z}^2}$.

\subsubsection*{\bf Notation} For any $n\in \mathbb{N}$, let $\Lambda_n:={[\![-n,n]\!]}^2$ be a subdomain of $\mathbb{Z}^2$. Denote by $\partial \Lambda_n$ the boundary of $\Lambda_n$, i.e,
$$\partial \Lambda_n:=\left\{(a,b)\in \Lambda_n: |a|=n \, \text{ or }\, |b|=n\right\}.$$
We write $\overset{\circ}{\Lambda}_n:=\Lambda_n \setminus \partial\Lambda_n$ for the interior of $\Lambda_n$.
Let $\gamma$ be a finite self-avoiding walk: we say that $\gamma$ is \emph{a self-avoiding walk in the domain $\Lambda_n$} if for any $0\leq k\leq |\gamma|$, we have $\gamma(k)\in \Lambda_n$. Denote by $\Omega(\Lambda_n)$ the set of self-avoiding in $\Lambda_n$ starting from the origin $o=(0,0)$.
\begin{lemma}
  \label{lem:rightcontinuous}
  The         functions        $\mathcal C(\lambda,\mathcal T_{\mathbb{Q}})$, $\mathcal C(\lambda,\mathcal T_{\mathbb{H}})$ and  $\mathcal C(\lambda,\mathcal T_{\mathbb{Z}^2})$      are        right continuous        on
  $(\lambda_c,+\infty)$.
\end{lemma}

\begin{proof}
  It follows immediately from Lemma~\ref{lemma2}.
\end{proof}

\begin{lemma}
  \label{lem:leftcontinuous}
  The         functions        $\mathcal C(\lambda,\mathcal T_{\mathbb{Q}})$, $\mathcal C(\lambda,\mathcal T_{\mathbb{H}})$ and  $\mathcal C(\lambda,\mathcal T_{\mathbb{Z}^2})$      are left continuous        on
  $(\lambda_c,+\infty)$.
\end{lemma}

\begin{figure}[ht!]
  \centering
  \includegraphics[scale=0.6]{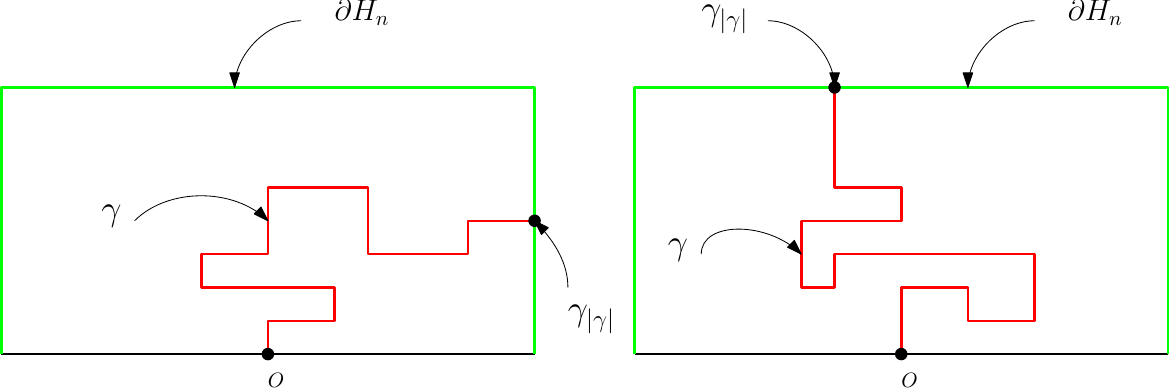}
  \caption{The boundary of $\mathbb{H}_n$ is green and the self-avoiding walk $\gamma$ is red. Recall that $\gamma$ is a vertex of the tree $\mathcal{T}_\mathbb{H}$. On the left (resp.\ right), we can add a  new quadrant $\mathbb{Q}$ (resp.\ new half-plane $\mathbb{H}$) rooted at $\gamma_{|\gamma|}$. Hence, on the left (resp.\ on the right) the subtree ${\left(\mathcal{T}_\mathbb H\right)}^\gamma$ contains the tree $\mathcal{T}_\mathbb{Q}$ (resp. $\mathcal{T}_\mathbb{H}$). }
  \label{fig:cutsets}
\end{figure}
\begin{proof}
  We prove this Lemma for the case $\mathcal{T}_\mathbb{H}$ and we use the same argument for other cases ($\mathcal{T}_\mathbb{Q}$ and $\mathcal{T}_{\mathbb{Z}^2}$). Note that $\mathcal{T}_\mathbb{H}$ is not uniformly transient, therefore we cannot use Theorem~\ref{ctg}. Fortunately, we can prove that $\mathcal{T}_\mathbb{H}$ is weakly uniformly transient. For this purpose, we define a sequence of cutsets $(\pi_n, n\geq 1)$ as follows. Set $\mathbb{H}_n:=\Lambda_n\bigcap \mathbb{H}$ and $\partial \mathbb{H}_n:=\left(\partial \Lambda_n\right)\bigcap \mathbb{H}$ (see Figure~\ref{fig:cutsets}). Recall that $\Omega\left(\mathbb{H}_n\right)$ is the set of self-avoiding walks of domain $\mathbb{H}_n$. For any $n\geq 1$,
  $$\pi_n :=\left \{  \gamma\in \Omega\left(\mathbb{H}_n\right):  \text{ for any
    } 0\leq k< \left   |    \gamma   \right   |   ,   \gamma(k)\in
    \overset{\circ}{\mathbb{H}}_n \text{ and } \gamma_{|\gamma|}\in
    \partial(\mathbb{H}_n)\right \}$$ Since $\mathbb{H}_n$ is a finite domain of
    $\mathbb{H}$, therefore any infinite self-avoiding walk starting at the origin
    of $\mathbb{H}$, must touch the boundary of $\mathbb{H}_n$. Hence, for any
    $n\geq 1$, we have $\pi_n$ is a V-cutset of $\mathcal{T}_\mathbb{H}$.  Let
    $\Gamma:=\underset{n\geq 1}{\bigcup}\pi_n$: it remains to verify that
  \begin{equation}
    \label{equ:recrec1}
    \forall  \lambda>\lambda_{c}(=\frac{1}{\mu}),\exists
    \alpha_{\lambda}>0,\forall    \gamma    \in    \Gamma,\mathbb{P}(\forall
    n>0,X^{\gamma}_{n}\neq \gamma)\geq \alpha_{\lambda}.
  \end{equation}
  Note that for any $\gamma\in \Gamma$, the subtree
  ${\left(\mathcal{T}_\mathbb{H}\right)}^\gamma$ contains  the tree $\mathcal
  T_{\mathbb{H}}$ or $\mathcal T_{\mathbb{Q}}$ (see Figure~\ref{fig:cutsets}).
  Hence,~\eqref{equ:recrec1} is a consequence of  Theorem~\ref{proposition1.1}
  and Theorem~\ref{RL}. We use Theorem~\ref{ctgg} to complete the proof of
  Lemma.
\end{proof}

Theorem~\ref{theorem1.4} is a consequence of Lemmas~\ref{lem:rightcontinuous}
and~\ref{lem:leftcontinuous}.

\section{The biased walk on the self-avoiding tree}
\label{newmeasure}

We now begin  the study of our  main object of interest,  which is the
biased random walk on the self-avoiding  tree. We will use the results
that were obtained  in the  previous section  to prove the properties of  the limit
walk. In the next section, we will gather a few natural conjectures.

\subsection{The limit walk}
\label{section6.1}

Let $\lambda  \in \left [0,+\infty  \right ]$ and consider  the biased
random     walk    $RW_{\lambda}$     on     $\mathcal    T$     where
$\mathcal{T}=\mathcal                T_{\mathbb{H}}$                or
$\mathcal{T}=\mathcal T_{\mathbb{Z}^2}$.  For $\lambda>\lambda_c$, the
biased random walk is transient so almost surely, the random walk does
not visit $\mathcal  T_k$ anymore after a sufficiently  large time. We
can    then    define    the     limit    walk,    as    denoted    by
$\omega_{\lambda}^{\infty}$ in the following way:
$$\omega_{\lambda}^{\infty}(i)=x_i                     \iff
  \left\{\begin{matrix}
    x_i \in \mathcal T_i \\
    \exists n_0, \forall n>n_0: X_n \in \mathcal T^{x_i}
  \end{matrix}\right \}.
$$
$\omega_{\lambda}^{\infty}$     is      a     random      ray.    Denote by
$\mathbb{P}_{\lambda}^{\mathbb{H}}$    the     law     of
$\omega_{\lambda}^{\infty}$   in  the   half-plane  $\mathbb{H}$   and
$\mathbb{P}_{\lambda}^{\mathbb{Z}^2}$,        the        law        of
$\omega_{\lambda}^{\infty}$ in  the plane  $\mathbb{Z}^2$. We  can see
$\mathbb{P}_{\lambda}^{\mathbb{H}}$                      (respectively
$\mathbb{P}_{\lambda}^{\mathbb{Z}^2}$)  as  a probability  measure  on
$SAW_{\infty}$ in the half-plane (respectively the plane).

For what follows, it will be  useful to have the following definition:
removing all the  finite branches of $\mathcal T_{R}$ (where  $R$ is a
regular lattice),  leads to  a new  tree without  leaf, which  we will
denote by $\widetilde{\mathcal T}_{R}$.

\subsection{The case \texorpdfstring{$\lambda=+\infty$}{λ=infinity} and percolation}
\label{sec:percolation}

First, we  review some definitions of  percolation theory. Percolation
was   introduced   by   Broadbent   and  Hammersley   in   1957   (see~\cite{broadbent1957percolation}). For $p \in \left [ 0,1 \right ]$, we
consider the  triangular lattice $\mathbb{T}$, a  site of $\mathbb{T}$
is  open  with  probability  $p$ or  closed  with  probability  $1-p$,
independently  of the  others.  This  can also  be  seen  as a  random
colouring  (in black  or  white)  of the  faces  of hexagonal  lattice
$\mathbb{T}^{*}$ dual of $\mathbb{T}$.

We  define  the  exploration  curve as  follows  (see~\cite{WN:book},
section 6.1.2  for more  details). Let $\Omega$  be a  simply connected
subgraph of the  triangular lattice and $A$, $B$ be  two points on its
boundary. We can then divide  the hexagonal cells of $\partial \Omega$
into two arcs, going from $A$  to $B$ in two directions (clockwise and
counter-clockwise).  These arcs  will be  denoted by  $\mathbb{B}$ and
$\mathbb{W}$ such that $A,\mathbb{B},B,\mathbb{W}$ is in the clockwise
direction. Assume that all of the hexagons in $B$ are colored in black
and that all of the hexagons in $\mathbb{W}$ are colored in white. The
color of  the hexagonal faces in  $\Omega$ is chosen at  random (black
with probability $p$ and  white with probability $1-p$), independently
of  the  others.  We  define  the  \emph{exploration  curve}  $\gamma$
starting at $A$ and ending at  $B$ which separates the black component
containing   $\mathbb{B}$   from   the  white   component   containing
$\mathbb{W}$.

Then the exploration curve $\gamma$  is a self-avoiding walk using the
vertices and edges of hexagonal  lattice $\mathbb{T}^*$. We can define
this interface  $\gamma$ in  an equivalent, dynamical  way, informally
described  as follows.  At  each  step, $\gamma$  looks  at its  three
neighbors on  the hexagonal lattice, one  of which is occupied  by the
previous  step  of $\gamma$.  For  the  next step,  $\gamma$  randomly
chooses one of these neighbors that  has not yet occupied by $\gamma$.
If there is just one neighbor that  has not yet been occupied, then we
choose this  neighbor and if there  are two neighbors, then  we choose
the right  neighbor with  probability $p$ and  the left  neighbor with
probability $1-p$.

We know that there exists $p_c \in \left [ 0,1 \right ]$ such that for
$p<p_c$ there is almost surely  no infinite cluster, while for $p>p_c$
there is almost  surely an infinite cluster. This  parameter is called
\emph{critical  point}.  It  is  known  that  the  critical  point  of
site-percolation on  the triangular lattice equals  $\frac{1}{2}$. The
lower   bound   of   critical   point  was   proven   by   Harris   in~\cite{harris1960lower}.  A  similar  theorem   in  the  case  of  bond
percolation   on    square   lattice   was   given    by   Kesten   in~\cite{kesten1980critical}, and the result on the triangular lattice is
obtained in a similar fashion.

Now,  take  $\Omega=\mathbb{T}^*_{+}$,  the  half-plane  of  hexagonal
lattice. The hexagons on the  boundary of $\Omega$ ($\partial \Omega$)
and on  the right of  origin (denoted  by $\partial ^{+}  \Omega$) are
colored in black and the hexagons on $\partial \Omega$ and on the left
of origin ($\partial^{-} \Omega$) are  colored in white. In this case,
the  exploration curve  is  an (random)  infinite self-avoiding  walk.
Denote  by   $\mathcal  T_{\mathbb{T}^*_+}$  the   self-avoiding  tree
constructed from the self-avoiding walks in $\mathbb{T}^*_{+}$.

In  the  case  $\lambda=+\infty$,   one  can  reinterpret  the  second
construction   of   the   exploration   curve  as   the   limit   walk
$\omega^{\infty}$ on $\widetilde{\mathcal T}_{\mathbb T^*_+}$. This is
very useful  because every feature of  the curve $\gamma$ is  also one
for $\omega^{\infty}$  and can therefore  be restated in terms  of the
biased walk on the self-avoiding tree. One of these properties is that
$\gamma$ almost  surely reaches the  boundary of $\Omega$  infinitely
many times, which  follows from  Russo-Seymour-Welsh type arguments.  As we
will  see   below,  this   property  is  still   valid  in   the  case
$RW_{\lambda}$,       for      all       $\lambda>\lambda_c$      (see
Theorem~\ref{thm:propertieoflimitwalk}).

\subsection{Proof of Theorem~\ref{thm:propertieoflimitwalk}}
\label{sec:plambda}

In this section, for any $z\in  \mathbb{Z}^2$, we write $\Re z$ (resp.
$\Im z$) for the real part (resp.\ imaginary part) of $z$. Before going
into  the  proof, we  introduce  a  map from  the  set  of all  finite
self-avoiding walks on  a graph to the set of  vertices of that graph,
which to  each walk associates  the location  of its last  vertex (the
``head        of        the       snake''):        formally,        if
$\omega  = (\omega_0,\ldots,  \omega_{|\omega|})$  is a  self-avoiding
path, its  head is the  vertex $\omega_{|\omega|}$. In our  setup, the
underlying graph is  either $\mathbb H$ or $\mathbb Z^2$,  so the walk
itself     is    a     vertex     of     the    corresponding     tree
($\mathcal      T      =      \mathcal     T_{\mathbb      H}$      or
$\mathcal T =  \mathcal T_{\mathbb Z^2}$ respectively),  and we define
$p : V(\mathcal T) \to \mathbb Z^2$ by letting
\[p(\omega) := \omega_{|\omega|}.\]
The image  by $p$ of a  random walk on  $\mathcal T$ is thus  a random
process on $\mathbb Z^2$, which we are going to investigate.

The proof of Theorem~\ref{thm:propertieoflimitwalk} has several steps.
In the first  step, we study the trajectory of  the biased random walk
$X_n$   on  $\mathcal   T$.  We   prove  that,   under  the   measures
$\mathbb{P}^{\mathbb           H}_{            \lambda}$           and
$\mathbb{P}^{\mathbb Z^2}_{\lambda}$,  $p(X_n)$ almost  surely reaches
the  line  $\mathbb{Z}\times\{0\}$.  In  the  second  step,  we  prove
furthermore    that    it    almost   surely    reaches    the    line
$\mathbb{Z}\times\{0\}$  an infinite  number  of times.  In the  third
step, we  prove that  under $\mathbb{P}^{\mathbb  Z^2}_{\lambda}$, the
limit walk  (\emph{i.e.}, the infinite self-avoiding  path obtained as
the limit of  $(X_n)$ as $n\to\infty$) almost surely  touches the line
$\mathbb{Z}\times \{0\}$  an infinite  number of  times ---  note that
this is  not a deterministic  consequence of the previous  step, since
there are trajectories for $(X_n)$  which do touch the line infinitely
many times but whose  limit path does not. In the  last step, we prove
the corresponding result for the  half-plane, \emph{i.e.} we show that
under $\mathbb{P}^{\mathbb H}_{\lambda}$, the limit walk almost surely
touches the line $\mathbb{Z}\times \{0\}$ an infinite number of times.
For simplicity, we will write $Y_n$ for $p(X_n)$.

\subsubsection{\bf The first step}

In this step, we study the trajectory of $RW_{\lambda}$. We begin with
the following simple lemma:

\begin{lemma}
  \label{3.2}
  Let $\lambda>\lambda_c$ and consider  the biased random walk $(X_n)$
  of   law   $RW_{\lambda}$    on   $\mathcal   T_{\mathbb{Z}^2}$   or
  $\mathcal       T_{\mathbb{H}}$.       Then      almost       surely
  $\limsup \left | \Re(p(X_n)) \right |=+\infty$.
\end{lemma}

\begin{figure}[ht!]
  \centering
  \includegraphics[scale=0.6]{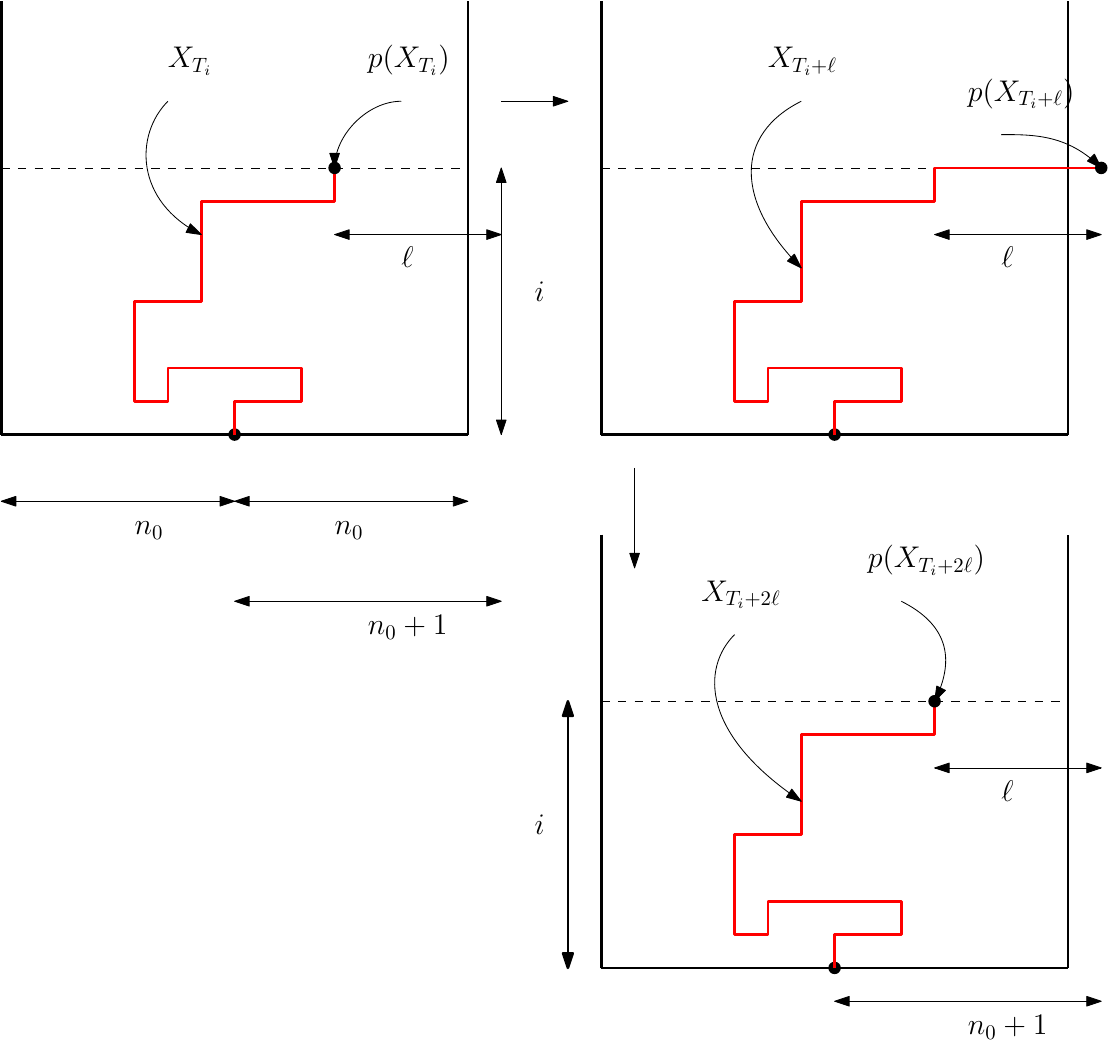}
  \caption{Illustration of the proof of Lemma~\ref{3.2}}
  \label{fig:firststep1}
\end{figure}

\begin{proof}
  Recall that  we defined $Y_n :=  p(X_n)$. We prove the  lemma in the
  case    of    $\mathcal     T_{\mathbb{H}}$;    the    result    for
  $\mathcal T_{\mathbb{Z}^2}$  can be obtained  in a similar  way. The
  idea of the  argument is straightforward: if the real  part of $Y_n$
  is constrained, then its imaginary part has to take large values and
  every time  it visits a  new height, the real  part has a  chance of
  becoming large. What follows is a formalization of this idea. Assume
  that
  $\alpha:=\mathbb{P}(\limsup \left |  \Re(Y_n) \right | <+\infty)>0$,
  then there exists a constant $n_0>0$ such that,
  \begin{equation}
    \label{equ:firststeprec1}
    \beta:= \mathbb{P}\left\{\forall n>0, -n_{0}\leq  \Re(Y_n)\leq n_{0}\right\}>0.
  \end{equation}
  For  any    $i    \geq    0$, define
  \begin{equation}
    \label{equ:firststeprec2}
    T(i):=\inf \left \{ n\geq 0: \Im(Y_n)=i \right \}.
  \end{equation}
  Note       that       $T(i)<+\infty$        on       the       event
  $\left\{\text{for all } n>0: -n_{0}\leq \Re(Y_n)\leq n_{0}\right\}$.
  We remark that,  at time $T(i)$, $X$ can always  be extended towards
  the left or the right. For any $i\geq 0$, define
  $$\scalebox{.9}{$S_{i}:=\left  \{  \exists!  k: \left  |  \Re(Y_k)  \right|=n_{0}+1,  \Im(Y_{k})=i \text{ and }\forall{}   n  \neq{}   k:  -n_{0}   \leq \Re
      (Y_n)\leq n_{0} \right \}$}.$$
  If the walk is at time $T(i)$, then we go   towards  the   left  or   the   right  to   reach  the   domain
  $$\left \{ \Re z=n_{0}+1 \right  \}\bigcup \left \{ \Re z=-n_{0}-1\right\},$$ and after, we go back to $X_{T(i)}$ (see Figure~\ref{fig:firststep1}). We need at most $2n_{0}$
  steps  to do this.  Then,  there exist  a  constant  $c>0$ such  that for any $i>0$,
  \begin{equation}
    \label{equ:firststeprec3}
    \mathbb{P}(S_{i})\geq c \,\beta.
  \end{equation}
  On the other hand, we have
  \begin{equation}
    \label{equ:firststeprec4}
    \bigcup_{i=0}^{+\infty}S_{i} \subset \left \{\forall n\geq 0, -n_0-1\leq  \Re (Y_n)
    \leq n_{0}+1\right \}.
  \end{equation}
  Since these $S_i$ are pairwise disjoint, by using~\eqref{equ:firststeprec3} and~\eqref{equ:firststeprec4} we obtain:
  $$\mathbb{P}\left(\forall n\geq 0, -n_0-1\leq  \Re (Y_n)
    \leq n_{0}+1\right )\geq \sum_{i=0}^{\infty} \mathbb{P}(S_{i})\geq
    \sum_{i=0}^{\infty}c\, \beta=+\infty.$$ This  is a contradiction and
  therefore                        almost                       surely
  $\limsup \left | \Re (Y_n) \right |=+\infty $.
\end{proof}

\begin{lemma}
  \label{theorem11}
  Let $\lambda>\lambda_c$ and consider  the biased random walk $(X_n)$
  of   law   $RW_{\lambda}$    on   $\mathcal   T_{\mathbb{Z}^2}$   or
  $\mathcal                    T_{\mathbb{H}}$.                   Then
  $\#\left \{ n>0:\Im(p(X_n))=0 \right \}\geq 1$ almost surely.
\end{lemma}

\begin{proof}
  Recall again that $Y_n = p(X_n)$ is  the location of the head of the
  walk at time  $n$. The idea of  the proof is somehow  similar to the
  previous  one: given  that  $|\Re(Y_n)|$  reaches arbitrarily  large
  values, we will look at the path at record times of $|\Re(Y_n)|$ and
  argue from the  previous lemma that from each such  record time, the
  walk touches the real axis with positive probability. We again write
  the proof in  $\mathbb Z^2$ in detail, since the  half-plane case is
  almost identical.

  Say that  a path $\omega$ is  \emph{good} if its head  is further to
  the left or right than its other vertices, namely if
  \[|\Re   p(\omega)|   >   \max_{k<|\omega|}   |\Re(\omega_k)|.\]   A
  consequence of  the previous lemma  is that almost surely,  $X_n$ is
  good  infinitely  many  times:  indeed,  it  is  automatically  good
  whenever $|\Re(Y_n)|$ reaches a record value. Now define inductively
  the following stopping times: $T_0 = S_0 = 0$,
  \[T_{k+1} =  \inf \{ n>S_k :  X_n \text{~is good} \}  \text{~and} \]
  \[ S_{k+1} = \inf \{ n>T_{k+1} : \Im (Y_n) = 0 \text{~or~} \Im (Y_n)
    =  2  \Im(Y_{T_{k+1}})\}.\] All  these  stopping  times are  a.s.\
  finite,  the $T_k$  as mentioned  above and  the $S_k$  by the  same
  argument  as in  the proof  of  the previous  lemma, which  excluded
  $(Y_n)$ from remaining within a strip.

  Assume              by               contradiction              that
  $\alpha := \mathbb  P(\forall n>0, \Im Y_n>0)$ is  positive. The key
  observation, which  we will  use several times  in various  forms in
  what follows, is that the behavior of the walk after time $T_k$, and
  until its  first visit to the  parent of $X_{T_k}$ in  $\mathcal T$,
  matches       the       similar       process       defined       in
  $\mathbb Z^2 \setminus \{X_{T_k}(i) : 0 \leq i < |X_{T_k}|\}$. Here,
  this         domain          contains         the         half-plane
  \[\mathbb H_k := \{ z \in \mathbb Z^2 : |\Re(z)| \geq |\Re(Y_{T_k})|
    \text{~and~} \Re(z)  \Re(Y_{T_k}) \geq  0 \}. \]  With conditional
  probability $\lambda/(1+3\lambda)$, the first step of $Y$ after time
  $T_k$  will  be  horizontal  and into  $\mathbb  H_k$,  after  which
  $Y$ will remain within $\mathbb H_k$ with conditional probability at
  least $\alpha$ by our running hypothesis.  If such is the case, then
  by symmetry, $\Im(Y_{S_k}) =  0$ with conditional probability $1/2$.
  To summarize,  \[ \mathbb P(\Im(Y_{S_k})=0 |  \mathcal F_{T_k}) \geq
    \frac {\alpha\lambda} {2(1+3\lambda)} > 0,\] where as is customary
  $\mathcal  F_{T_k}$  denotes  the $\sigma$-field  generated  by  the
  process $(X_n)$ up to the stopping time $T_k$. Since this inequality
  holds  for  all  $k$,  there  a.s.\  exists  some  $k_0$  such  that
  $\Im(Y_{S_{k_0}}) = 0$, which contradicts our assumption.
\end{proof}

As  a shortcut  in  later arguments,  we  will refer  to  the kind  of
construction in the proof above  as \emph{considering a new half-plane
  with origin $Y_{T_k}$}.

\subsubsection{\bf The second step}

The goal of this step is to prove the following lemma:

\begin{lemma}
  \label{theorem12}
  Let  $\lambda>\lambda_c$ and consider the  biased random
  walk  $RW_{\lambda}$  on   $\mathcal T_{\mathbb{Z}^2}$  or  $\mathcal T_{\mathbb{H}}$. Then almost surely $\# \left  \{ n>0:\Im(Y_{n})=0 \right \} = +\infty$.
\end{lemma}

\begin{proof}
  \begin{figure}[ht!]
    \centering
    \includegraphics[scale=0.55]{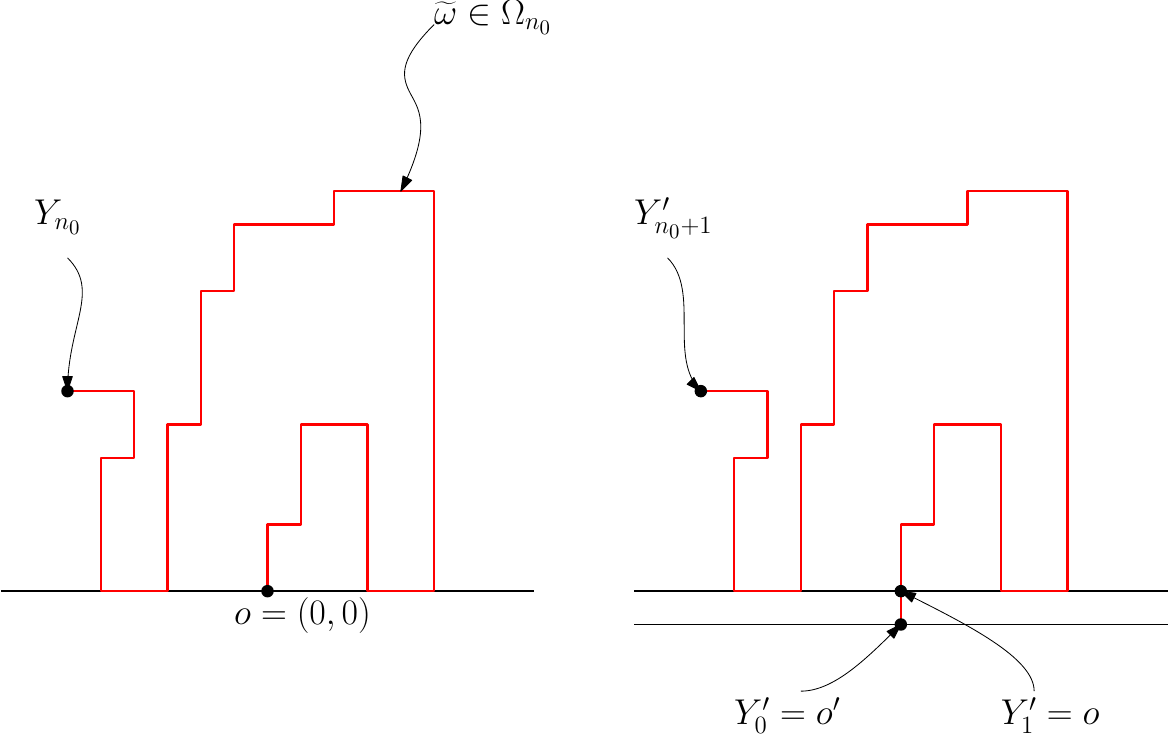}
    \caption{Illustration of the proof of Lemma~\ref{theorem12}, case $\mathcal T_{\mathbb{H}}$}
    \label{fig:SAW5}
  \end{figure}

  We again need to deal separately with the two cases.

  \paragraph{\bf Case {I}: the tree $\mathcal T_{\mathbb{H}}$}

  Denote by $A$ the event that  $(Y_n)$ touches the real line infinitely
  often, namely
  $$A:=\left \{  \# \left  \{ n>0:  \Im Y_n  =0 \right  \}=+\infty \right\}=\left \{ \forall  k, \exists n>k: \Im Y_n=0 \right  \}.$$ Assume by
  contradiction  that $\mathbb{P}(A^c)>0$.  Then,  there exists  $n_0>0$
  such that
  \begin{equation}
    \label{equ:secondstep11}
    \mathbb{P}(\forall n\geq n_0: \Im Y_n >0)>0.
  \end{equation}
  Now,  consider  the  random  walk  until  time  $n_0$  and  denote  by
  $\Omega_{n_0}$ the  set of all  possible trajectories of  $(X_{n})$ in
  the  tree  $\mathcal  T_{\mathbb  H}$   up  to  that  time.  For  each
  $\omega \in \Omega_{n_0}$, define the events
  \begin{equation}
    \label{secondstep10}
    A_{\omega} := \{  {(X_{n})}_{n \leq n_0} = \omega  \} \quad \text{and}
    \quad B_{\omega} := A_\omega \cap \{ \forall
    n\geq n_0, \Im(Y_n)>0 \}.
  \end{equation}
  Since  the  cardinality  of  $\Omega_{n_0}$ is  finite,  there  exists
  $\tilde       \omega      \in       \Omega_{n_0}$      such       that
  $\mathbb{P}(B_{\tilde    \omega})>0$;    fix   such    a    trajectory
  $\tilde\omega =  (\tilde \omega_0, \ldots, \tilde  \omega_{n_0})$ from
  now on.

  We add a new line under the line $\mathbb{Z}\times \{0\}$ and consider
  a  new half-plane  $\mathbb{H}'$ with  origin $O'  = (0,-1)$  (see the
  Figure~\ref{fig:SAW5}   and   the   discussion   in   the   proof   of
  Lemma~\ref{theorem11}).  Observe  an  independent biased  random  walk
  $X'_n$  with parameter  $\lambda$  on  $\mathcal T_{\mathbb{H}'}$  and
  denote by $Y'_n$ its head, $Y'_n := p(X'_n)$. Let
  $$A'_{\tilde \omega}  := \{ \forall n  \leq n_0, X'_{1+n} =  O' \oplus
    \tilde     \omega_n     \}$$     where     as     defined     earlier,
  $O' \oplus \tilde \omega_n$ is  the self-avoiding path in $\mathbb H'$
  obtained by prepending the vertex  $O'$ to the path $\tilde \omega_n$;
  obviously  $P(A'_{\tilde  \omega}) >  0$.  Now,  conditionally on  the
  events $A_{\tilde \omega}$ and $A'_{\tilde \omega}$, the walks $(X_n)$
  and $(X'_n)$ can be  coupled after time $n_0$ in such  a way that they
  agree on the event $B_{\tilde \omega}$; and whenever this is the case,
  the imaginary  part of  $Y'_n$ remains positive  for all  times $n>0$,
  thus  showing   that  this  happens  with   positive  probability  and
  contradicting Lemma~\ref{theorem11} applied to $X'$.

  \bigskip

  \paragraph{\bf Case {II}: the tree $\mathcal T_{\mathbb{Z}^2}$.}

 \begin{figure}[ht!]
   \centering
   \includegraphics[scale=0.55]{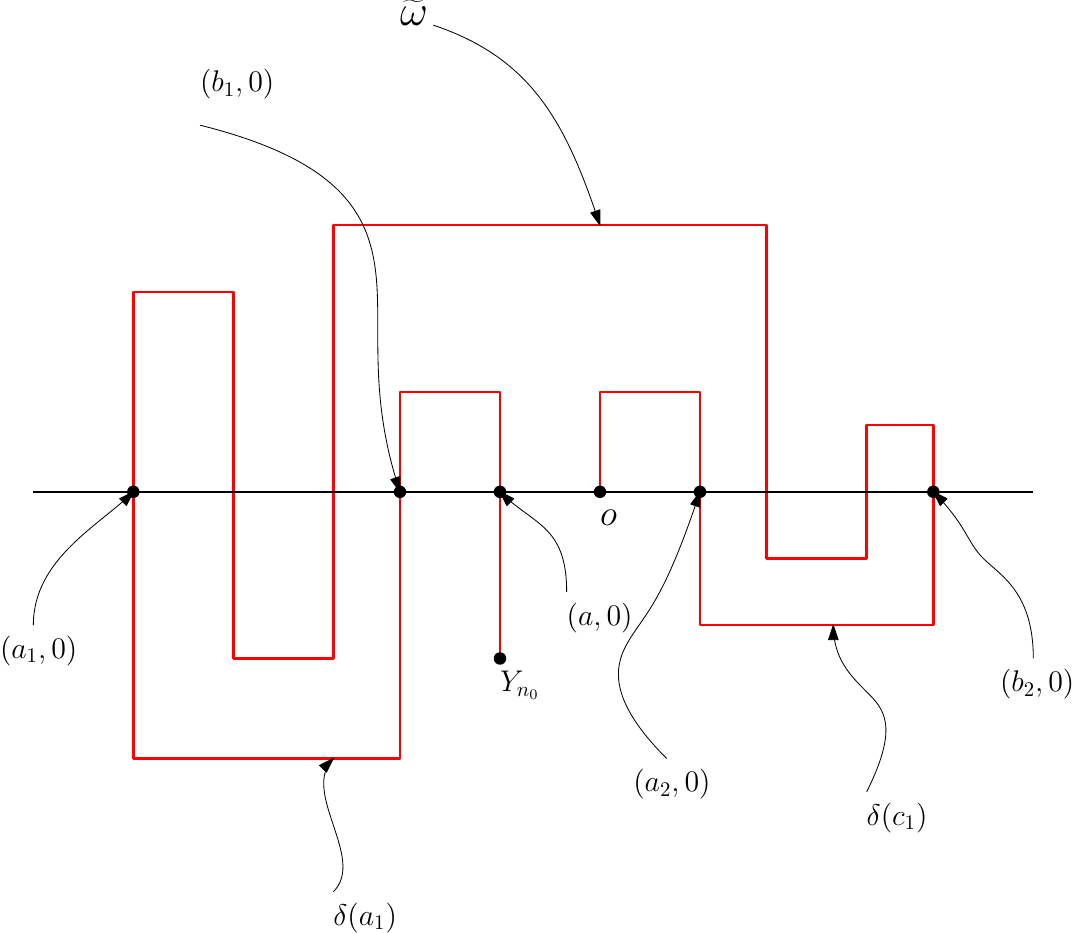}
   \caption{Illustration of the proof of Lemma~\ref{theorem12}, case $\mathcal T_{\mathbb{Z}^2}$}
   \label{fig:SAW6}
 \end{figure} 

  \begin{figure}[ht!]
    \centering
    \includegraphics[scale=0.55]{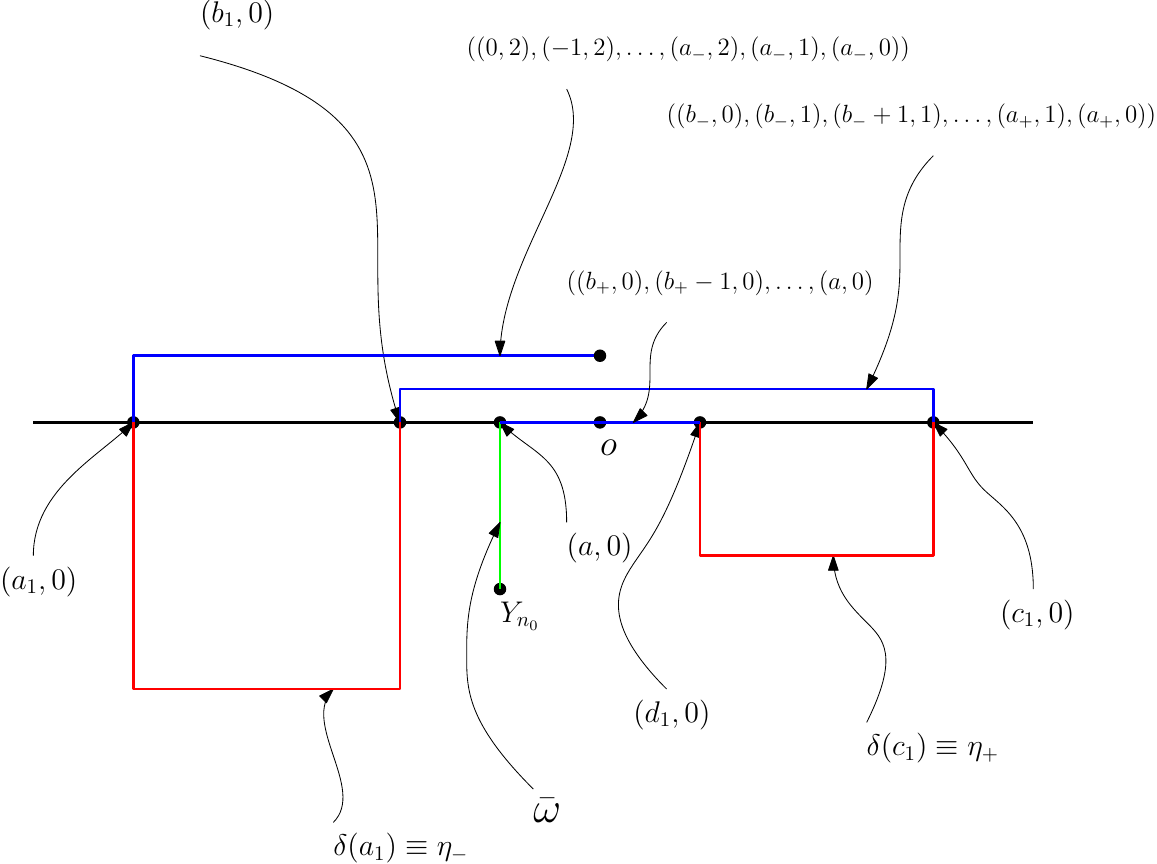}
    \caption{Illustration of the proof of Lemma~\ref{theorem12}, case $\mathcal T_{\mathbb{Z}^2}$}
    \label{fig:secondstep3}
  \end{figure}

  Assume  again  that with  positive  probability,  the process  $(Y_n)$
  reaches the line $\mathbb{Z}\times \{0\}$  finitely many times. By the
  same  argument as  in the  case  of $\mathcal T_{\mathbb{H}}$,  there exists  a
  positive number  $n_0$ and  a trajectory $\tilde  \omega$ in  the tree
  $\mathcal T_{\mathbb Z^2}$ such that the events
  $$A_{\tilde\omega} := \{  {(X_n)}_{n \leq n_0} = \tilde  \omega \} \quad
    \text{and} \quad B_{\tilde\omega} := A_{\tilde \omega} \cap \{ \forall n
    \geq n_0, \Im  Y_n < 0 \}$$ both have  positive probability. The key step of
    the proof is to show  the existence of an integer $h>0$ and of a
    self-avoiding   path    $\eta$    in    the   half-plane    below $\mathbb Z
    \times \{h\}$, starting at $(0,h)$, which has the same head as 
    $\tilde{\omega}_{n_0}$ and whose complement in the lower half-plane has the same
    unbounded connected component as that of $\tilde \omega_{n_0}$ (meaning that
    its intersection with the lower half-plane is the same except for irrelevant
    ``hidden parts''). Indeed,  if such  a path exists,  then a  similar
    coupling argument  as in  the  proof of  the previous  lemma  shows that
    with  positive  probability, the  process $RW_\lambda$ in  the half-plane
    below  $\mathbb Z \times  \{h\}$ first grows along  $\eta$ and then does
    not visit $\mathbb Z  \times \{0\}$ afterwards,  which  contradicts the
    half-plane  case  of the  present lemma.

  It  remains  to  construct  the  path  $\eta$;  here  is  an  informal
  description of one possible  construction: consider all the excursions
  of $\tilde \omega_{n_0}$ below the line $\mathbb Z \times \{0\}$, keep
  the maximal  ones (\emph{i.e.}, those  not separated from  infinity by
  any other)  and ``collect'' them  linearly along the axis,  first from
  the left up to the exit  point of $\tilde \omega_{n_0}$, then from the
  right --- see Figure~\ref{fig:secondstep3}.

  By an \emph{excursion} of $\tilde \omega_{n_0}$, we will always mean a
  sub-path     of     $\tilde      \omega_{n_0}$     of     the     form
  $(\tilde  \omega_{n_0}(s),  \tilde \omega_{n_0}(s+1),  \ldots,  
  \tilde{\omega}_{n_0}(t))$                                                where
  $\Im\tilde  \omega_{n_0}(s)   =  \tilde   \omega_{n_0}(t)  =   0$  and
  $\Im \tilde \omega_{n_0}  (u) <0$ whenever $s<u<t$.  Notice first that
  the head $Y_{n_0}  = p(\tilde \omega_{n_0})$ is not  surrounded by any
  of  the   excursions  of  $\tilde  \omega_{n_0}$,   since  that  would
  contradict the assumption that $P(B_{\tilde \omega}) > 0$. Let $(a,0)$
  be   the   last    point   of   $\mathbb   Z    \times   \{0\}$   that
  $\tilde \omega_{n_0}$  visits, $\bar \omega$  be the path  followed by
  $\tilde  \omega_{n_0}$ after  visiting $(a,0)$,  and define  $\Delta_-$
  (resp.\ $\Delta_+$)  as the set  of all integers $x<a$  (resp.\ $x>a$)
  such that the  vertex $(x,0)$ is visited by  $\tilde \omega_{n_0}$. If
  $(x,0)$  is one  endpoint of  an excursion  of $\tilde  \omega_{n_0}$,
  denote its other  endpoint by $(\varphi(x),0)$ and  let $\delta(x)$ be
  sub-path    of    $\tilde    \omega_{n_0}$   between    $(x,0)$    and
  $(\varphi(x),0)$, parameterized to start at $(x,0)$ (walking backwards
  if needed); otherwise, set $\varphi(x)=x$ and $\delta(x) = \{x\}$.

  Assume   first    that   $\Delta_-    \neq   \emptyset$:    then   let
  $a_1 :=  \min \Delta_-$,  $b_1 :=  \varphi(a_1)$, and  inductively let
  $a_{i+1}   :=   \min   \{   x>b_i   :   x   \in   \Delta_-   \}$   and
  $b_{i+1} = \varphi (a_{i+1})$, as long as the set in the definition of
  $b_{i+1}$   is   nonempty.   This   leads   to   a   finite   sequence
  $a_1\leq b_1<  a_2 \leq b_2  < \cdots<a_{i_0} \leq  b_{i_0}<a$. Define
  $\eta_-$ as the  path obtained by concatenating  the $\delta(a_i)$ and
  the    interval     of    the    form     $[(b_i,0),    (a_{i+1},0)]$:
  \[ \eta_-  := \delta(a_1) \oplus [(b_1,0),(a_2,0)]  \oplus \delta(a_2)
    \oplus \cdots  \oplus \delta(a_{i_0}).\] If $\Delta_-  = \emptyset$,
  let  $\eta_-   :=  \{(-1,0)\}$.  Similarly,  collect   the  excursions
  intersecting with $\Delta_+$ starting  from the rightmost one, leading
  to               a               non-increasing               sequence
  $c_1 \geq d_1 > c_2 \geq d_2  > \cdots > c_{j_0} \geq d_{j_0} >0$, and
  concatenate                them                 to                form
  \[ \eta_+  := \delta(c_1) \oplus [(d_1,0),(c_1,0)]  \oplus \delta(c_2)
    \oplus   \cdots    \oplus   \delta(c_{j_0}),   \]    again   setting
  $\eta_+   :=  \{(1,0)\}$   if  $\Delta_+   =  \emptyset$.   Denote  by
  $A_-  = (a_-,0)$  and  $B_- =  (b_-,0)$, resp.\  $A_+  = (a_+,0)$  and
  $B_+  =  (b_+,0)$, the  first  and  last  points of  $\eta_-$,  resp.\
  $\eta_+$. Last, define
  \begin{align*}
    \eta := & ((0,2),(-1,2),\ldots,(a_-,2),(a_-,1),(a_-,0)) \oplus \eta_-
    \\
            & \oplus  ((b_-,0),(b_-,1),(b_-+1,1),\ldots, (a_+,1),(a_+,0)) \oplus
    \eta_+                                                                       \\
            & \oplus ((b_+,0),(b_+-1,0),\ldots,(a,0)) \oplus \bar \omega.
  \end{align*}
  It  is easy  now  to  see that  $\eta$  satisfies our
  requirements, thus concluding the proof of the lemma.
\end{proof}

\begin{remark}
  Recall that if $\mathcal T$ is a tree, we denote by $\widetilde{\mathcal T}$
  the subtree obtained from $\mathcal T$ by recursively erasing all its
  leaves; in terms of our dynamical self-avoiding walk model, this
  corresponds to preventing the path from entering traps. The reader can
  easily check that the previous arguments still apply in the cases of
  $\widetilde{\mathcal T}_{\mathbb H}$ and
  $\widetilde{\mathcal T}_{\mathbb Z^2}$.
  Note that since the limit walk is the same on these trees without leaves
  as in the original ones, it  is  sufficient  to  prove
  Theorem~\ref{thm:propertieoflimitwalk}  in  the  case
  of $\widetilde{\mathcal T}_{\mathbb H}$   and
  $\widetilde{\mathcal T}_{\mathbb Z^2}$.
\end{remark}

\subsubsection{\bf The third step}
\label{thethirdstep}

In      this       step,      we       give      a       proof      of
Theorem~\ref{thm:propertieoflimitwalk}        in       the        case
of $\mathbb{P}_{\lambda}^{\mathbb{Z}^2}$.  We  start with  a
definition:

\begin{definition}
  Let $A$ be a finite subset of $\mathbb Z^2$, its \emph{boundary} is the set
  \[\partial A := \{ x \notin A : \exists z \in A, d (x,z) \leq \sqrt 2 \} \] (where here $d$ denotes the euclidean distance on $\mathbb Z^2$) and its \emph{outer boundary} is the set $\partial_e A$ of all vertices in $\partial A$ from which there exists an infinite self-avoiding path that does not intersect $A$.
\end{definition}

A crucial fact, which we will use in several instances below, is that if $A$ is
connected (seen as a sub-graph of $\mathbb Z^2$), then $\partial_e A$ is
connected as well. This is intuitively clear: informally, one can simply walk
around $A$ while remaining in $\partial_e A$ (see figure~\ref{fig:SAW0111}). A
formal proof is easy but tedious to write, and is therefore omitted here.

 \begin{figure}[ht!]
  \centering
  \includegraphics[scale=0.6]{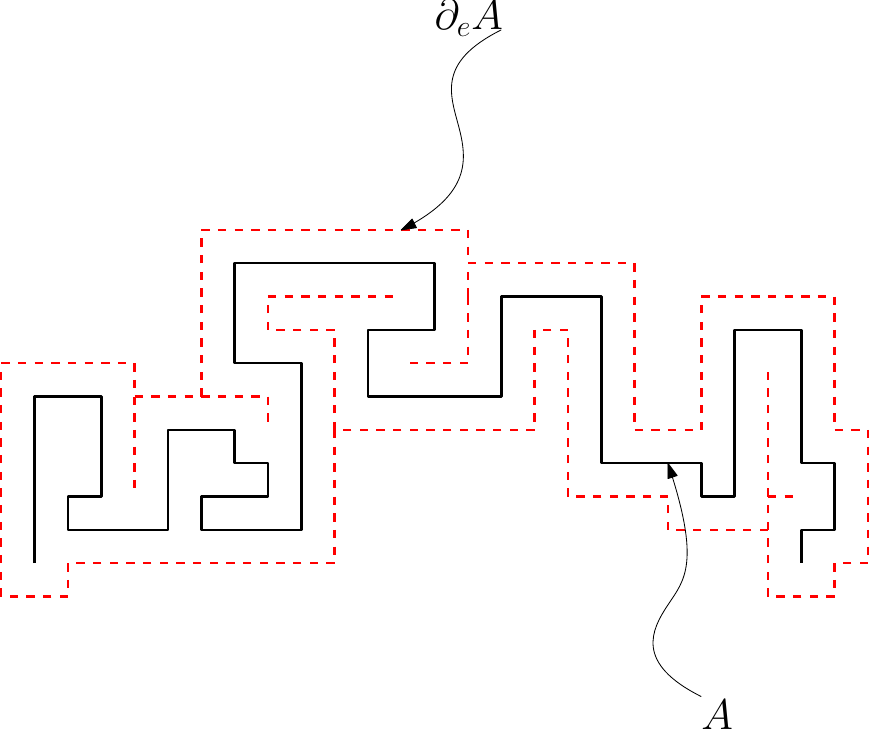}
  \caption{The outer boundary of a sub-graph $A$ of $\mathbb{Z}^2$}
  \label{fig:SAW0111}
 \end{figure}

\begin{proof}[\bf  Proof of  Theorem~\ref{thm:propertieoflimitwalk} in
    the case of $\mathbb{P}_{\lambda}^{\mathbb{Z}^2}$]
  Denote by $\mathcal A$ the following event:
  $$\mathcal A:=\left \{  \#\left \{ n>0: \Im \omega_{\lambda}^{\infty}(n)
    =0 \right \}=\infty \right \}.$$
  We want to show that $P[\mathcal A] = 1$, and to do that we are going to
  apply a strategy which is similar to the one used in step 2. We will need
  some additional notation: for every finite self-avoiding path $\omega$, let $\mathcal
    B_\omega$ denote the event
  $$\mathcal B_\omega :=\begin{cases}
      \omega_{\lambda}^{\infty}(0)                                     =
      \omega(0),\omega_{\lambda}^{\infty}(1)
      =\omega(1),\ldots,\omega_{\lambda}^{\infty}(|\omega|)=\omega(|\omega|);
      \\
      \forall n>|\omega|: \Im \omega_{\lambda}^{\infty}(n)< 0.
    \end{cases}$$
  It is enough to show that $P[\mathcal B_\omega] = 0$ for every $\omega$, so
  we fix a finite self-avoiding path $\omega$ for the rest of the proof. Let
  $n_0 := |\omega|$ be its length;
  without loss of generality we can always assume that $\Im \omega(n_0)<0$ and
  $P[\mathcal B_\omega] > 0$.
  %
  Define
  $$D_\omega := \Big\{ (x,y) \in \mathbb Z^2 : y \geq 0 \text{ and } x \notin
    \{ \Re \omega(i) : 0 \leq i \leq n_0 \} \Big\}$$ and
  let $V_\omega$ be the set of all points in $(\mathbb{Z} \times \{0\})
    \setminus D_\omega$ from which there exists an infinite self-avoiding path
  in the lower half-plane which does not intersect $\omega$ (see
  Figure~\ref{fig:thirdstep2}).
  \begin{figure}[ht!]
    \centering
    \includegraphics[scale=0.5]{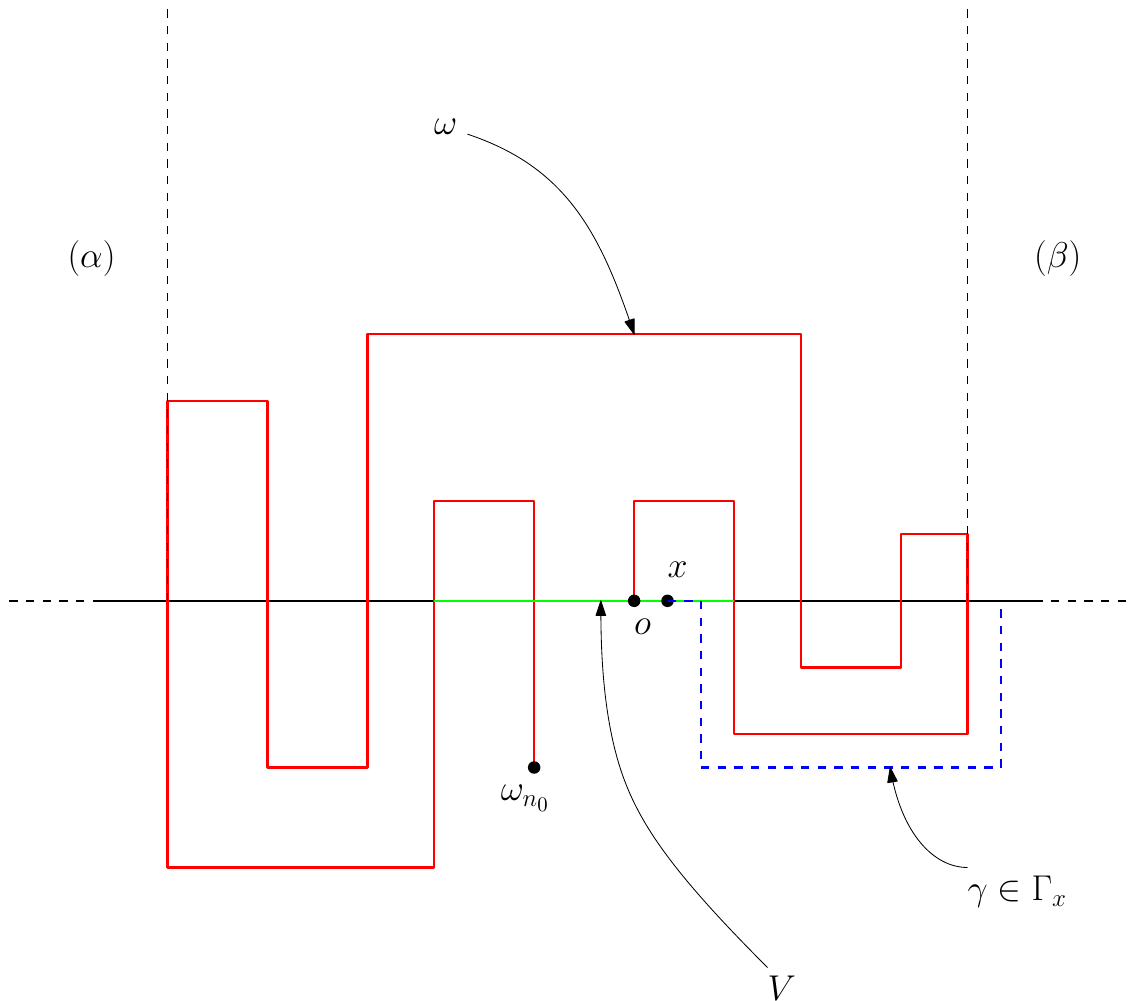}
    \caption{The self-avoiding walk $\omega$ is colored by red; the domain $D$ is the union of two quadrants $\alpha$ and $\beta$ and the set $V$ is colored by green.}
    \label{fig:thirdstep2}
  \end{figure}

  For each $x\in V_\omega$, denote by $\Gamma_x$ the collection of all
  (finite) self-avoiding paths from $x$ to a point in $D_\omega$ that are not
  intersecting $\omega$ and are contained in $\partial
    \omega \cup (\mathbb Z \times \{0\})$. The set $\Gamma_x$ is finite and it is not
  empty by the previously noted connectivity fact about set
  boundaries. We then set $p_\omega:=\max_{x \in V_\omega}\max_{\gamma\in
      \Gamma_x}\left | \gamma \right |$.

  We first make the following remark: let $\tau$ be a stopping time such
  that on the event $\{\tau < \infty\}$, $\omega$ is a prefix of $X_\tau$ and
  $X_\tau \cap (\mathbb Z \times \{0\}) = \omega \cap (\mathbb Z \times \{0\})$.
  On this event, defining \[\begin{cases}
      \sigma_1 := \inf \{ t>\tau : p(X_t) \in \mathbb Z \times \{0\} \text{~and} \\
      \sigma_2 := \inf \{ t>\tau : \omega \text{~not~prefix~of~} X_t \},
    \end{cases}\] it is always the case that $\min(\sigma_1,\sigma_2) < \infty$.
  Indeed, if neither of these occur, then $(X_n)$
  in particular does not touch the axis infinitely many times, thus
  contradicting Lemma~\ref{theorem12} under our assumption that $P[\mathcal
        B_\omega]>0$. On the other hand, on the event $\mathcal B_\omega$, the case
  where $\min(\sigma_1,\sigma_2) = \sigma_2<\infty$ can only occur finitely many times, by definition of
  the limit walk. Thus we only need to rule out the scenario where the case
  $\min(\sigma_1,\sigma_2) = \sigma_1 < \infty$
  occurs infinitely many times.

  To do this, consider a realization of the process where $\sigma_1 <
    \sigma_2$. Let $x := p(X_{\sigma_1})$. We claim that at least one
  element $\gamma$ of $\Gamma_x$ intersects $X_{\sigma_1}$ only at $x$: indeed,
  otherwise the path $X_{\sigma_1}$ would have to form a loop around $x$, closed
  strictly before time $\sigma_1$, and this is impossible for the process
  constructed on the leafless tree that we are considering. Thus, with
  conditional probability at least ${(\lambda/(1+3\lambda))}^{p_\omega}$, the
  trajectory of the process $(X_n)$ right after time $\sigma_1$ will follow the
  path $\gamma$ until reaching $D_\omega$ at some point $y$. When this occurs, the process now
  sees an empty quadrant, and by the positivity of effective conductance
  $C(\lambda, \mathcal T_{\mathbb Q})$ of the
  tree constructed on such a quadrant, with uniformly positive probability it
  will never backtrack further than $y$, in which case the event $\mathcal B_\omega$
  will not be realized. This concludes the proof.
\end{proof}

\subsubsection{\bf The last step}

In      this      section,      we      give      a      proof      of
Theorem~\ref{thm:propertieoflimitwalk}        in       the        case
$\mathbb{P}_{\lambda}^{\mathbb{H}}$. Before beginning the proof, we will need
some combinatorial information about the connective constants of strips.

\begin{definition}
  A \emph{strip of width $\ell$}, denoted by $B_\ell$, is a sub-domain of $\mathbb{Z}^2$ which is
  limited by two lines $\{\Im z=a\}$ and $\{\Im z=b\}$ (for a horizontal strip) or $\{\Re z=a\}$ and
  $\{\Re z=b\}$ (for a vertical strip), such that $|a-b|=\ell$.

  Fix an origin $O\in \{\Im z=a\}\cup
    \{\Im z=b\}$ (or $\{\Re z=a\}\cup \{\Re z=b\}$) of $B_\ell$ and let $\gamma$
  be a finite self-avoiding path starting at $O$. We say that $\gamma$ is
  a \emph{self-avoiding path in the strip $B_\ell$} if for any $0\leq k\leq
    |\gamma|$, we have $\gamma(k)\in B_\ell$; we define the self-avoiding tree
  $\mathcal{T}_{B_\ell}$ from the self-avoiding paths from $O$ in $B_\ell$ as in
  Notation~\ref{not:proofofprop1.1}.

  A \emph{bridge} (resp. \emph{irreducible bridge}) in
  $B_\ell$ is defined in the same way as in the half-plane; see
  Figure~\ref{fig:bridgeofstrip}.
\end{definition}

\begin{figure}[ht!]
  \centering
  \includegraphics[scale=0.6]{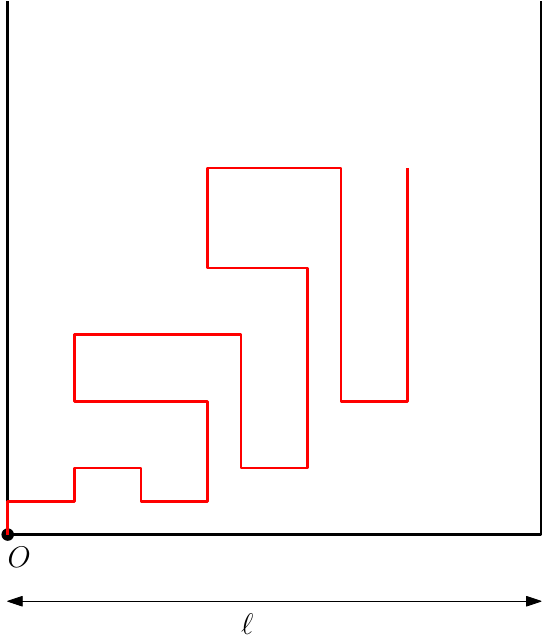}
  \caption{A bridge in a strip $B_\ell$}
  \label{fig:bridgeofstrip}
\end{figure}

\begin{lemma}[The subadditivity property]
  \label{subadditivityproperty}
  For  $\ell, n$ positive natural  numbers, denote  by
  $p^{(\ell)}_n$ the  number of  bridges of
  length $n$ starting  at the origin in the strip  $B_\ell$. Then, for any $\ell, n,
    m, k \in \mathbb{N}^*$,
  $$ p^{(2\ell)}_{n+m}\geq
    p^{(\ell)}_{m}\,p^{(\ell)}_{n} \, \text{ and }\,  p^{(2\ell)}_{kn} \geq {(p^{(\ell)}_{n})}^k.$$
\end{lemma}


\begin{proof}
  Divide the (vertical) strip  $B_{2\ell}$  into two  small  strips
  $B^1_{\ell},B^2_{\ell}$  of width  $\ell$ (see Figure~\ref{fig:SAW270}).
  Consider $\gamma_1,\gamma_2$ two bridges in the  strip $B^1_{\ell}$, of length $m$
  and  $n$ respectively,  and  concatenate   $\gamma_1$ and $\gamma_2$ to obtain
  a  new bridge $\gamma_{12}:=\gamma_1 \oplus \gamma_2$ of length $m+n$ in the
  strip $B_{2\ell}$ (see Figure~\ref{fig:SAW270} again). This is an injection,
  and hence for any $\ell,n,m \in
    \mathbb{N}^*$,
  $$p^{(2\ell)}_{n+m}\geq p^{(\ell)}_{m}\,p^{(\ell)}_{n}$$
  which is the first claim of the Lemma.

  \begin{figure}[ht!]
    \centering
    \includegraphics[scale=0.6]{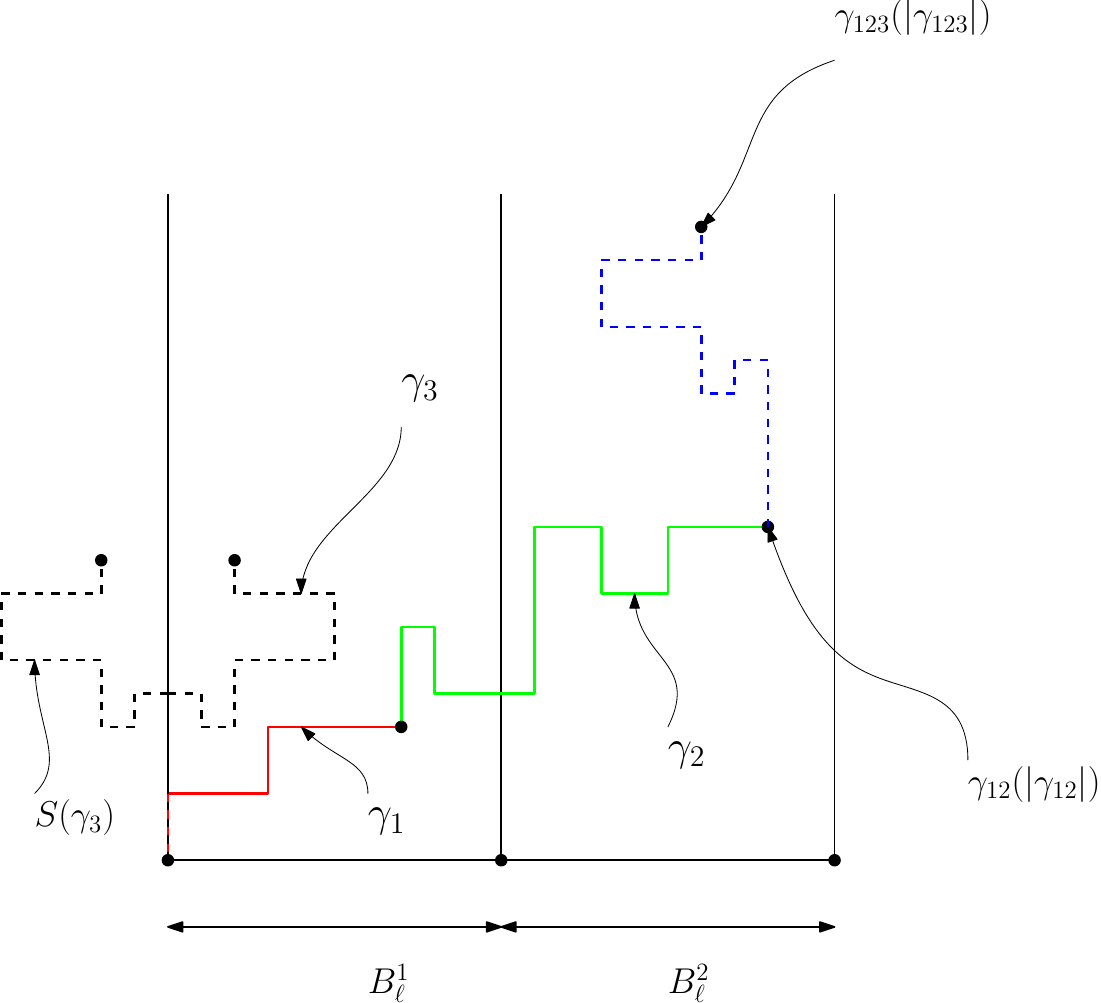}
    \caption{A concatenation of 3 bridges in $B^1_{L}$.}
    \label{fig:SAW270}
  \end{figure}

  Given a  third bridge  $\gamma_3$ in  $B^1_{\ell}$ of
  length  $q$, we build a bridge $\gamma_{123}$ of length $m+n+q$ in $B_{2\ell}$ as follows (see Figure~\ref{fig:SAW270}):
  $$\begin{cases}
      \gamma_{123}=\gamma_{12}\oplus \gamma_3    & \text{~if~} \gamma_{12}(\left | \gamma_{12} \right |)\in B^1_{\ell}   \\
      \gamma_{123}=\gamma_{12}\oplus S(\gamma_3) & \text{~if~} \gamma_{12}(\left | \gamma_{12} \right |) \in B^2_{\ell},
    \end{cases}$$
  where $S$ denotes the reflection across the vertical line going through $(0,0)$.
  This is again an injection, so for any $\ell,n,m,q \in \mathbb{N}^*$,
  $$p^{(2\ell)}_{n+m+q}\geq p^{(\ell)}_{m}\,p^{(\ell)}_{n}\,p^{(\ell)}_q.$$
  Iterating the same construction leads to the second claim of the Lemma, and
  thus finishes the proof.
\end{proof}

\begin{lemma}
  \label{p1}
  Denote by $\mu(\ell)$ the connective  constant of the strip $B_\ell$, which
  exists from the previous Lemma. Then,
  $$\underset{\ell\rightarrow\infty}{\lim}\mu(\ell)=\mu,$$  where  $\mu$  is  the  connective
  constant of $\mathbb{Z}^2$.
\end{lemma}

\begin{proof}
  Denote by $b^{\mathbb{Q}}_n$ the number of bridges of length $n$ in
  $\mathbb{Q}$, starting at the origin. Note that for any $\ell$, we have:
  \begin{equation}
    \label{equ:recrecrec1}
    \underset{n\rightarrow\infty}{\lim}{(p^{(\ell)}_n)}^{\frac{1}{n}}=\mu(\ell) \quad \text{and} \quad p^{(\ell)}_\ell=b^{\mathbb{Q}}_\ell.
  \end{equation}
  Moreover, we also have:
  \begin{equation}
    \label{equ:recrecrec2}
    \underset{n\rightarrow\infty}{\lim}{\left(b^{\mathbb{Q}}_n\right)}^{\frac{1}{n}} =\mu.
  \end{equation}
  By using Lemma~\ref{subadditivityproperty}, for any  $\ell,n,k$:
  \begin{equation}
    \label{equ:recrecrec3}
    p^{(2\ell)}_{kn}\geq {(p^{(\ell)}_n)}^k.
  \end{equation}
  Fix $\epsilon>0$ and by~\eqref{equ:recrecrec2}, there exists $n_0$ such that for any $n>n_0$, we have
  \begin{equation}
    \label{equ:recrecrec4}
    \left | {\left(b^{\mathbb{Q}}_n\right )}^{\frac{1}{n}}-\mu \right |\leq \epsilon.
  \end{equation}
  Let     $\ell>n_0$    and     $k>0$.    By~\eqref{equ:recrecrec1},~\eqref{equ:recrecrec3} and~\eqref{equ:recrecrec4}, we have:
  \begin{equation}
    \label{equ:recrecrec5}
    {\left(p^{(2\ell)}_{k\ell}\right)}^{\frac{1}{k\ell}}\geq
    {\left(p^{(\ell)}_{\ell}\right)}^{\frac{1}{\ell}}={\left(b^{\mathbb{Q}}_{\ell}\right)}^{\frac{1}{\ell}}\geq \mu-\epsilon.
  \end{equation}
  Since the  sequence ${(p^{(2\ell)}_{k\ell})}^{\frac{1}{k\ell}}$  converges towards
  $\mu({2\ell})$ when $k$ goes to infinity, we use~\eqref{equ:recrecrec5} to obtain:

  \begin{equation}
    \label{equ:recrecrec6}
    \mu\geq \mu_{2\ell} \geq \mu -\epsilon,
  \end{equation}
  where inequality $\mu\geq \mu_{2\ell}$ is obvious. Hence, the sequence $(\mu(\ell), \ell\geq 1)$ converges towards $\mu$ when $\ell$ goes to $+\infty$.
\end{proof}

\begin{proposition}
  \label{p2}
  Denote  by   $br(\mathcal{T}_{B_\ell})$  the  branching   number  of
  $\mathcal{T}_{B_\ell}$. Then,
  $$\underset{\ell\rightarrow\infty}{\lim}br(\mathcal{T}_{B_\ell})=\mu,$$
  where $\mu$ is again the connective constant of $\mathbb{Z}^2$.
\end{proposition}

\begin{proof}
  The following argument is very close in spirit to the proof of Proposition~\ref{prop10}. Recall that  an infinite self-avoiding path starting at the origin in $B_\ell$ is \emph{$m$-good} if it possesses a  decomposition into
  irreducible bridges of length at most $m$.  Denote by $G_m(B_\ell)$ the set
  of infinite, $m$-good
  self-avoiding paths in $B_\ell$. Let $\smash{\mathcal
      T^{(m)}_{B_\ell}}$ be the subtree of $\mathcal{T}_{B_\ell}$,  which  we will
  refer to as the \emph{$m$-good tree}, spanning
  $$V(\mathcal{T}^{(m)}_{B_\ell}):=\{\omega\in V(\mathcal{T}_{B_\ell}):
    \text{ there exists } \gamma\in G_m(B_\ell) \text{ such that }
    \gamma_{[0, |\omega|]}=\omega\}$$
  (\emph{i.e.}, the tree formed of all finite self-avoiding paths in $B_\ell$
  that can be extended into an infinite, $m$-good path).

  Denote by $p^{(\ell,m)}_n$ be  the number  of bridges of length $n$ in
  $B_\ell$, starting at the origin, which
  possess a decomposition  into bridges of length at most $m$. Recall that $p^{(\ell)}_n$ is the  number of  bridges of
  length $n$ in $B_\ell$ starting  at origin, and that
  ${(\mathcal{T}^{(m)}_{B_\ell})}_n$ denotes the number of vertices of
  $\mathcal{T}^{(m)}_{B_\ell}$ at generation $n$. Then for any $n>0$, we have
  \begin{equation}
    \label{equ:nicetree1}
    \left | {\left(\mathcal{T}^{(m)}_{B_\ell}\right)}_n \right |\geq p^{(m)}_n.
  \end{equation}
  By using Lemma~\ref{subadditivityproperty}, for any $\ell, m, n, k$ we have:
  \begin{equation}
    \label{equ:nicetree2}
    p^{(2\ell)}_{nk}\geq {(p^{(\ell)}_n)}^k \quad \text{and} \quad p^{(2\ell,m)}_{nk}\geq {(p^{(\ell,m)}_{n})}^k.
  \end{equation}
  We know (see Notation~\ref{subsubsection:notation}) that all  bridges  in  a
  half-plane  can be decomposed into a sequence of irreducible bridges in a
  unique way; therefore each bridge  in $B_{\ell}$ of length $m$ possesses  a
  unique decomposition  into  irreducible  bridges of length at most $m$. Hence, for any $m,
    \ell>0$,
  \begin{equation}
    \label{equ:nicetree0}
    p_{m}^{(\ell)}=p_{m}^{(\ell, m)}.
  \end{equation}
  Fix $\epsilon>0$, by Lemma~\ref{p1},  there exists $\ell_0$ such that
  for any $\ell>\ell_0$,
  \begin{equation}
    \label{equ:nicetree3}
    \mu\geq \mu(2\ell)>\mu-\epsilon.
  \end{equation}
  Moreover, since $\mu(2\ell)=\underset{n\rightarrow \infty}{\lim}{(p^{(2\ell)}_n)}^\frac{1}{n}$, then there exists  $n_0$ such that for any $n>n_0$:
  \begin{equation}
    \label{equ:nicetree4}
    {(p^{(2\ell)}_n)}^\frac{1}{n}>\mu(2\ell)-\epsilon.
  \end{equation}
  Hence by~\eqref{equ:nicetree0},~\eqref{equ:nicetree2},~\eqref{equ:nicetree3} and~\eqref{equ:nicetree4},

  \begin{equation}
    \label{equ:nicetree5}
    {(p^{(4\ell,n)}_{kn})}^\frac{1}{kn}\geq
    {(p^{(2\ell,n)}_n)}^{\frac{1}{n}}={(p^{(2\ell)}_n)}^{\frac{1}{n}}\geq
    \mu(2\ell)-\epsilon \geq \mu-2\epsilon.
  \end{equation}
  Therefore  for $\ell>\ell_0$ and $m>m_0 (\ell)$, we have
  \begin{equation}
    \label{equ:nicetree6}
    \overline{gr(\mathcal T_{B_{4\ell}}^{(m)})}\geq  \mu-2\epsilon.
  \end{equation}

  On the other hand, noting that by construction,
  $\mathcal T_{B_{4\ell}}^{(m)}$        is         $(m+4\ell)$-super-periodic        and
  has finite growth,      we     can use
  Theorem~\ref{sousperiodic}  to get:
  \begin{equation}
    \label{equ:nicetree7}
    gr(\mathcal T_{B_{4\ell}}^{(m)}) \quad\text{exists  and}\quad  gr(\mathcal T_{B_{4\ell}}^{(m)})=br(\mathcal T_{B_{4\ell}}^{(m)}).
  \end{equation}
  Since $\mathcal T_{B_{4\ell}}^{(m)}\subset  \mathcal T_{B_{4\ell}}$, by using~\eqref{equ:nicetree6},~\eqref{equ:nicetree7} and Proposition~\ref{prop:br-gr} we obtain for any $\ell>\ell_0$:
  \begin{equation}
    \label{equ:nicetree8}
    \mu\geq  br(\mathcal T_{B_{4\ell}})\geq \mu-2\epsilon,
  \end{equation}
  where we used $\mathcal T_{B_{4\ell}}\subset \mathcal{T}_{\mathbb{H}}$ for the first inequality. Therefore,
  the sequence ${(br(\mathcal T_{B_\ell}))}_{\ell\geq 1}$ converges towards $\mu$ when $\ell$ goes to infinity.
\end{proof}




We can now turn to the study of the limit self-avoiding path in the
half-plane.

\begin{proposition}
  \label{proposition4.2}
  Consider    the   biased   random   walk    $RW_{\lambda}$   on
  $\widetilde{\mathcal T}_{\mathbb H}$.  Let  ${(B_\ell)}_{\ell\geq 1}$ be  the
  sequence  of (horizontal) strips in
  $\mathbb{H}$ where $B_\ell$ is the strip between the lines $\{\Im z =0\}$ and
  $\{\Im z =\ell\}$. Suppose that $\lambda > \frac{1}{\mu}$, where $\mu$ is the
  connective constant of $\mathbb{H}$.  Then, whenever $\ell>0$ is large enough
  that $\mu(B_\ell) > 1/\lambda$ (which holds for $\ell$ large enough by the
  previous Lemma),
  the limit walk
  $\omega_{\lambda}^{\infty}$ in $\mathbb H$ almost surely touches the
  strip $B_\ell$ infinitely often.
\end{proposition}

\begin{figure}[ht!]
  \centering
  \includegraphics[scale=0.5]{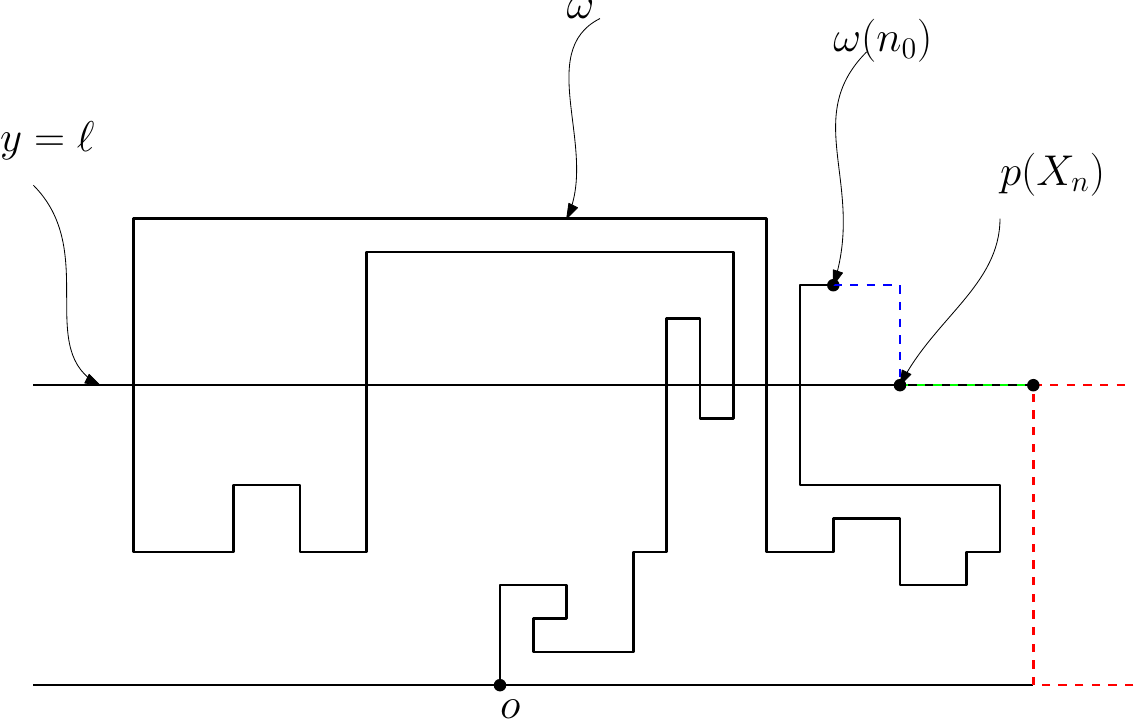}
  \caption{Illustration of the proof of Proposition~\ref{proposition4.2}}
  \label{fig:SAW111111111}
 \end{figure}   

\begin{proof}
  Fix $\ell$ such that $\mu(B_\ell) > 1/\lambda$. We  use again  the
  same argument as in the previous steps: assuming that $\omega_\lambda$ touches
  $B_\ell$ finitely many times with positive probability, there exists $n_0>0$
  and a self-avoiding path
  $\omega=\left[\omega(0),\ldots,\omega(n_0)\right]$
  such that the  following event has  strictly positive probability:
  $$B:=\begin{cases}
      \omega_{\lambda}^{\infty}(0)              =             \omega(0),
      \omega_{\lambda}^{\infty}(1)       =      \omega(1),       \ldots,
      \omega_{\lambda}^{\infty}(n_0) = \omega(n_0)
      \\
      \forall n>n_0: \Im \omega_{\lambda}^{\infty}(n)> \ell.
    \end{cases}$$
  By Lemma~\ref{theorem12},  we know that the head $p(X_n)$ of the dynamical
  path almost surely reaches
  the line $\mathbb{Z}\times \{0\}$ an infinite number of  times, therefore on
  the event $B$ it must be the case that $\Im p(X_n) = \ell$ infinitely many
  times after time $n_0$. Let $n$ be the first time when this occurs: at time
  $n$, the process sees a strip of width $\ell$ that is empty except for the
  trace of $\omega$; the corresponding subtree has positive conductance (with
  a lower bound depending only on $\lambda$, $\ell$ and $\omega$), and
  therefore $X$ has positive probability of never backtracking through $X_n$,
  in which case the second condition defining the event $B$ does not hold.
  Iterating the same argument at each successive visit of $B_\ell$ separated
  by backtracking, this leads to a contradiction. See
    Figure~\ref{fig:SAW111111111} for illustration.
\end{proof}

\begin{proof}[\bf  Proof of  Theorem~\ref{thm:propertieoflimitwalk} in
    the case of $\mathbb{P}_\lambda^{\mathbb{H}}$]

  By Proposition~\ref{proposition4.2}, we can fix a number $\ell>0$ such that
  the limit walk  almost  surely reaches the strip $B_\ell$ infinitely often.
  We just need to prove  that it therefore visits the  line  $\mathbb{Z}\times
    \{0\}$.

  The argument is again similar to the previous ones, so we do not flesh it out
  in detail. We know that $\omega_\lambda$ visits $B_\ell$ infinitely many
  times, therefore it occurs infinitely many times that \[p(X_n) \in B_\ell
    \text{ and } |p(X_n)| > \max \{ |p(X_m)| : m < n \text{ and } p(X_m) \in
    B_\ell \}.\] Whenever this holds,
  the process has positive probability to reach the line $\mathbb Z \times
    \{0\}$ by going down vertically, and once it reaches it, seeing an empty
  half-strip of width $\ell$, to never backtrack through that point.
  %
  %
  %
\end{proof}

\begin{remark}
  We applied a similar reasoning several times above; informally, the general
  principle is that if there is a region $V \subset \mathbb Z^2$ that the head
  $p(X_n)$ visits infinitely often with probability $1$, and which is such that
  the effective conductance of $\mathcal T_{V \setminus K}$ can be bounded below
  for every finite $K$ uniformly in the choice of the root, then the limit walk
  has infinitely many points in $V$. Writing a formal statement with this flavor
  would involve technicalities that are not needed in the setups that we were in
  (where $V$ is either a half-plane or a strip).
\end{remark}

\subsection{The law of first steps of the limit walk}

We consider the biased random walk $RW_{\lambda}$ on $\mathcal T_{\mathbb{Z}^2}$
(resp.\ $\mathcal T_{\mathbb H}$, $\mathcal T_{\mathbb Q}$ --- the argument
below is exactly the same all three cases). Recall that
$\omega^{\infty}_{\lambda}$   is  the   associated   limit  walk   and
$\mathbb{P}_{\lambda}^{\mathbb{H}}$ denotes its law.

Let $k \in  \mathbb{N}^{*}$ and $y_1,y_2,\ldots,y_k$ be  $k$ elements of
$V(\mathcal T_{\mathbb{H}})$ such that the path $(o,y_1,y_2,\ldots,y_k)$
in $\mathcal T_{\mathbb{H}}$ is simple.
For each $\lambda> \lambda_c$, recall that the law of first $k$ steps of
$\omega_\lambda$ is defined by:
\begin{equation}\label{equ:recs1}
  \varphi^{\lambda,k}(y_1,y_2,\ldots,y_k)                            =
  \mathbb{P}_{\lambda}^{\mathbb{H}}(\omega^{\infty}_{\lambda}(1)        =       y_{1},
  \omega^{\infty}_{\lambda}(2)         =          y_{2},         \ldots,
  \omega^{\infty}_{\lambda}(k)=y_{k}).
\end{equation}
We prove the continuity of this function.

\begin{theorem}
  \label{ctggg}
  For every $k \in \mathbb{N}^{*}$ and $(y_1,y_2,\ldots,y_k) \in V^k$,
  the  function  $\varphi^{\lambda,k}$  depends continuously on $\lambda$ on the
  whole interval $(\lambda_c,+\infty)$.
\end{theorem}

Let $\mathcal T$ be an infinite, locally finite and rooted tree and $\nu$ is a
child of the root. Recall the definition of
$\widetilde{\mathcal C}(\lambda,\mathcal T)$ and $\widetilde{\mathcal
    C}(\lambda, \mathcal{T}, \nu)$ in Section~\ref{randomwalkontrees}. To prove
Theorem~\ref{ctggg}, we need the following lemma:

\begin{lemma}
  \label{lemmactggg}
  We have
  $$\varphi^{\lambda,k}(y_1,y_2,\ldots,y_k)=
    \frac{\widetilde{\mathcal C}(\lambda, \mathcal{T}, y_1)}{\widetilde{\mathcal C}(\lambda, \mathcal{T})}                \times
    \frac{\widetilde{\mathcal C}(\lambda, \mathcal{T}^{y_1}, y_2)}{\widetilde{\mathcal C}(\lambda, \mathcal{T}^{y_1})}   \times     \cdots
    \times \frac{\widetilde{\mathcal C}(\lambda, \mathcal{T}^{y_{k-1}}, y_k)}{
      \widetilde{\mathcal C}(\lambda, \mathcal{T}^{y_{k-1}})}.$$
\end{lemma}

\begin{proof}
  We prove this lemma in the  case $k=1$, and leave the (slightly more
  complicated, but  following the same  lines) cases $k\geq 2$  to the
  reader. Denote by $\widetilde{\mathcal C_i}(\lambda,\mathcal T)$ the
  probability that the biased random walk on $\mathcal T$ returns to  origin
  exactly $i$
  times before going to  infinity.  Define  the following events:
  \[\begin{cases}
      \mathcal A :=     & \{  \omega^{\infty}_{\lambda}(1)=y_{1} \},                          \\
      \mathcal A_{i} := & \left \{  \omega^{\infty}_{\lambda}(1)=y_{1} \quad\text{and}\quad \#\{ n>0:
      X_n=o \}=i \right \}.
    \end{cases}\]
  The events $\mathcal A_{i}$ are disjoint, and by transience, $\mathcal
    A=\bigcup \mathcal A_{i}$.
  On the other hand,  by the Markov property, for any $i\geq 0$, we have
  \begin{equation*}
    \label{equ:recrecrecrec1}
    \mathbb{P}(\mathcal A_{i})
    =\widetilde{\mathcal C}(\lambda,\mathcal T,y_{1}){\left(1- \widetilde{\mathcal C}(\lambda,\mathcal T)\right)}^i.
  \end{equation*}
  Summing this identity over $i \geq 0$ leads to
  $$\mathbb{P}(\mathcal A)=\sum_{i=0}^{+\infty}\mathbb{P}(\mathcal A_{i})
    =\frac{\widetilde{\mathcal C}(\lambda,\mathcal T,y_{1})}{\widetilde{\mathcal C}(\lambda,\mathcal T)},$$
  which is indeed the claim of the Lemma for the case $k=1$.
\end{proof}

\begin{proof}[Proof of Theorem~\ref{ctggg}.]
  By Lemma~\ref{lemmactggg} applied to the case $\mathcal T = \mathcal
    T_{\mathbb H}$, we have
  $$\varphi^{\lambda,k}(y_1,y_2,\ldots,y_k)=
    \frac{\widetilde{\mathcal C}(\lambda, \mathcal{T}, y_1)}{\widetilde{\mathcal C}(\lambda, \mathcal{T})}                \times
    \frac{\widetilde{\mathcal C}(\lambda, \mathcal{T}^{y_1}, y_2)}{\widetilde{\mathcal C}(\lambda, \mathcal{T}^{y_1})}   \times     \cdots
    \times \frac{\widetilde{\mathcal C}(\lambda, \mathcal{T}^{y_{k-1}}, y_k)}{
      \widetilde{\mathcal C}(\lambda, \mathcal{T}^{y_{k-1}})}.$$
  It is
  enough    to    prove    that the functions   $\widetilde{\mathcal C}(\lambda,\mathcal T^{y_i},y_{i+1})$    and
  $\widetilde{\mathcal C}(\lambda,\mathcal T^{y_i})$   are  continuous. This
  follows from previous arguments, which we are going to adapt.

  The  continuity   of
  $\widetilde{\mathcal C}(\lambda,\mathcal T^{y_i})$ is a consequence of
  Theorem~\ref{ctgg}. All that is needed is to show that the trees $\mathcal
    T^{y_i}$ are all weakly uniformly transient, and this is true because they are
  finite modifications of $\mathcal T_{\mathbb H}$, which is itself weakly
  uniformly transient.

  For        the        continuity        of
  $\widetilde{\mathcal C}(\lambda,\mathcal T^{y_i},y_{i+1})$, this function can
  be written in terms of effective conductances: denoting by
  $\hat {\mathcal T}^{y_i}$ the tree obtained from $\mathcal T^{y_i}$ by adding
  one extra vertex that is a parent of the root, using the Markov property at
  time $1$ we get
  \[ \widetilde {\mathcal C} (\lambda, \mathcal T^{y_i}, y_{i+1}) = \frac
    {1} {d(y_i)} \cdot \mathcal C(\hat {\mathcal T}^{y_i})\] which is therefore
  continuous by another application of Theorem~\ref{ctgg}.
\end{proof}



\section{The critical probability measure via biased random walk}
\label{section6}

\subsection{The critical probability measure}
\begin{figure}[ht!]
  \centering
  \includegraphics[scale=0.6]{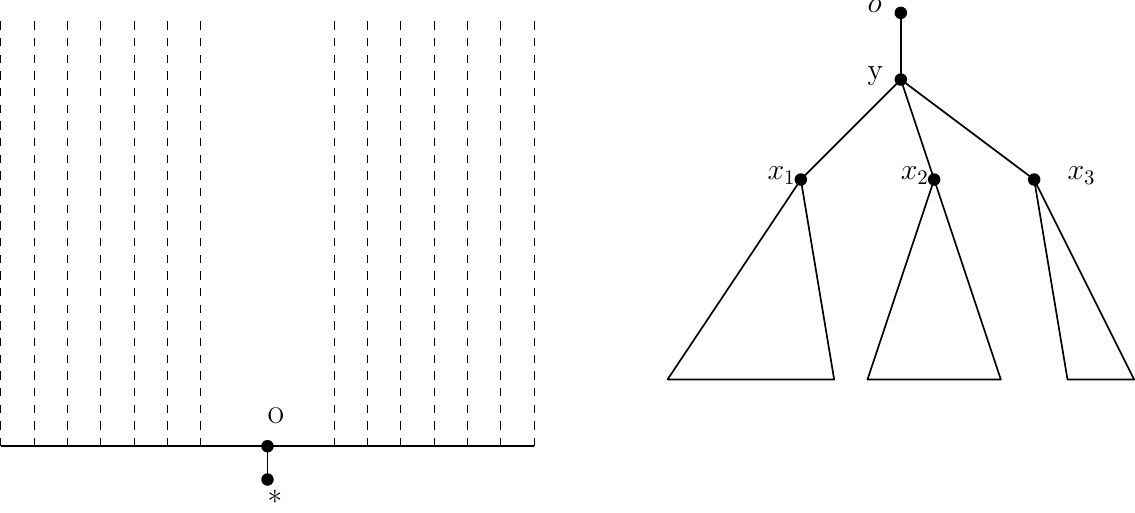}
  \caption{The upper-half plane on the left and the tree $\mathcal{T}_{\bar{\mathbb{H}}}$ on the right.}
  \label{fig:SAW111}
  \label{figure:critical1}
\end{figure}

In this section, we only work on the half-plane $\mathbb H$ and the associated
self-avoiding tree $\mathcal T_{\mathbb H}$. It will be useful below to extend
the half-plane by adding one more vertex below the origin: namely,
\[\bar{\mathbb H} := \mathbb H \cup \{\ast\} \quad \text{with} \quad \ast := (0,-1)\] and to consider the biased walk
on the corresponding tree $\mathcal T_{\bar{\mathbb H}}$; notice that the limit
walk in $\bar {\mathbb H}$ has the same law as the concatenation $[\ast,o] \oplus \gamma$
where $\gamma$ is the limit walk on $\mathbb H$.


We aim to construct a  critical probability measure through the biased
random walk on the self-avoiding tree by letting $\lambda$ decrease to its critical value. First, we  review the construction
of  Madras and  Slade (see~\cite{madras2013self} for  detail). Let $\mathcal
  B_n$ denote the set of all self-avoiding bridges of length $n$ in 
  $\bar{\mathbb{H}}$ starting at $\ast$, and let $\beta_n := |\mathcal B_n|$ be the number of
such bridges. Given  $n\geq m$ and  an $m$-step self-avoiding walk  $\gamma$ in
$\bar{\mathbb{H}}$, let  $\mathbb{P}^{\mathcal B}_{m,n}(\gamma)$  denote the
fraction of $n$-step bridges in $\mathcal B_n$ that extend $\gamma$, \emph{i.e.}
\begin{equation}
  \label{equation6}
  \mathbb{P}^{\mathcal B}_{m,n}(\gamma)=\frac{\left | \mathcal F_n(\gamma)\cap {\mathcal B_n} \right |}{\beta_n}=\frac{\left | \mathcal F_n(\gamma) \right |}{\beta_n},
\end{equation}
where $\mathcal F_n(\gamma)$ is the set of all $n$-step bridges which extend
$\gamma$. One can think of $\mathbb{P}^{\mathcal B}_{m,n}(\gamma)$ as the probability that a
bridge uniformly chosen from among all $n$-step bridges is an extension of
$\gamma$, it defines a probability measure on $\mathcal T_{\bar{\mathbb H}}^m$.
Theorem~8.3.1 in~\cite{madras2013self} states that the limit
\begin{equation}
  \label{equation7}
  \mathbb{P}^{\mathcal B}_m(\gamma):=\lim_{n \rightarrow \infty}\mathbb{P}^{\mathcal B}_{m,n}(\gamma)
\end{equation}
exists.
By the Kolmogorov theorem, this allows for the definition of a measure
$\mathbb{P}^{\mathcal B}_{\infty}$  on   the   set $\SAW_\infty$  of infinite self-avoiding
paths in
$\bar{\mathbb{H}}$ by imposing,  for  every  $\gamma \in  \SAW_{\infty}$,
$m\geq0$ and $\gamma'$ of length $m$,
$$\mathbb{P}^{\mathcal B}_{\infty}(\gamma[0,m]=\gamma')=\mathbb{P}^{\mathcal B}_m(\gamma').$$

\begin{theorem}[\cite{madras2013self}, Theorem 8.3.2]
  The probability measure $\mathbb{P}^{\mathcal B}_{\infty}$ coincides with the
  Kesten measure at parameter $\mu^{-1}$, where $\mu$ is the connective constant
  of the square lattice.
\end{theorem}

Recall  that for  all  $m\geq  1$, $\mathcal T_{\bar{\mathbb H}}^{(m)}$  is  the  $m$-good tree built
from self-avoiding paths having only irreducible bridges of length at most $m$ in their
decomposition (see Notation~\ref{subsubsection:notation}).
Fix $k>0$ and a self-avoiding path $\gamma$ of length $k$ started at $\ast$;
define $\varphi^{\lambda}(\gamma)$ and $\varphi^{\lambda,m}(\gamma)$ as the
probability that $\gamma$ is the prefix of length $m$ of the limit walk with
parameter $\lambda$ respectively on $\mathcal T_{\bar{\mathbb H}}$ and $\mathcal
  T_{\bar{\mathbb H}}^{(m)}$. Note that $\varphi^\lambda$ is defined for all $\lambda >
  \mu^{-1}$ and $\varphi^{\lambda,m}$ for all $\lambda > \lambda_m :=
  \lambda_c(\mathcal T_{\bar{\mathbb H}}^{(m)})$. 

\begin{lemma}
  As $m\to\infty$, we have $\lambda_m \to \lambda_c(\mathcal T_{\bar{\mathbb H}})$.
\end{lemma}

\begin{proof}
  This follows from the same proof as that of Proposition~\ref{p2}: in the last
  paragraph, instead of embedding the tree $\mathcal T^{(m)}_{B_{4\ell}}$ into $\mathcal
  T_{B_{4\ell}}$, one can embed it into $\mathcal T_{\bar{\mathbb H}}^{(m)}$ to derive the lower
  bound $br(\mathcal T_{\bar{\mathbb H}}^{(m)}) \geq \mu - 2\varepsilon$ and conclude in a similar
  fashion.
\end{proof}

\begin{theorem}
  \label{theorem6.1}
  For every $k>0$ and $\gamma$ of length $k$, the following hold:
  \begin{enumerate}
    \item \label{item1}
          $\varphi^{\lambda,m}(\gamma)$ converges as $\lambda$ decreases to $\lambda_m$:
          let $\varphi^{(m)}(\gamma) := \lim_{\lambda \searrow \lambda_m} \varphi^{\lambda,m}(\gamma)$;
    \item \label{item2} $\varphi^{(m)}(\gamma)$ converges as $m\to\infty$: let
          $\varphi^{(\infty)}(\gamma) := \lim_{m\to\infty}\varphi^{(m)}(\gamma)$;
    \item \label{item3}
          Moreover, we have the following diagram:
          \[\begin{tikzcd}
            \varphi^{m,\lambda}(\gamma)
            \ar[r, "m\to\infty", "\lambda>\mu^{-1}" below]
            \ar[d, "\lambda \to \lambda_m" left] &
            \varphi^{\lambda}(\gamma) \\
            \varphi^{(m)}(\gamma)
            \ar[r, "m\to\infty"] & 
            \varphi^{(\infty)}(\gamma)
          \end{tikzcd}\]
  \end{enumerate}
\end{theorem}

\begin{proof}
  As was the case for Theorem~\ref{ctggg}, all the ideas of the argument are
  already present in the case $k=2$, so we focus on that case below and omit the
  details of the extension to general $k\geq3$.


  \textbf{Item (\ref{item1}):}
  The tree $\mathcal T^m_{\bar{\mathbb H}}$ is periodic, so we are in the
  framework of Proposition~\ref{proposition6.1}: as $\lambda \searrow
    \lambda_m$, convergence will follow from that proposition and
  formula~\eqref{equation4} will give the limit. However, a little care has to
  be given because of the fact that the second step of $\gamma$ is not
  necessarily within the first irreducible bridge in the decomposition.
  Let $x_1 := (-1,0)$, $x_2 := (0,1)$ and $x_3 := (1,0)$ be the three neighbors
  of $o$ in $\mathbb H$. Moreover, let $S_i^m$ denote the set of all irreducible
  bridges of length at most $m$ having $[\ast,o,x_1]$ as a prefix.

  Formula~\eqref{equation4} applies directly to the cases of $x_1$ and $x_3$: as
  $\lambda \searrow \lambda_m$,
  \[ \varphi^{m,\lambda} ([\ast,o,x_1]) = \varphi^{m,\lambda} ([\ast,o,x_1]) \to
    \sum_{\gamma\in S_1^m} \lambda_m^{|\gamma|}. \] The case of $x_2$ has two
  sub-cases, since either the first step $[\ast,o]$ is an irreducible bridge in
  the decomposition by itself (in which case passing through $x_2$ afterwards is
  automatic), or it is the first step of a bridge going through $x_2$:
  therefore, applying~\eqref{equation4} to both sub-cases, we obtain that as
  $\lambda \searrow \lambda_m$,
  \[ \varphi^{m,\lambda} ([\ast,o,x_2]) \to \lambda_m + \sum_{\gamma\in S_2^m}
    \lambda_m^{|\gamma|}. \] Notice that relation~\eqref{eq:critvalueperiodic}
  implies that the sum of the three limits is indeed equal to $1$, since the set
  of relevant irreducible bridges is ${[\ast,o]} \cup S^m_1 \cup S^m_2 \cup
  S^m_3$. This gives both the convergence and the value of $\varphi^{(m)}$. To
  summarize, letting $p_{i,n}$ be the number of irreducible bridges of length
  $n\geq2$ which pass through $x_i$, we can rewrite the limit as
  \begin{equation}\label{eq:phim}
    \varphi^{(m)}([\ast,o,x_i]) =
    \delta_i^2 \lambda_m + \sum_{n=2}^{m}p_{i,n}\lambda_m^{n}.      
  \end{equation} 


  \bigskip\textbf{Item (\ref{item2}):}
  For all $m$ we have $\lambda_m \geq \lambda_c$ because $\mathcal T^m \subset
  \mathcal T_{\bar{\mathbb{H}}}$, hence
  \begin{equation}
    \label{equ:res4}
    \varphi^{(m)}([\ast,o,x_i]) \geq \delta_i^2 \lambda_c + \sum_{n=2}^{m}p_{i,n}\lambda_c^{n}.
  \end{equation}
  As $m\to\infty$, the right-hand term above increases to \[
    \tilde{\varphi}([\ast,o,x_i]) := \delta_i^2 \lambda_c +
  \sum_{n=2}^{\infty}p_{i,n}\lambda_c^{n}\] which coincides with the law of the
  first two steps in the Kesten measure discussed earlier; in particular,
  Kesten's identity implies that the sum of the $\tilde\varphi([\ast,o,x_i])$ is
  equal to $1$. We claim that
  $\varphi^{(m)}([\ast,o,x_i]) \to \tilde \varphi([\ast,o,x_i])$ (so that in the
  limit $\varphi^{(\infty)} = \tilde \varphi$); this directly follows from an
  elementary result:
  \begin{lemma}
    Let $(x_m)$, $(y_m)$, $(z_m)$, $(\alpha_m)$, $(\beta_m)$, $(\gamma_m)$ be
    sequences taking values in $[0,1]$ and satisfying the assumptions
    \begin{enumerate}
      \item For all $m$, $x_m \geq \alpha_m$, $y_m \geq \beta_m$ and $z_m \geq
      \gamma_m$;
      \item For all $m$, $x_m + y_m + z_m = 1$;
      \item As $m\to\infty$, $\alpha_m\to\alpha$, $\beta_m\to\beta$ and
      $\gamma_m\to\gamma$ with $\alpha+\beta+\gamma=1$.
    \end{enumerate}
    Then, $x_m\to\alpha$, $y_m\to\beta$ and $z_m\to\gamma$.
  \end{lemma}

  We only need to prove the lemma. To do that, let $(\xi,\eta,\zeta)$ be any
  subsequential limit of $(x_m,y_m,z_m)$: assumption (1) implies that
  $\xi\geq\alpha$, $\eta\geq\beta$ and $\zeta\geq\gamma$; assumption (2) shows
  that $\xi+\eta+\zeta=1$; and these readily imply that $\xi=\alpha$,
  $\eta=\beta$ and $\zeta=\gamma$, thus showing that the only possible
  subsequential limit of $(x_m,y_m,z_m)$ is $(\alpha,\beta,\gamma)$.



  \bigskip\textbf{Item~\eqref{item3}:}
  It            remains            to            prove            that for every
  $\lambda > \lambda_c$ and every finite $\gamma$, $\varphi^{m,\lambda}(\gamma)
  \to \varphi^\lambda(\gamma)$ as $m\to\infty$. We begin with a combinatorial
  remark. Let $\ell>1$; if $\gamma$ is a self-avoiding path of length
  $\ell$ in $\bar{\mathbb H}$ which can be extended to an infinite self-avoiding
  path (\emph{i.e.}, $\gamma$ is a vertex at height $\ell$ in the leaf-erased
  version of $\mathcal T_{\bar{\mathbb H}}$), let $m(\gamma)<\infty$ be the
  smallest $m$ such that $\gamma\in \mathcal T_{\bar{\mathbb H}}^m$, and let
  $M_\ell<\infty$ be the largest such $m(\gamma)$. By definition, the
  leaf-erased trees $\tilde{\mathcal T}_{\bar{\mathbb H}}$ and 
  $\tilde{\mathcal{T}}_{\bar{\mathbb H}}^{M_\ell}$ coincide within distance $\ell$ of their roots.
  
  Now fix $\lambda>\lambda_c$ and $m_0$ such that $\lambda_{m_0}<\lambda$; also
  fix a finite self-avoiding path $\gamma$ and some $\varepsilon>0$. By uniform
  transience, there exists $L<\infty$ such that the probability that the biased
  walk in either $\tilde{\mathcal T}_{\bar{\mathbb H}}$ or 
  $\tilde{\mathcal{T}}_{\bar{\mathbb H}}^{m}$ (for any $m \geq m_0$) has probability at most
  $\varepsilon$ of revisiting level $|\gamma|$ after visiting level $L$. 

  For any $m \geq \max(m_0,M_L)$, the biased walks in 
  $\tilde{\mathcal{T}}_{\bar{\mathbb H}}$ and $\tilde{\mathcal T}_{\bar{\mathbb H}}^{m}$ coincide
  until they reach level $L$ (because the trees coincide below that level),
  after which they revisit level $|\gamma|$ with probability at most
  $\varepsilon$ (because of the definition of $L$), thus \(|
  \varphi^{\lambda,m} (\gamma) - \varphi^\lambda(\gamma)| \leq
  \varepsilon\). This concludes the proof of the Theorem.
\end{proof}



\subsection{Conjectures}

If we take a sequence of cutsets $\pi_n:= \mathcal T_n$ and we set
$c(e)=\mu^{-|e|}$, then
$$\sum_{n}{\left(\sum_{e \in
\pi_n}c(e)\right)}^{-1}=\sum_{n=1}^{+\infty}\frac{\mu^n}{c_n}.$$ If the
prediction of Nienhuis~\cite{nienhuis1982exact} were true, we would obtain
$$\sum_{n=1}^{+\infty}\frac{\mu^n}{c_n}\geq
c\sum_{n=1}^{+\infty}\frac{1}{n^{\frac{11}{32}}}=+\infty$$
and by Theorem~\ref{NW}, this would imply the recurrence of the critical biased
random walk on the self-avoiding tree. We believe that recurrence does hold and
that it might be provable without computing the critical exponent above, so we
leave it as a conjecture:
\begin{conj}
  \label{conj2}
  The biased random walk  $RW_{\lambda_c}$ on $T_{\mathbb{H}}$ (or
  $T_{\mathbb{Z}^2}$) is recurrent.
\end{conj}

Finally,        we        believe         that    one can take the limit
$\lambda \to \lambda_c$ without restricting the lengths of the irreducible
bridges in the decomposition:

\begin{conj}
  The following convergence diagram holds
  \[\begin{tikzcd}
    \varphi^{m,\lambda}(\gamma)
    \ar[r, "m\to\infty", "\lambda>\mu^{-1}" below]
    \ar[d, "\lambda \to \lambda_m" left] &
    \varphi^{\lambda}(\gamma) 
    \ar[d, "\lambda\to\lambda_c"] \\
    \varphi^{(m)}(\gamma)
    \ar[r, "m\to\infty"] & 
    \varphi^{(\infty)}(\gamma)
  \end{tikzcd}\]
\end{conj}

\bibliographystyle{siam}
\bibliography{bib}
\end{document}